\documentclass[12pt,a4paper,reqno]{amsart}

\usepackage[all,2cell,arrow,color]{xy}
\newdir{ >}{{}*!/-10pt/@{>}}
\usepackage[english]{babel}
\usepackage{stmaryrd}
\usepackage{wasysym}
\numberwithin{equation}{section}
\usepackage{graphicx} 
\usepackage{rotating}
\usepackage[pagebackref,colorlinks,urlcolor=blue,citecolor=darkgreen,linkcolor=darkgreen]{hyperref}
\usepackage{xcolor}
\usepackage{MnSymbol}
\usepackage{bbm}
\usepackage[normalem]{ulem}
\usepackage{marginnote}
\usepackage{floatpag}

\definecolor{darkgreen}{rgb}{0,0.45,0} 

  \newtheorem{proposition}{Proposition}[section]
  \newtheorem{lemma}[proposition]{Lemma}
  
  \newtheorem{theorem}[proposition]{Theorem}

  \theoremstyle{definition}
  \newtheorem{definition}[proposition]{Definition}
  \newtheorem{example}[proposition]{Example}
  
  \newtheorem{claim}[proposition]{}

\theoremstyle{remark}
  \newtheorem{remark}[proposition]{Remark}

  \newcounter{c}
  
  \newcommand{\etyk}[1]{\vspace{-7.4mm}$$\begin{equation}\Label{#1}
  \addtocounter{c}{1}}
  \renewcommand{\]}{\ifnum \value{c}=1 $$\else \end{equation}\fi}
\newenvironment{amssidewaysfigure}
  {\begin{sidewaysfigure}\vspace*{.5\textwidth}\begin{minipage}{\textheight}\centering}
  {\end{minipage}\end{sidewaysfigure}}

\newcommand{\black}{\color{black}}

\newcommand{\Longdownarrow}[1]{\mbox{\rotatebox{270}{$\Longrightarrow$} \raisebox{-12pt}{$#1$}}}
\newcommand{\Longuparrow}[1]{\mbox{\rotatebox{90}{$\Longrightarrow$} \raisebox{8pt}{$#1$}}}
\newcommand{\BiHom}{$\mathsf{BiHom}$}
\newcommand{\lax}[1]{\raisebox{-7pt}{$\stackrel {\scalebox{1.6}{$\bullet$}} {\scalebox{.6}{$#1$}}$}}
\newcommand{\oplax}[1]{\raisebox{-7pt}{$\stackrel {\scalebox{1.6}{$\circ$}} {\scalebox{.6}{$#1$}}$}}
\newcommand{\Lax}{\raisebox{-2pt}{\scalebox{1.6}{$\bullet$}}}
\newcommand{\OpLax}{\raisebox{-2pt}{\scalebox{1.6}{$\circ$}}}

\newcommand{\lw}{\raisebox{-2pt}{\scalebox{.5}[1.6]
{$\Leftcircle$}}}
\newcommand{\rw}{\raisebox{-2pt}{\scalebox{.5}[1.6]{$\Rightcircle$}}}
\newcommand{\lb}{\raisebox{-2pt}{\scalebox{.55}[1.6]{$\LEFTCIRCLE$}}}
\newcommand{\rb}{\raisebox{-2pt}{\scalebox{.55}[1.6]{$\RIGHTCIRCLE$}}}
\newcommand\bluesout{\bgroup\markoverwith{\textcolor{blue}{\rule[0.5ex]{2pt}{0.4pt}}}\ULon}

\begin{document}

\title{BiHom Hopf algebras viewed as Hopf monoids}

\author{Gabriella B\"ohm} 
\address{Wigner Research Centre for Physics, H-1525 Budapest 114,
P.O.B.\ 49, Hungary}
\email{bohm.gabriella@wigner.mta.hu}

\author{Joost Vercruysse}
\address{D\`epartement de Math\`ematiques, Facult\`e des Sciences,
Universit\`e Libre de Bruxelles, Boulevard du Triomphe, B-1050 Bruxelles, Belgium}
\email{jvercruy@ulb.ac.be}

\date{March 2020}

 
\begin{abstract}
We introduce monoidal categories whose monoidal products of any positive number of factors are lax coherent and whose 
nullary products are oplax coherent. We call them $\mathsf{Lax}^+\mathsf{Oplax}^0$-monoidal.
Dually, we consider $\mathsf{Lax}_0\mathsf{Oplax}_+$-monoidal categories which are oplax coherent for positive numbers of factors and lax coherent for nullary monoidal products.
We define $\mathsf{Lax}^+_0\mathsf{Oplax}^0_+$-
duoidal categories with compatible $\mathsf{Lax}^+\mathsf{Oplax}^0$- and $\mathsf{Lax}_0\mathsf{Oplax}_+$-monoidal structures. We introduce comonoids in $\mathsf{Lax}^+\mathsf{Oplax}^0$-monoidal categories, monoids in $\mathsf{Lax}_0\mathsf{Oplax}_+$-monoidal categories and bimonoids in $\mathsf{Lax}^+_0\mathsf{Oplax}^0_+$-
duoidal categories.

Motivation for these notions comes from a generalization of a construction in \cite{CaenepeelGoyvaerts}. This assigns a $\mathsf{Lax}^+_0\mathsf{Oplax}^0_+$-duoidal category $\mathsf D$ to any symmetric monoidal category $\mathsf V$. The 
unital \BiHom-monoids, counital \BiHom-comonoids, and unital and counital \BiHom-bimonoids of \cite{GraMakMenPan} in $\mathsf V$ are identified with the monoids, comonoids and bimonoids in $\mathsf D$, respectively.
\end{abstract}
 
\maketitle


\section*{Introduction} \label{sec:intro}

In recent years, lead by diverse motivations, several different generalizations of Hopf algebra have been proposed. Their similar features naturally raise the question whether they are instances of the same, more general notion. In \cite{BohmLack} several examples, such as groupoids, Hopf monoids in braided monoidal categories, Hopf algebroids over central base algebras and weak Hopf algebras, were unified as Hopf monoids in the duoidal (called 2-monoidal in \cite{AguiarMahajan}) endohom category of a naturally Frobenius map monoidale in a suitable monoidal bicategory. In \cite{Bohm} also Hopf group algebras \cite{Turaev}, Hopf categories \cite{BatCaeVer} and Hopf polyads \cite{Bruguieres} were shown to fit this framework.

The aim of the current paper is an interpretation of the (unital and counital) \BiHom-bimonoids in \cite{GraMakMenPan} as bimonoids in a category with some generalized duoidal structure. Recall that a {\em \BiHom-monoid} in a monoidal category $(\mathsf V,\otimes,I)$ (whose coherence natural isomorphisms are omitted) consists of an object $a$ together with morphisms $\alpha,\beta:a\to a$ and $\mu:a\otimes a \to a$ such that $\alpha$ and $\beta$ commute, both of them preserve the multiplication $\mu$, and, instead of the associativity of $\mu$, the first diagram of 
\begin{equation} \label{eq:BiHom} \tag{$\ast$}
\xymatrix@C=35pt{
a\otimes a \otimes a \ar[d]_-{\mu \otimes 1} &
\ar[l]_-{1\otimes 1\otimes \beta} 
a\otimes a \otimes a 
\ar[r]^-{\alpha \otimes 1 \otimes 1} &
a\otimes a \otimes a 
\ar[d]^-{1\otimes \mu} \\
a\otimes a \ar[r]_-\mu &
a &
\ar[l]^-\mu a\otimes a}
\qquad 
\xymatrix{
a \ar[d]_-{\eta \otimes 1} \ar[r]^-\beta & 
a \ar@{=}[d] &
\ar[l]_-\alpha a \ar[d]^-{1\otimes \eta} \\
a \otimes a \ar[r]_-\mu &
a &
\ar[l]^-\mu a\otimes a}
\end{equation}
commutes. A \BiHom-monoid $(a,\alpha,\beta,\mu)$ is said to possess a {\em unit} $\eta:I\to a$ if it is preserved by $\alpha$ and $\beta$ and also the second diagram above commutes.
Diagrams with reversed arrows define {\em (counital) \BiHom-comonoids}. The {\em (unital and counital) \BiHom-bimonoids} in a {\em symmetric} monoidal category consist of a (unital) \BiHom-monoid structure and a (counital) \BiHom-comonoid structure on the same object $a$ (with possibly different endomorphism parts $\alpha,\beta$ and $\kappa,\nu$) such that the multiplication $\mu$ (and the unit $\eta$) are morphisms of (counital) \BiHom-comonoids; equivalently, the comultiplication (and the counit) are morphisms of (unital) \BiHom-monoids.

One can see several attempts in the literature aiming at a description of (unital) \BiHom-monoids, (counital) \BiHom-comonoids and (unital and counital) \BiHom-bi\-monoids, respectively, as monoids, comonoids and bimonoids in a suitable category. All of these ideas originate from the same construction in \cite{CaenepeelGoyvaerts}. In \cite{CaenepeelGoyvaerts}, to any monoidal category $(\mathsf V,\otimes, I)$, a category is associated whose objects are pairs consisting of an object $a$ of $\mathsf V$ and an automorphism $\alpha$ of $a$. The morphisms $(a,\alpha) \to (a',\alpha')$ are morphisms $\phi: a\to a'$ in $\mathsf V$ such that $\phi.\alpha=\alpha'.\phi$. It is equipped with a monoidal structure with monoidal unit $(I,1)$, monoidal product $(a,\alpha) \otimes (a',\alpha'):=(a\otimes a',\alpha\otimes \alpha')$ and associativity and unitality coherence isomorphisms 
\begin{eqnarray*}
\alpha \otimes 1 \otimes \alpha^{\prime\prime-1}:
((a,\alpha) \otimes (a',\alpha')) \otimes (a^{\prime\prime},\alpha^{\prime\prime}) &\to& 
(a,\alpha) \otimes ((a',\alpha') \otimes (a^{\prime\prime},\alpha^{\prime\prime})),
\\
\alpha:(I,1) \otimes (a,\alpha) &\to&  (a,\alpha),
\\
\alpha:(a,\alpha) \otimes (I,1) &\to&  (a,\alpha).
\end{eqnarray*}
The monoids (respectively, comonoids) in this monoidal category are those unital \BiHom-monoids (respectively, those counital \BiHom-comonoids) in $\mathsf V$ whose two constituent endomorphisms are equal automorphisms.
They are called in \cite{CaenepeelGoyvaerts} {\em monoidal} $\mathsf{Hom}$-monoids (respectively, {\em monoidal} $\mathsf{Hom}$-comonoids).
Any possible symmetry on $(\mathsf V,\otimes, i)$
is inherited by this monoidal category of \cite{CaenepeelGoyvaerts}. 
Then the bimonoids therein are those unital and counital \BiHom-bimonoids in $\mathsf V$ whose comonoid part contains two copies of an automorphism and the monoid part contains two copies of its inverse. They are called in \cite{CaenepeelGoyvaerts} {\em monoidal} $\mathsf{Hom}$-bimonoids.

The above construction from \cite{CaenepeelGoyvaerts} was widely generalized in \cite[Section 2]{GraMakMenPan}. In the particular case which is relevant here, the objects are triples consisting of an object $a$ of $\mathsf V$ and {\em two} commuting automorphisms $\alpha$ and $\beta$ of $a$. The resulting category also admits a similar (symmetric) monoidal structure (where half of the occurring morphisms $\alpha$ is replaced by $\beta$). The monoids, comonoids and bimonoids in this (symmetric) monoidal category cover also those unital \BiHom-monoids, counital \BiHom-comonoids and unital and counital \BiHom-bimonoids, respectively, in which two possibly unrelated automorphisms occur.
Such \BiHom-objects are called in \cite{GraMakMenPan} {\em monoidal} \BiHom-monoids, {\em monoidal} \BiHom-comonoids and {\em monoidal} \BiHom-bimonoids respectively.

So the essential observation, originated in \cite{CaenepeelGoyvaerts}, is that these {\em monoidal} $\mathsf{Hom}$- and \BiHom-structures are monoids, comonoids and bimonoids in a suitable (symmetric) monoidal category. Consequently, the standard machinery applies to them. This explains why many aspects of this theory have been obtained in recent literature with a mild adaptation of the classical proofs.

However, more general \BiHom-structures, with not necessarily invertible endomorphism parts, are not covered by this construction, and it is the main aim of this paper to provide a categorical construction that covers these cases.

Firstly, let us remark that t\black he category whose objects consist of an object $a$ of a symmetric monoidal category $\mathsf V$ together with {\em four} commuting automorphisms of $a$ carries a {\em duoidal} structure (termed {\em 2-monoidal} in \cite{AguiarMahajan}). One monoidal structure is defined in terms of one half of the automorphisms and a second monoidal structure is defined in terms of the other half. The compatibility morphisms are 
provided by the symmetry of $\mathsf V$. The monoids, comonoids and bimonoids in this duoidal category finally cover also those unital \BiHom-monoids, counital \BiHom-comonoids and unital and counital \BiHom-bimonoids, respectively, in which four possibly different, but still invertible endomorphisms occur.
(A construction of this flavor occurred in \cite{ZhangWang:BiHom}. However, that construction is based on four monoidal natural automorphisms of the identity functor on $\mathsf V$. This seems to be quite restrictive since the identity functor on the category of vector spaces, for example, admits no other monoidal natural automorphism but the identity.)

In order to deal with \BiHom-structures with not necessarily invertible endomorphism parts, one should give up the invertibility of the coherence natural transformations and use the {\em unbiased} variants of monoidal category \cite{Leinster}. Although there is a meaningful theory of duoidal categories with one lax and one oplax monoidal structure \cite[Section 4]{DayStreet}, for the description of \BiHom-structures neither the lax nor the oplax variant looks suitable but rather a certain mixture of both. Indeed, for the formulation of the diagrams of \eqref{eq:BiHom}, one needs coherence morphisms of the kind
$$
\xymatrix{
(a\otimes b) \otimes c & 
\ar[l] a\otimes b \otimes c \ar[r] &
a\otimes (b\otimes c)}
\qquad
\xymatrix{
I \otimes a \ar[r] &
a &
\ar[l] a\otimes I.}
$$
Those on the left are of the oplax type, while those on the right are of the lax type.
In this paper we define unbiased monoidal categories of the above mixed type. 

As a first step, we consider $\mathsf{Lax}^+$- (respectively, $\mathsf{Oplax}_+$- ) monoidal categories with monoidal products of only positive number of factors with lax (respectively, oplax) coherence morphisms. We define cosemigroups (respectively, semigroups) in such monoidal categories. We also define $\mathsf{Lax}^+\mathsf{Oplax}_+$-duoidal categories which possess compatible $\mathsf{Lax}^+$- and $\mathsf{Oplax}_+$-monoidal structures. 
Semigroups in a $\mathsf{Lax}^+\mathsf{Oplax}_+$-duoidal category are shown to constitute a $\mathsf{Lax}^+$-monoidal category; dually, cosemigroups in a $\mathsf{Lax}^+\mathsf{Oplax}_+$-duoidal category are shown to constitute an $\mathsf{Oplax}_+$-monoidal category. So we can define bisemigroups in a $\mathsf{Lax}^+\mathsf{Oplax}_+$-duoidal category as cosemigroups in the category of semigroups; equivalently, as semigroups in the category of cosemigroups.

Next we introduce so called $\mathsf{Lax}_0\mathsf{Oplax}_+$-monoidal categories, which have $n$-fold monoidal products for any non-negative integer $n$ such that the monoidal products of positive number of factors are oplax coherent while the $0$-fold monoidal product is lax coherent. 
Forgetting about $0$-fold monoidal products, a $\mathsf{Lax}_0\mathsf{Oplax}_+$-monoidal category can be seen $\mathsf{Oplax}_+$-monoidal.
We define monoids in $\mathsf{Lax}_0\mathsf{Oplax}_+$-monoidal categories.
Dually, we introduce so called $\mathsf{Lax}^+\mathsf{Oplax}^0$-monoidal categories, again with $n$-fold monoidal products for any non-negative integer $n$ such that the monoidal products of positive number of factors are lax coherent while the $0$-fold monoidal product is oplax coherent. We define comonoids in $\mathsf{Lax}^+\mathsf{Oplax}^0$-monoidal categories.
Finally we define $\mathsf{Lax}^+_0\mathsf{Oplax}^0_+$-duoidal categories with compatible $\mathsf{Lax}^+\mathsf{Oplax}^0$- and $\mathsf{Lax}_0\mathsf{Oplax}_+$-monoidal structures. Monoids in them are shown to constitute a $\mathsf{Lax}_0\mathsf{Oplax}_+$-monoidal category and, dually, comonoids in them are shown to constitute a $\mathsf{Lax}^+\mathsf{Oplax}^0$-monoidal category. So we can define bimonoids in them as comonoids in the category of monoids; equivalently, as monoids in the category of comonoids.

Generalizing the construction in \cite{CaenepeelGoyvaerts}, we associate a $\mathsf{Lax}^+_0\mathsf{Oplax}^0_+$-duoidal category $\mathsf D$ to any symmetric monoidal category $\mathsf V$. We identify the semigroups, cosemigroups and bisemigroups in $\mathsf D$ with the \BiHom-monoids, \BiHom-comonoids and the \BiHom-bimonoids in $\mathsf V$, respectively. We also identify the monoids, comonoids and bimonoids in $\mathsf D$ with the unital \BiHom-monoids, counital \BiHom-comonoids and the unital and counital \BiHom-bimonoids in $\mathsf V$, respectively, in the sense of \cite{GraMakMenPan}.

The consequence of our construction is twofold. Firstly, by extending the {\em microcosm} hosting bimonoids \cite{BaezDolan}, it provides a suitable categorical framework where \BiHom-objects with not necessarily invertible endomorphisms can be studied. On the other hand, our proposed setting goes beyond the known framework of symmetric monoidal categories, and even of duoidal categories. In this way results for such objects can no longer directly be obtained from the existing theory of monoids, comonoids a bimonoids in duoidal categories, let alone in (symmetric) monoidal categories. Indeed, one would need first to extend this theory to the framework of $\mathsf{Lax}^+_0\mathsf{Oplax}^0_+$-duoidal categories, which is more involved than the classical case. This indicates that the theory of \BiHom-objects with not necessarily invertible endomorphisms is itself more involved; no longer a straightforward generalization of the classical theory.
\black

The same axioms defining our various lax and oplax monoidal categories can be used to define monoids of the same lax and oplax type in any Gray monoid. 
Such monoids can be seen as the 0-cells of a 2-category.
If the ambient Gray monoid is a symmetric strict monoidal 2-category, then this 2-category of monoids is again strict monoidal. In this case we can also define duoids of various lax and oplax type, as suitably lax monoids in the strict monoidal 2-category of suitably lax monoids. 

However, not to distract attention from the main aim to describe \BiHom-structures, we do not work at the level of generality of the previous paragraph. We restrict our study to the symmetric strict monoidal 2-category $\mathsf{Cat}$ of categories, functors, and natural transformations. 

We do not assume that the monoidal categories in the paper are strict monoidal but --- relying on coherence --- we omit explicitly denoting their Mac Lane type coherence natural isomorphisms.

\subsection*{Acknowledgement}
GB is grateful for the financial support by the Hungarian National Research, Development and Innovation Office – NKFIH (grant K124138).
JV thanks the FNRS (National Research Fund of the French speaking community in Belgium) for support via the MIS project `Antipode' (grant F.4502.18).

\section{Some combinatorial background} \label{sec:preli}

As a preparation for the constructions of the subsequent sections, we begin with recording some notation and a few technical results that will be heavily applied later.

\begin{claim}
In the 2-category $\mathsf{Cat}$ of categories, functors and natural transformations we denote by $\cdot$ the composition of functors and the corresponding Godement product of natural transformations.
$\mathsf{Cat}$ is strict monoidal via the Cartesian product $\times$ and symmetric via the flip maps. We often denote the Cartesian product (of any of categories, functors, and natural transformations) by juxtaposition. The Cartesian product of $n$ copies of the same category $\mathsf A$ is also denoted by $\mathsf A^n$. By $\mathsf A^0$ we mean the singleton category $\mathbbm 1$ for any category $\mathsf A$.
For all non-negative integers $n,p$, consider the particular components
$$
\xymatrix@C=2pt @R=4pt{
\mathsf{Cat}^{np} \ar@{=}[r] \ar@{=}[d] 
\ar@{}[rrrrdddd]|-{\Longdownarrow {\tau_{np}}} &
(\mathsf{Cat}^{p} )^n \ar[rrr]^-{\stackrel \times n}  &&&
\mathsf{Cat}^p \ar[dddd]^-{\stackrel \times p} \\
(\mathsf{Cat}^{n})^p
\ar[ddd]_-{\stackrel \times p} \\
\\
\\
\mathsf{Cat}^n \ar[rrrr]_-{\stackrel \times n}  &&&&
\mathsf{Cat}}
$$
of the symmetry. They are 2-natural transformations whose component at any double sequence of categories $\{\mathsf A_{ij}\}_{\stackrel{i=1,\dots,p}{{}_{j=1,\dots,n}}}$ is the flip functor
\begin{eqnarray*}
& (A_{11},\dots ,A_{1n}), \dots, (A_{p1},\dots,A_{pn}) &\to
(A_{11},\dots, A_{p1}), \dots, (A_{1n},\dots, A_{pn}),
\\
& (a_{11},\dots, a_{1n}), \dots, (a_{p1},\dots, a_{pn}) &\mapsto
(a_{11},\dots ,a_{p1}), \dots, (a_{1n},\dots ,a_{pn}).
\end{eqnarray*}
They satisfy the equalities
\begin{equation}\label{eq:tau}
\tau_{np}\cdot \tau_{pn}=1
\qquad \textrm{and} \qquad 
\tau_{n1}=1
\qquad  \textrm{and} \qquad
\tau_{np} \cdot (\tau_{nk_1} \dots \tau_{nk_p})=
\tau_{n\sum_{i=1}^pk_i}
\end{equation}
for all non-negative integers $n,p$ and $k_1,\dots,k_p$.
\end{claim}

\begin{claim}
We adopt the convention that empty sums are equal to zero. 
The set of non-negative integers is denoted by $\mathbb N$ and the set of (strictly) positive integers is denoted by $\mathbb N^+$.

The following maps $\mathbb N\to \mathbb N$ will occur frequently.
$$
Z(m):= \left\{
\begin{array}{ll}
0 & \textrm{if} \quad m>0\\
1 & \textrm{if} \quad m=0
\end{array} \right.
\qquad \qquad
\overline{m}:=m+Z(m).
$$

For $n\in\mathbb N$ and a sequence of non-negative integers $\{m_1,\dots,m_n\}$ we denote the sum and the number of zeros, respectively by
$$
M:=\sum_{i=1}^n m_i
\qquad \textrm{and}\qquad 
Z:= \sum_{i=1}^n Z(m_i).
$$
Clearly, $M+Z \geq n$ and thus $M+Z=0$ iff $n=0$.

For a double sequence of non-negative integers 
$\{k_{ij}\}_{\stackrel{i=1,\dots,n}{{}_{j=1,\dots,m_i}}}$ (also in the particular case when $n=1$), we put
$$
K_i:=\sum_{j=1}^{m_i} k_{ij} 
\qquad \qquad
Z_i:=\sum_{j=1}^{m_i} Z(k_{ij})
\qquad \qquad
K:=\sum_{i=1}^n K_i 
\qquad \qquad
Z:=\sum_{i=1}^n Z_i.
$$
We introduce two associated double sequences with $i=1,\dots,n$ and $j=1,\dots,\overline m_i$:
$$
\widetilde k_{ij}:=\left\{
\begin{array}{ll}
k_{ij} & \textrm{if} \quad m_i >0 \\
0 & \textrm{if} \quad m_i =0
\end{array}
\right.
\qquad \qquad
\widehat k_{ij}:=\left\{
\begin{array}{ll}
k_{ij} & \textrm{if} \quad K_i >0 \\
1 & \textrm{if} \quad K_i =0, m_i>0, j=1 \\
0 & \textrm{if} \quad K_i =0, m_i>0, j>1 \\
1 & \textrm{if} \quad m_i =0.
\end{array}
\right.
$$
We put
$$
\widetilde Z:=\sum_{i=1}^n \sum_{j=1}^{\overline m_i} 
Z(\widetilde k_{ij})=
Z+\sum_{i=1}^n Z(m_i)
=Z+\sum_{i=1}^n Z(K_i+Z_i).
$$
Note that
$$
\sum_{i=1}^n \sum_{j=1}^{\overline m_i} 
(\widehat k_{ij}+Z(\widehat k_{ij}))=
\sum_{i=1}^n \sum_{j=1}^{\overline m_i} 
(\widetilde k_{ij}+Z(\widetilde k_{ij}))=
\sum_{i=1}^n K_i+ Z_i + Z(K_i+Z_i)=
K+\widetilde Z.
$$
\end{claim}

For any non-negative integer $k$, and for any functor $F$ from the singleton category $\mathbbm 1$ to an arbitrary category $\mathsf A$, we introduce the functor
\begin{equation} \label{eq:[k]} 
\xymatrix@C=17pt{\mathsf A^k \ar[r]^-{[k]} & \mathsf  A^{\overline k}}:=\left\{
\begin{array}{ll}
\xymatrix@C=17pt{\mathsf A^{k} \ar[r]^-1 & \mathsf A^{k}}
& \textrm{if} \quad k>0 \\
\xymatrix{\,  \mathbbm 1  \ar[r]^-F & \mathsf A\ }
& \textrm{if} \quad k=0
\end{array}
\right.
\end{equation}
The proof of the following easy lemma about such functors in \eqref{eq:[k]} is left to the reader. 

\begin{lemma} \label{lem:[F]}
For any sequence of non-negative integers $\{k_1,\dots,k_m\}$, for any functor $F$ from the singleton category $\mathbbm 1$ to an arbitrary category $\mathsf A$, and for the associated functors of \eqref{eq:[k]}, the following diagrams commute.
$$
\xymatrix@C=50pt{
\mathsf A^K \ar[r]^-{[k_1] \dots[k_m]}
\ar[rd]_-{[\widetilde k_1] \dots[\widetilde k_{\overline m}]} &
\mathsf A^{K+Z} \ar[d]^-{[K+Z]} \\
& \mathsf A^{K+\widetilde Z}}
\qquad \qquad
\xymatrix@C=50pt{
\mathsf A \ar[r]^-{[K]}
\ar[rd]_-{[\widetilde k_1] \dots[\widetilde k_{\overline m}]} &
\mathsf A^{K+Z(K)}
\ar[d]^-{[\widehat k_1] \dots[\widehat k_{\overline m}]} \\
& \mathsf A^{K+\widetilde Z}}
$$
\end{lemma}

\section{Cosemigroups in $\mathsf{Lax}^+$-monoidal categories}

Following the {\em microcosm} principle \cite[Section 4.3]{BaezDolan}, in order to define semigroups and cosemigroups internally to some monoidal category, one only needs to have products of positive number of factors, not necessarily nullary ones. In this section we introduce cosemigroups in monoidal categories with monoidal products of arbitrary positive number of factors and lax coherence morphisms between them. The dual situation is discussed in the next section.

Throughout, we use the notation introduced in Section \ref{sec:preli}.

\begin{definition}
\label{def:lax+}
A {\em $\mathsf{Lax}^+$-monoidal category} consists of 
\begin{itemize}
\item a category $\mathsf L$
\item for all positive integers $n$, a functor $\lax{n}:{
\mathsf L}^{n} \to \mathsf L$ 
\item for all positive integers $n,k_1,\dots,k_n$,
natural transformations
$$
\xymatrix@C=80pt{
{\mathsf L}^K
\ar[r]^-{\lax {k_1} \ \cdots \ \lax{k_n}}
\ar@/_1.6pc/[rd]_-{\lax {K}}
\ar@{}[rd]|-{\qquad \Longdownarrow {\Phi_{k_1,\dots,k_n}}} &
{\mathsf L}^n \ar[d]^-{\lax n}  \\
& {\mathsf L}}
\qquad\qquad
\raisebox{-17pt}{$
\xymatrix@C=60pt{
{\mathsf L} \ar@{=}@/^1.7pc/[r] \ar@/_1.7pc/[r]_-{\lax 1}
\ar@{}[r]|-{\Longdownarrow \iota} &
\mathsf L}$}
$$
\end{itemize}
such that the diagrams
$$
\xymatrix@C=45pt@R=24pt{
\lax n \!\cdot \!
(\lax {m_1}\cdots \lax{m_n}) \! \cdot \!
(\raisebox{-1.5pt}{$\lax{k_{11}} \cdots \lax {k_{1m_1}} \cdots \lax{k_{n1}} \cdots \lax{k_{nm_n}}$})
\ar[r]^-{\Phi_{m_1,\dots,m_n} \cdot 1}
\ar[d]^(.45){1\cdot(\Phi_{k_{11},\dots,k_{1m_1}} \cdots \Phi_{k_{n1},\dots,k_{nm_n}})} &
\lax M \!\cdot \!
(\lax{k_{11}} \cdots \lax {k_{1m_1}} \cdots \lax{k_{n1}} \cdots \lax{k_{nm_n}}) \ar[d]_(.7){\Phi_{k_{11},\dots,k_{nm_n}}} \\
\lax n \cdot 
(\raisebox{-1.5pt}{$\lax{K_1} \cdots \lax{K_n}$})
\ar[r]_-{\Phi_{K_1,\dots,K_n}} &
\lax K}
\hspace{-.5cm}
\xymatrix@C=30pt@R=26pt{
\lax n \ar[r]^-{\iota\cdot 1} \ar[d]^(.6){1\cdot (\iota \cdots \iota)}
\ar@{=}[rd] &
\raisebox{-1.5pt}{$\lax 1$} \cdot \lax n \ar[d]^-{\Phi_{n}} \\
\lax n \cdot (\raisebox{-1.5pt}{$\lax 1 \cdots \lax 1$}) \ar[r]_-{\Phi_{1,\dots,1}} &
\lax n}
$$
commute, for any positive integers $n$, $\{m_i\}$ labelled by ${i=1,\dots,n}$, and $\{k_{ij}\}$ labelled by $i=1,\dots ,n$ and $j=1,\dots,m_i$.

A $\mathsf{Lax}^+$-monoidal category is {\em normal} if $\iota$ is invertible and {\em strict normal} if $\iota$ is the identity natural transformation.
\end{definition}

More generally, $\mathsf{Lax}^+$-monoids can be defined in any Gray monoid --- so in particular in any strict monoidal 2-category --- by the same diagrams in Definition \ref{def:lax+}. $\mathsf{Lax}^+$-monoidal categories are then re-obtained by applying it to $\mathsf{Cat}$.

Throughout, in a $\mathsf{Lax}^+$-monoidal category we denote by $\lb a \rb$ the image of any object $a$ under the functor $\lax 1: \mathsf L \to \mathsf L$; and we denote by $a_1 \Lax \cdots \Lax a_n$ the image of an object $(a_1,\dots,a_n)$ under the functor $\lax n:\mathsf L^n \to \mathsf L$.

\begin{proposition}
\label{prop:Cartesian+}
The Cartesian product $\mathsf L \mathsf L'$ of $\mathsf{Lax}^+$-monoidal categories $(\mathsf L,\Lax,\Phi,\iota)$ and $(\mathsf L',\Lax',\Phi',\iota')$ is again $\mathsf{Lax}^+$-monoidal via 
\begin{itemize}
\item the functors 
$\xymatrix{
(\mathsf L \mathsf L')^n \ar[r]^-{\tau_{2n}} &
\mathsf L^n \mathsf L^{\prime n} \ar[r]^-{\lax n \lax n'} &
\mathsf L \mathsf L'}$
\item the natural transformations $\iota\iota':1\to \lax 1\lax 1'$ and
$$
\xymatrix@C=90pt {
(\mathsf L \mathsf L')^K \ar[r]^-{\tau_{2k_1}\cdots \tau_{2k_n}}
\ar@/_1.5pc/[rd]_-{\tau_{2K}} &
\mathsf L^{k_1} \mathsf L^{\prime k_1} \cdots 
\mathsf L^{k_n} \mathsf L^{\prime k_n}
\ar[r]^-{\lax {k_1} \lax {k_1}' \cdots \lax {k_n} \lax {k_n}'}
\ar[d]_-{\tau_{2n}} &
(\mathsf L \mathsf L')^n \ar[d]^-{\tau_{2n}} \\
& \mathsf L^K \mathsf L^{\prime K}
\ar@/_1.5pc/[rd]_-{\lax {K} \lax {K}'}
\ar@{}[rd]|-{\qquad \ \ \Longdownarrow {\Phi_{k_1,\dots ,k_n}\Phi'_{k_1,\dots ,k_n}}}
\ar[r]^-{\lax {k_1} \cdots \lax {k_n}\lax {k_1}' \cdots \lax {k_n}'} &
\mathsf L^n \mathsf L^{\prime n} 
\ar[d]^-{\lax n \lax n'} \\
& & \mathsf L \mathsf L'}
$$
where the unlabelled regions denote identity natural transformations.
\end{itemize}
\end{proposition}

\begin{proof}
It is left to the reader to check that the stated datum satisfies the axioms in Definition \ref{def:lax+} (analogously to Example (6) on page 80 of \cite{DayStreet:lax}). 
\end{proof}

\begin{definition}
\label{def:lax+-functor}
A {\em $\mathsf{Lax}^+$-monoidal functor} consists of
\begin{itemize}
\item a functor $G:\mathsf L \to \mathsf L'$
\item for all positive integers $n$, a natural transformation
$\Gamma_n:G\cdot \lax n \to \lax n'\cdot (G\cdots G)$
\end{itemize}
such that for all sequences of positive integers $\{k_1,\dots,k_n\}$ the following diagrams commute.
$$
\xymatrix@R=67pt{
G \ar[rd]_-{\iota' \cdot 1} \ar[r]^-{1\cdot \iota} &
G \cdot  \lax 1 \ar[d]^-{\Gamma_1} \\
& \lax 1' \cdot G}
\qquad
\xymatrix@C=45pt@R=20pt{
G\cdot \lax n \cdot (\raisebox{-1.5pt}{$\lax{k_{1}} \cdots \lax{k_{n}}$})
\ar[d]_-{\Gamma_n\cdot 1}
\ar[r]^-{1\cdot \Phi_{k_1,\dots,k_n}} &
G \cdot \raisebox{-1.5pt}{$\lax{K}$} 
\ar[dd]^-{\Gamma_{K}}  \\
\lax n'\cdot (G \cdots G) \cdot (\raisebox{-1.5pt}{$\lax{k_{1}} \cdots \lax{k_{n}}$})
\ar[d]_-{1\cdot (\Gamma_{k_1} \cdots \Gamma_{k_n})}  \\
\lax n'\cdot (\raisebox{-1.5pt}{$\lax{k_{1}}' \cdots \lax{k_{n}}'$})
\cdot (G \cdots G)
\ar[r]_-{\Phi'_{k_1,\dots,k_n}\cdot 1} &
\raisebox{-1.5pt}{$\lax{K}'$} \cdot (G \cdots G) }
$$
\end{definition}

\begin{proposition}
\label{prop:comp_functor}
\begin{itemize}
\item[{(1)}]
The composite of (composable) $\mathsf{Lax}^+$-monoidal functors $(G,\Gamma)$ and $(H,\Xi)$ is again $\mathsf{Lax}^+$-monoidal via the natural transformations
$$
\xymatrix{
H\cdot G \cdot \lax n \ar[r]^-{1\cdot \Gamma_n} &
H\cdot  \lax n \cdot (G \cdots G) \ar[r]^-{\Xi_n\cdot 1} &
\lax n \cdot (H\cdots H) \cdot (G \cdots G)
=\lax n \cdot (H\cdot G)^n.}
$$
\item[{(2)}]
The Cartesian product of $\mathsf{Lax}^+$-monoidal functors $(G,\Gamma)$ and $(H,\Xi)$ is again $\mathsf{Lax}^+$-monoidal via the natural transformations
$$
\xymatrix@C=60pt{
(\mathsf{LN})^n \ar[r]^-{\tau_{2n}} \ar[d]_-{(GH)^n} &
\mathsf{L}^n \mathsf{N}^n \ar[r]^-{\lax n \lax n} \ar[d]_-{G^nH^n}
\ar@{}[rd]|-{\Longdownarrow {\Gamma_n\Xi_n}}  &
\mathsf{LN} \ar[d]^-{GH} \\
(\mathsf{L'N'})^n \ar[r]_-{\tau_{2n}}  &
\mathsf{L'}^n \mathsf{N'}^n \ar[r]_-{\lax n' \lax n'} &
\mathsf{L'N'}}
$$
where the unlabelled region denotes the identity natural transformation.
\end{itemize}
\end{proposition}

\begin{proof}
It is left to the reader to check that the stated data satisfy the axioms in Definition \ref{def:lax+-functor}.
\end{proof}

\begin{definition}
\label{def:lax+-nattr}
A {\em $\mathsf{Lax}^+$-monoidal natural transformation} $(G,\Gamma) \to (G',\Gamma')$ is a natural transformation $\omega:G \to G'$ for which the following diagram commutes for all positive integers $n$.
$$
\xymatrix@C=45pt{
G\cdot \lax n  \ar[r]^-{\omega \cdot 1} \ar[d]_-{\Gamma_n} &
G'\cdot \lax n \ar[d]^-{\Gamma'_n} \\
\lax n'\cdot (G \cdots G) \ar[r]_-{1\cdot (\omega \cdots \omega)} &
\lax n'\cdot (G'\cdots G')}
$$
\end{definition}

\begin{proposition}
\label{prop:comp_nattr}
All of the composites, the Godement products, and the Cartesian products of (composable) $\mathsf{Lax}^+$-monoidal natural transformations are again $\mathsf{Lax}^+$-monoidal. Consequently, there is a strict monoidal 2-category $\mathsf{Lax}^+$ of $\mathsf{Lax}^+$-monoidal categories, $\mathsf{Lax}^+$-monoidal functors and $\mathsf{Lax}^+$-monoidal natural transformations.
\end{proposition}

\begin{proof}
It is left to the reader to check that the diagram of Definition \ref{def:lax+-nattr} commutes in all of the stated cases.
\end{proof}

\begin{definition} \label{def:cosemigroup}
A {\em cosemigroup} in a $\mathsf{Lax}^+$-monoidal category $(\mathsf L,\Lax,\Phi,\iota)$ consists of 
\begin{itemize}
\item an object $a$ of $\mathsf L$ and 
\item a morphism $\delta:a \to a \Lax a$
\end{itemize}
rendering commutative the first (coassociativity) diagram below.
A {\em morphism of cosemigroups} $(a,\delta) \to (a',\delta')$ is a morphism $\phi:a \to a'$ rendering commutative the second diagram.
$$
\xymatrix{
a\Lax a \ar[d]_-{\delta \bullet \iota} &
a \ar[l]_-\delta \ar[r]^-\delta &
a\Lax a \ar[d]^-{\iota \bullet \delta} \\
(a\Lax a) \Lax \lb a \rb \ar[r]_-{\Phi_{2,1}} &
a\Lax a \Lax a &
\lb a\rb\Lax (a\Lax a) \ar[l]^-{\Phi_{1,2}} }
\qquad 
\xymatrix{
a \ar[r]^-\delta \ar[d]_-\phi &
a \Lax a \ar[d]^-{\phi \bullet \phi} \\
a' \ar[r]_-{\delta'} &
a' \Lax a'}
$$
\end{definition}

\begin{theorem}
\label{thm:cosemigroup}
For any $\mathsf{Lax}^+$-monoidal category $(\mathsf L,\Lax,\Phi,\iota)$, the following categories are isomorphic.
\begin{itemize}
\item[{(i)}]
The category of cosemigroups and their morphisms in $(\mathsf L,\Lax,\Phi,\iota)$.
\item[{(ii)}]
The category of $\mathsf{Lax}^+$-monoidal functors $(\mathbbm 1,1,1,1) \to (\mathsf L,\Lax,\Phi,\iota)$ and their $\mathsf{Lax}^+$-monoidal natural transformations. 
\end{itemize}
Consequently, ${\sf Lax}^+$-monoidal functors preserve cosemigroups.
\end{theorem}

\begin{proof}
A $\mathsf{Lax}^+$-monoidal functor $(\mathbbm 1,1,1,1) \to (\mathsf L,\Lax,\Phi,\iota)$ as in part (ii) consists of 
\begin{itemize}
\item
an object $a$ of $\mathsf L$
\item for all positive integers $n$, a
morphism $\delta_n$ from $a$ to the $n$-fold product $a \Lax \cdots \Lax a$ 
\end{itemize}
such that $\delta_1=\iota$ and for all sequences of positive integers $\{k_1,\dots,k_n\}$ the following diagram commutes.
\begin{equation}
\label{eq:n-coass}
\xymatrix@C=60pt{
a \ar[r]^-{\delta_n} \ar[d]_-{\delta_K} &
a\Lax \cdots \Lax a \ar[d]^-{\delta_{k_1}{\bullet} \cdots {\bullet} \delta_{k_n}} \\
a\Lax \cdots \Lax a &
\ar[l]^-{\Phi_{k_1,\dots,k_n}}
(a\Lax \cdots \Lax a) \Lax \cdots \Lax (a\Lax \cdots \Lax a)}
\end{equation}
Comparing \eqref{eq:n-coass} at $n=2$, $k_1=2$, $k_2=1$ and $n=2$, $k_1=1$, $k_2=2$, we see that $(a,\delta_2)$ is a cosemigroup.

Conversely, starting with a cosemigroup $(a,\delta)$, we put $\delta_1:=\iota$, $\delta_2:=\delta$ and for $i >1$ iteratively 
\begin{equation}\label{eq:delta_k+1}
\delta_{i +1}:=(
\xymatrix{
a \ar[r]^-\delta &
a\bullet a \ar[r]^-{\iota \bullet \delta_i } &
\lb a \rb \Lax (a\Lax \cdots \Lax a) \ar[r]^-{\Phi_{1,i}} &
a\Lax \cdots \Lax a} ).
\end{equation}
Then the coassociativity of the cosemigroup tells exactly that $\delta_3$ equals
$$
\xymatrix{
a \ar[r]^-{\delta} &
a\Lax a 
\ar[r]^-{\delta \bullet \iota} &
(a \Lax a) \Lax \lb a \rb 
\ar[r]^-{\Phi_{2,1}} &
a\Lax a \Lax a.} 
$$
By induction on $i$, one proves that the morphism of \eqref{eq:delta_k+1} is equal to
\begin{equation} \label{eq:delta_alt}
\xymatrix@C=20pt{
a \ar[r]^-{\delta_i} &
a\Lax \cdots \Lax a 
\ar[rr]^-{\delta \bullet \iota \bullet \cdots \bullet \iota} &&
(a \Lax a) \Lax \lb a \rb \Lax \cdots \Lax \lb a \rb
\ar[rr]^-{\Phi_{2,1,\dots,1}} &&
a\Lax \cdots \Lax a}
\end{equation}
for all $i>2$.

Now we check, by induction on $K$, that the so defined family of morphisms $\{\delta_i\}_{i\in \mathbb N^+}$ renders commutative \eqref{eq:n-coass} for all sequences of positive integers $\{k_1,\dots,k_n\}$. For $K=1$ both $n$ and $k_{1}$ must be equal to $1$ and thus \eqref{eq:n-coass} commutes by one of the unitality axioms in Definition \ref{def:lax+}.
There are two kinds of induction step:
\begin{itemize}
\item replacing $\{k_1,\dots,k_n\}$ with $\{1,k_1,\dots,k_n\}$ and
\item replacing $\{k_1,\dots,k_n\}$ with $\{1+k_1,\dots,k_n\}$.
\end{itemize}
Assuming that \eqref{eq:n-coass} commutes for the original sequence $\{k_1,\dots,k_n\}$, it is seen to commute for the modified sequence using the $\mathsf{Lax}^+$-monoidal category axioms in Definition \ref{def:lax+} together with \eqref{eq:delta_k+1} in the first case and \eqref{eq:delta_alt} in the second case.

From the constructions, it is obvious that the cosemigroup constructed from the ${\sf Lax}^+$-monoidal functor induced by a cosemigroup is exactly the initial cosemigroup.
In order to see triviality of the composite of the above constructions in the opposite order, note that \eqref{eq:n-coass} for $n=2$, $k_1=1$, and $k_2=i$ yields \eqref{eq:delta_k+1}.

A $\mathsf{Lax}^+$-monoidal natural transformation $(a,\{\delta_i\}_{i\in \mathbb N^+}) \to (a',\{\delta'_i\}_{i\in \mathbb N^+})$ is a morphism $f:a\to a'$ such that for all positive integers $i$,
\begin{equation}
\label{eq:cosgr_mor}
\xymatrix@C=50pt{
a\ar[r]^-{\delta_i} \ar[d]_-f &
a\Lax \cdots \Lax a \ar[d]^-{f\bullet \cdots \bullet f} \\
a'\ar[r]_-{\delta'_i} &
a'\Lax \cdots \Lax a'}
\end{equation}
commutes. Such a morphism $f$ is clearly a cosemigroup morphism $(a,\delta_2) \to (a',\delta'_2)$.

Conversely, for a cosemigroup morphism $f:(a,\delta) \to (a',\delta')$, the diagram of \eqref{eq:cosgr_mor} commutes for $i=1$ by the naturality of $\delta_1=\iota =\delta'_1$. For the morphisms \eqref{eq:delta_k+1} for $i>1$, commutativity of \eqref{eq:cosgr_mor} is easily checked by induction on $i$.
\end{proof}

We close this section by introducing a natural notion of comodule over a cosemigroup in the sense of Definition \ref{def:cosemigroup}, without entering its deeper analysis.

\begin{definition} \label{def:comodule}
A {\em comodule} of a cosemigroup $(a,\delta)$ in a $\mathsf{Lax}^+$-monoidal category $(\mathsf L,\Lax,\Phi,\iota)$ consists of
\begin{itemize}
\item an object $x$ of $\mathsf L$
\item a morphism $\varrho: x \to x \Lax a$ (the {\em coaction})
\end{itemize}
such that the (coassociativity) diagram on the left of
$$
\xymatrix{
x \Lax a  \ar[d]_-{\iota \bullet \delta} & 
\ar[l]_-\varrho  x \ar[r]^\varrho & 
x \Lax a \ar[d]^-{\varrho \bullet \iota} \\
\lb x \rb \Lax (a \Lax a) \ar[r]_-{\Phi_{1,2}} &
x \Lax a \Lax a 
& \ar[l]^-{\Phi_{2,1}} (x \Lax a) \Lax \lb a \rb }
\qquad
\xymatrix{
x \ar[d]_-\theta \ar[r]^-\varrho &
x \Lax a  \ar[d]^-{\theta  \bullet 1}  \\
x' \ar[r]_-{\varrho'}  & 
x' \Lax a}
$$
commutes.
A {\em morphism of comodules} $(x,\varrho) \to (x',\varrho')$ is a morphism $\theta:x \to x'$ in $\mathsf L$ such that the diagram on the right above commutes.
\end{definition}

\begin{example} \label{ex:reg_comod}
Comparing the coassociativity conditions of Definition \ref{def:comonoid} and Definition \ref{def:comodule}, it becomes obvious that $(a,\delta)$ is a comodule over an arbitrary cosemigroup $(a,\delta)$ in any $\mathsf{Lax}^+$-monoidal category.

Notice, however, that it is not clear if the forgetful functor, from the category of comodules over a cosemigroup $(a,\delta)$ in some $\mathsf{Lax}^+$-monoidal category $(\mathsf L,\Lax)$ to $\mathsf L$, is comonadic. In particular, there is no obvious candidate for its right adjoint. Indeed, the usual construction of `free' comodules does not seem to work: for an arbitrary object $x$ of $\mathsf L$, the comultiplication $\delta$ does not seem to induce a coaction on $x\Lax a$.
\end{example} 

\section{Semigroups in $\mathsf{Oplax}_+$-monoidal categories}
\label{sec:semigroup}

Dually to the previous section, here we introduce semigroups in monoidal categories in which monoidal products of positive number of factors are available together with oplax coherence morphisms.

\begin{definition}
\label{def:oplax+}
An {\em $\mathsf{Oplax}_+$-monoidal category} consists of 
\begin{itemize}
\item a category $\mathsf R$
\item for all positive integers $n$, a functor $\oplax{n}:{\mathsf R}^{n} \to \mathsf R$ 
\item for all positive integers $n,k_1,\dots,k_n$, natural transformations
$$
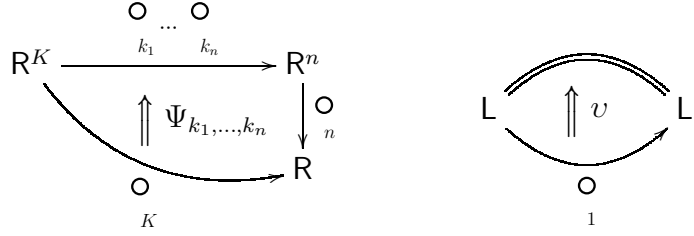
\xymatrix@C=80pt{
{\mathsf R}^K
\ar[r]^-{\oplax {k_1} \ \cdots \ \oplax{k_n}}
\ar@/_1.6pc/[rd]_-{\oplax {K}}
\ar@{}[rd]|-{\qquad \Longuparrow {\Psi_{k_1,\dots,k_n}}} &
{\mathsf R}^n \ar[d]^-{\oplax n}  \\
& {\mathsf R}}
\qquad\qquad
\raisebox{-17pt}{$
\xymatrix@C=60pt{
{\mathsf L} \ar@{=}@/^1.7pc/[r] \ar@/_1.7pc/[r]_-{\oplax 1}
\ar@{}[r]|-{\Longuparrow \upsilon} &
\mathsf L}$}
$$
\end{itemize}
satisfying the so-called coassociativity and counitality axioms encoded in the diagrams which are obtained from the diagrams of Definition \ref{def:lax+} by reversing the arrows.
\end{definition}

Again, Definition \ref{def:oplax+} can be extended in a straightforward way to define $\mathsf{Oplax}_+$-monoids in 
any Gray monoid --- so in particular in any strict monoidal 2-category.

$(\mathsf R,\OpLax,\Psi,\upsilon)$ is an $\mathsf{Oplax}_+$-monoidal category if and only if $(\mathsf R^{\mathsf{op}},\OpLax^{\mathsf{op}},\Psi^{\mathsf{op}},\upsilon^{\mathsf{op}})$ is a $\mathsf{Lax}^+$-monoidal category.

Throughout, in an $\mathsf{Oplax}_+$-monoidal category we denote by $\lw a \rw$ the image of any object $a$ under the functor $\oplax 1: \mathsf R \to \mathsf R$; and we denote by $a_1 \OpLax \cdots \OpLax a_n$ the image of an object $(a_1,\dots,a_n)$ under the functor $\oplax n:\mathsf R^n \to \mathsf R$.

Dually to Proposition \ref{prop:Cartesian+}, the Cartesian product of $\mathsf{Oplax}_+$-monoidal categories carries a canonical $\mathsf{Oplax}_+$-monoidal structure.

\begin{definition}
\label{def:loplax+-functor}
An {\em $\mathsf{Oplax}_+$-monoidal functor} consists of
\begin{itemize}
\item a functor $G:\mathsf R \to \mathsf R'$
\item for all positive integers $n$, a natural transformation
$\oplax n'\cdot (G\cdots G) \to G\cdot \oplax n$
\end{itemize}
such that the diagrams of Definition \ref{def:lax+-functor} with reversed arrows commute.
\end{definition}

Symmetrically to Proposition \ref{prop:comp_functor},
the composites and the Cartesian products of (composable) $\mathsf{Oplax}_+$-monoidal functors are again $\mathsf{Oplax}_+$-monoidal via the evident natural transformations.

\begin{definition}
\label{def:oplax2-nattr} 
An {\em $\mathsf{Oplax}_+$-monoidal natural transformation} is a natural transformation for which the diagrams of Definition \ref{def:lax+-nattr} with reversed arrows commute.
\end{definition}

Symmetrically to Proposition \ref{prop:comp_nattr},
all of the composites, the Godement products, and the Cartesian products of (composable) $\mathsf{Oplax}_+$-monoidal natural transformations are again $\mathsf{Oplax}_+$-monoidal. Consequently, there is a strict monoidal 2-category $\mathsf{Oplax}_+$ of $\mathsf{Oplax}_+$-monoidal categories, $\mathsf{Oplax}_+$-monoidal functors and $\mathsf{Oplax}_+$-monoidal natural transformations.

\begin{definition}
A {\em semigroup} in an $\mathsf{Oplax}_+$-monoidal category $(\mathsf R,\OpLax,\Psi,\upsilon)$ is a cosemigroup in the $\mathsf{Lax}^+$-monoidal category $(\mathsf R^{\mathsf{op}},\OpLax^{\mathsf{op}},\Psi^{\mathsf{op}},\upsilon^{\mathsf{op}})$. That is, it consists of 
\begin{itemize}
\item an object $a$ of $\mathsf R$ and 
\item a morphism $\mu:a \OpLax a \to a$
\end{itemize}
rendering commutative the first (associativity) diagram below.
A {\em morphism of semigroups} is a morphism $\phi$ rendering commutative the second diagram.
$$
\xymatrix{
(a\OpLax a) \OpLax \lw a \rw \ar[d]_-{\mu \circ \upsilon} &
\ar[l]_-{\Psi_{2,1}}a\OpLax a \OpLax a \ar[r]^-{\Psi_{1,2}}&
\lw a \rw \OpLax (a\OpLax a) \ar[d]^-{\upsilon \circ \mu} \\
a\OpLax a  \ar[r]_-\mu &
a  &
a\OpLax a \ar[l]^-\mu   }
\qquad 
\xymatrix{
a \OpLax a \ar[r]^-\mu \ar[d]_-{\phi \circ \phi} &
a \ar[d]^-\phi \\
a' \OpLax a' \ar[r]_-{\mu'} &
a'}
$$
\end{definition}

\begin{theorem}
\label{thm:semigroup}
For any $\mathsf{Oplax}_+$-monoidal category $(\mathsf R,\OpLax,\Psi,\upsilon)$, the following categories are isomorphic.
\begin{itemize}
\item[{(i)}]
The category of semigroups and their morphisms in $(\mathsf R,\OpLax,\Psi,\upsilon)$.
\item[{(ii)}]
The category of $\mathsf{Oplax}_+$-monoidal functors $(\mathbbm 1,1,1,1) \to (\mathsf R,\OpLax,\Psi,\upsilon)$ and their $\mathsf{Oplax}_+$-monoidal natural transformations. 
\end{itemize}
Consequently, $\mathsf{Oplax}_+$-monoidal functors preserve semigroups.
\end{theorem}

\begin{proof}
It follows by the application of Theorem \ref{thm:cosemigroup} to the $\mathsf{Lax}^+$-monoidal category $(\mathsf R^{\mathsf{op}},\OpLax^{\mathsf{op}},\Psi^{\mathsf{op}},\upsilon^{\mathsf{op}})$.
\end{proof}

Dually to comodules in Definition \ref{def:comodule}, one may consider modules over semigroups in the following sense.

\begin{definition} \label{def:module}
A {\em module} of a semigroup $(a,\mu)$ in an $\mathsf{Oplax}_+$-monoidal category $(\mathsf R,\OpLax,\Psi,\upsilon)$ is a comodule of the cosemigroup $(a,\mu)$ in the $\mathsf{Lax}^+$-monoidal category $(\mathsf R^{\mathsf{op}},\OpLax^{\mathsf{op}},\Psi^{\mathsf{op}},\upsilon^{\mathsf{op}})$. Thus it consists of
\begin{itemize}
\item an object $x$ of $\mathsf R$
\item a morphism $\varrho: x\OpLax a \to x$ (the {\em action})
\end{itemize}
such that the (associativity) diagram on the left of 
$$
\xymatrix{
\lw x \rw \OpLax (a \OpLax a) \ar[d]_-{\upsilon \circ \mu}
& \ar[l]_-{\Psi_{1,2}} x \OpLax a \OpLax a \ar[r]^-{\Psi_{2,1}} &
(x \OpLax a) \OpLax \lw a \rw \ar[d]^-{\varrho \circ \upsilon} \\
x \OpLax a \ar[r]_-\varrho &
x &
\ar[l]^\varrho x \OpLax a}
\qquad 
\xymatrix{
x \OpLax a \ar[r]^-\varrho \ar[d]_-{\theta  \circ 1} & 
x \ar[d]^-\theta \\
x' \OpLax a \ar[r]_-{\varrho'} &
x'}
$$
commutes.
A {\em morphism of modules} $(x,\varrho) \to (x',\varrho')$ is a morphism $\theta:x \to x'$ in $\mathsf R$ such that the diagram on the right above commutes.
\end{definition}

\section{Bisemigroups in $\mathsf{Lax}^+\mathsf{Oplax}_+$-duoidal categories}
\label{sec:bisgr}

In order to define bisemigroups in some category, one needs both a $\mathsf{Lax}^+$-monoidal structure --- for the cosemigroup part to live in --- and an $\mathsf{Oplax}_+$-monoidal structure --- to host the semigroup part. Also a suitable compatibility of these monoidal structures is needed so that the category of cosemigroups inherits the $\mathsf{Oplax}_+$-monoidal structure and, dually, the category of semigroups inherits the $\mathsf{Lax}^+$-monoidal structure. Then we can define bisemigroups by the usual principle: as cosemigroups in the category of semigroups; equivalently, as semigroups in the category of cosemigroups. In this section we formulate a suitable compatibility between $\mathsf{Lax}^+$- and $\mathsf{Oplax}_+$-monoidal structures which allows us to carry out this program.

\begin{definition}
\label{def:++duoidal}
A {\em $\mathsf{Lax}^+\mathsf{Oplax}_+$-duoidal category} consists of
\begin{itemize}
\item a category $\mathsf D$
\item a $\mathsf{Lax}^+$-monoidal structure $(\mathsf D,\Lax,\Phi,\iota)$
\item an $\mathsf{Oplax}_+$-monoidal structure $(\mathsf D,\OpLax,\Psi,\upsilon)$
\item for all positive integers $n$ and $p$, natural transformations
$$
\xymatrix@C=60pt{
\mathsf D^{np} \ar[r]^-{\tau_np} 
\ar[d]_-{\oplax n \cdots \oplax n}
\ar@{}[rrd]|-{\Longdownarrow {\xi_n^p}} &
\mathsf D^{pn} \ar[r]^-{\lax p \cdots \lax p} &
\mathsf D^{n} \ar[d]^-{\oplax n} \\
\mathsf D^{p} \ar[rr]_-{\lax p} &&
\mathsf D}
$$
\end{itemize}
satisfying the following equivalent compatibility conditions, for all positive integers $n,p,k_1,\dots,k_p$.
\begin{itemize}
\item[{(i)}] 
\begin{itemize}
\item[$\bullet$]
$(\oplax n,\xi_n)$ is a $\mathsf{Lax}^+$-monoidal functor $(\mathsf D,\Lax,\Phi,\iota)^n \to(\mathsf D,\Lax,\Phi,\iota);$ 
\item[$\bullet$]
$\Psi_{k_1,\dots,k_p}$ and $\upsilon$ are $\mathsf{Lax}^+$-monoidal natural transformations.
\end{itemize}
\noindent
Succinctly, $((\mathsf D,\Lax,\Phi,\iota),(\OpLax,\xi),\Psi,\upsilon)$ is an $\mathsf{Oplax}_+$-monoid in the strict monoidal 2-category $\mathsf{Lax}^+$ of Proposition \ref{prop:comp_nattr}.
\item[{(ii)}]
\begin{itemize}
\item[$\bullet$]
$(\lax n,\xi^n)$ is an $\mathsf{Oplax}_+$-monoidal functor $(\mathsf D,\OpLax,\Psi,\upsilon)^n \to (\mathsf D,\OpLax,\Psi,\upsilon);$ 
\item[$\bullet$]
$\Phi_{k_1,\dots,k_p}$ and $\iota$ are $\mathsf{Oplax}_+$-monoidal natural transformations.
\end{itemize}
\noindent
Succinctly, $((\mathsf D,\OpLax,\Psi,\upsilon),(\Lax,\xi),\Phi,\iota)$ is a $\mathsf{Lax}^+$-monoid in the strict monoidal 2-category $\mathsf{Oplax}_+$ of Section \ref{sec:semigroup}.
\item[{(iii)}] The following diagrams commute.
$$
\xymatrix@C=-25pt @R=19pt{
\oplax n \cdot (\lax p \cdots \lax p) \cdot
(\raisebox{-1.5pt}{$\lax {k_1} \cdots \lax {k_p}\cdots \lax {k_1} \cdots \lax {k_p}$}) \cdot \tau_{nK}
\ar@{=}[d]
\ar[r]^-{\raisebox{8pt}{${}_
{1\cdot (\Phi_{k_1,\dots,k_p}\cdots \Phi_{k_1,\dots,k_p})\cdot 1}$}} &
\oplax n \cdot(\raisebox{-1.5pt}{$\lax {K} \cdots \lax {K}$} ) 
\cdot \tau_{nK}  \ar[ddd]^-{\xi^{K}_n} \\
\oplax n \cdot (\lax p \cdots \lax p) \cdot \tau_{np} \cdot
(\raisebox{-1.5pt}{$\lax {k_1} \cdots \lax {k_1} \cdots \lax {k_p} \cdots \lax {k_p}$}) \cdot (\tau_{n k_1} \cdots \tau_{n k_p}) 
\ar[d]_-{\xi^p_n \cdot 1 \cdot 1} \\
\lax p \cdot (\oplax n \cdots \oplax n) \cdot 
(\raisebox{-1.5pt}{$\lax {k_1} \cdots \lax {k_1} \cdots \lax {k_p} \cdots \lax {k_p}$}) \cdot (\tau_{n k_1} \cdots \tau_{n k_p}) 
\ar[d]_-{1\cdot (\xi^{k_1}_n \cdots \xi^{k_p}_n)}  \\
\lax p \cdot (\raisebox{-1.5pt}{$\lax {k_1} \cdots \lax {k_p}$}) 
\cdot (\oplax n \cdots \oplax n) 
\ar[r]_-{\Phi_{k_1,\dots ,k_p}\cdot 1} &
\raisebox{-1.5pt}{$\lax{K}$} \cdot 
(\oplax n \cdots \oplax n)}
\qquad
\xymatrix@R=113pt{
\oplax n
\ar[r]^-{1 \cdot (\iota \cdots \iota)} \ar[rd]_-{\iota\cdot 1} &
\oplax n \cdot (\raisebox{-1.5pt}{$\lax 1 \cdots \lax 1$})  \ar[d]^-{\xi^1_n} \\
& \raisebox{-1.5pt}{$\lax 1$}  \cdot \oplax n}
$$
$$
\xymatrix@R=19pt{
\lax n \cdot (\oplax p \cdots \oplax p) \cdot
(\raisebox{-1.5pt}{$\oplax {k_1} \cdots \oplax {k_p} \cdots \oplax {k_1} \cdots \oplax {k_p}$}) &
\lax n \cdot(\raisebox{-1.5pt}{$\oplax {K} \cdots \oplax {K}$} ) 
\ar[l]_-{\raisebox{8pt}{${}_{1\cdot (\Psi_{k_1,\dots,k_p}\cdots \Psi_{k_1,\dots,k_p})}$}}  \\
\ar[u]^-{\xi^n_p\cdot 1}
\oplax p \cdot (\lax n \cdots \lax n) \cdot \tau_{pn} \cdot
(\raisebox{-1.5pt}{$\oplax {k_1} \cdots \oplax {k_p} \cdots \oplax {k_1} \cdots \oplax {k_p}$}) 
\ar@{=}[d] \\
\oplax p \cdot (\lax n \cdots \lax n)  \cdot
(\raisebox{-1.5pt}{$\oplax {k_1} \cdots \oplax {k_1} \cdots \oplax {k_p} \cdots \oplax {k_p}$}) \cdot \tau_{pn}  \\
\ar[u]^-{1\cdot (\xi_{k_1}^n \cdots \xi_{k_p}^n)\cdot 1}
\oplax p \cdot (\raisebox{-1.5pt}{$\oplax {k_1} \cdots \oplax {k_p}$}) 
\cdot (\lax n \cdots \lax n) \cdot 
(\tau_{k_1 n}\dots \tau_{k_p n})\cdot \tau_{p n}   &
\ar[uuu]_-{\xi^n_K\cdot 1}
\raisebox{-1.5pt}{$\oplax{K}$}\cdot 
(\lax n \cdots \lax n) \cdot \tau_{Kn} 
\ar[l]^-{\raisebox{-14pt}{${}_{\Psi_{k_1,\dots, k_p}\cdot 1 \cdot 1}$}}}
\qquad
\xymatrix@R=113pt@C=30pt{
\raisebox{-1.5pt}{$\oplax 1$} \cdot \lax n
\ar[d]_-{\xi^n_1} \ar[rd]^-{\upsilon \cdot 1} \\
\lax n \cdot (\raisebox{-1.5pt}{$\oplax 1 \cdots \oplax 1$})
\ar[r]_-{1 \cdot (\upsilon \cdots \upsilon)} &
\lax n}
$$
\end{itemize}
\end{definition}

\begin{definition} \label{def:Lax+Oplax+functor}
A {\em $\mathsf{Lax}^+\mathsf{Oplax}_+$-duoidal functor} $(\mathsf D,\Lax,\OpLax) \to (\mathsf D',\Lax',\OpLax')$ is defined as
\begin{itemize}
\item[{(i)}] an $\mathsf{Oplax}_+$-monoidal 1-cell 
$((G,\Gamma^\bullet),\Gamma^\circ):((\mathsf D,\Lax),\OpLax) \to ((\mathsf D',\Lax'),\OpLax')$
in the strict monoidal 2-category $\mathsf{Lax}^+$ of Proposition \ref{prop:comp_nattr}; equivalently,
\item[{(ii)}] a $\mathsf{Lax}^+$-monoidal 1-cell 
$((G,\Gamma^\circ),\Gamma^\bullet):((\mathsf D,\OpLax),\Lax) \to ((\mathsf D',\OpLax'),\Lax')$ in the strict monoidal 2-category $\mathsf{Oplax}_+$ of Section \ref{sec:semigroup}; equivalently,
\item[{(iii)}]
\begin{itemize}
\item[{$\bullet$}] a functor $G: \mathsf D \to \mathsf D'$
\item[{$\bullet$}] a $\mathsf{Lax}^+$-monoidal structure $\{\Gamma^\bullet_n:G\cdot \lax n \to \lax n'\cdot (G \cdots G)\}_{n\in \mathbb N}$
\item[{$\bullet$}] a $\mathsf{Oplax}_+$-monoidal structure $\{\Gamma^\circ_n: \oplax p'\cdot (G \cdots G) \to G\cdot \oplax p \}_{p \in \mathbb N}$
\end{itemize}

\noindent
rendering commutative the following diagram for all positive integers $n,p$.
\begin{equation} \label{eq:Lax+Oplax+functor}
\xymatrix@R=4pt{
\oplax p' \cdot (G \cdots G) \cdot (\lax n \cdots \lax n) \cdot \tau_{pn}
\ar[r]^-{\raisebox{8pt}{${}_{
1\cdot (\Gamma^\bullet_n \cdots \Gamma^\bullet_n)\cdot 1}$}}
\ar[dddd]_-{\Gamma^\circ_p \cdot 1\cdot 1} &
\oplax p' \cdot (\lax n' \cdots \lax n') \cdot (G \cdots G) \cdot \tau_{pn}
\ar@{=}[d] \\
& \oplax p' \cdot (\lax n' \cdots \lax n') \cdot \tau_{pn} \cdot (G \cdots G)
\ar[r]^-{\xi^{\prime n}_p\cdot 1} &
\lax n' \cdot ( \oplax p' \cdots  \oplax p') \cdot (G \cdots G)
\ar[ddd]^-{1\cdot (\Gamma^\circ_p \cdots \Gamma^\circ_p)} \\
\\
\\
G \cdot \oplax p \cdot (\lax n \cdots \lax n) \cdot \tau_{pn}
\ar[r]_-{1\cdot \xi^n_p} &
G \cdot \lax n \cdot ( \oplax p \cdots  \oplax p)
\ar[r]_-{\Gamma^\bullet_n \cdot 1} &
\lax n' \cdot (G \cdots G) \cdot ( \oplax p \cdots  \oplax p)}
\end{equation}
\end{itemize}
\end{definition}

It follows immediately from Definition \ref{def:Lax+Oplax+functor} and Proposition \ref{prop:comp_functor}~(1) that the composite of $\mathsf{Lax}^+\mathsf{Oplax}_+$-duoidal functors is again $\mathsf{Lax}^+\mathsf{Oplax}_+$-duoidal via the composite $\mathsf{Lax}^+$-monoidal structure in Proposition \ref{prop:comp_functor}~(1) and the symmetrically constructed composite $\mathsf{Oplax}_+$-monoidal structure.

\begin{definition} \label{def:Lax+Oplax+nattr}
A {\em $\mathsf{Lax}^+\mathsf{Oplax}_+$-duoidal natural transformation} is defined as 
\begin{itemize}
\item[{(i)}] an $\mathsf{Oplax}_+$-monoidal 2-cell in the strict monoidal 2-category $\mathsf{Lax}^+$ of Proposition \ref{prop:comp_nattr}; equivalently,
\item[{(ii)}] a $\mathsf{Lax}^+$-monoidal 2-cell in the strict monoidal 2-category $\mathsf{Oplax}_+$ of Section \ref{sec:semigroup}; equivalently,
\item[{(iii)}] a natural transformation which is both $\mathsf{Oplax}_+$-monoidal and $\mathsf{Lax}^+$-monoidal. That is, it renders commutative both the diagram of Definition \ref{def:lax+-nattr} and its symmetric counterpart with reversed arrows and inverted colors.
\end{itemize}
\end{definition}

\begin{theorem}
\label{thm:duoidal(co)semigroup}
For any $\mathsf{Lax}^+\mathsf{Oplax}_+$-duoidal category $(\mathsf D,(\Lax,\Phi,\iota),(\OpLax,\Psi,\upsilon),\xi)$, the following assertions hold.
\begin{itemize}
\item[{(1)}] The category of cosemigroups in the $\mathsf{Lax}^+$-monoidal category $(\mathsf D,\Lax,\Phi,\iota)$ is $\mathsf{Oplax}_+$-monoidal via the structure $(\OpLax,\Psi,\upsilon)$.
\item[{(2)}] The category of semigroups in the $\mathsf{Oplax}_+$-monoidal category $(\mathsf D,\OpLax,\Psi,\upsilon)$ is $\mathsf{Lax}^+$-monoidal via the structure $(\Lax,\Phi,\iota)$.
\end{itemize}
\end{theorem}

\begin{proof}
We only prove (1), part (2) follows symmetrically.

By Theorem \ref{thm:cosemigroup} and Definition \ref{def:++duoidal}~{(i)}, 
any sequence $\{a_1,\cdots, a_n\}$ of cosemigroups in $(\mathsf D,\Lax,\Phi,\iota)$ determines a $\mathsf{Lax}^+$-monoidal functor; and any sequence 
$\{f_1:a_1\to a'_1,\cdots, f_n:a_n\to a'_n\}$ of cosemigroup morphisms determines a $\mathsf{Lax}^+$-monoidal natural transformation
$$
\xymatrix{
\mathbbm 1
\ar@/^1.3pc/[rrr]^-{a_1\cdots a_n} 
\ar@{}[rrr]|-{\Longdownarrow {f_1\cdots f_n} }
\ar@/_1.3pc/[rrr]_-{a'_1\cdots a'_n}  &&&
\mathsf D^n \ar[r]^-{\oplax n} &
\mathsf D.}
$$
That is, $\oplax n$ determines a functor to the category of cosemigroups in $(\mathsf D,\Lax,\Phi,\iota)$ from the $n^{\textrm{th}}$ Cartesian power of this category.

Again by Theorem \ref{thm:cosemigroup} and Definition \ref{def:++duoidal}~{(i)}, the natural transformations $\Psi$ and $\upsilon$ --- if evaluated at cosemigroups --- yield morphisms of cosemigroups. Thus they can be seen as natural transformations between functors connecting Cartesian powers of categories of cosemigroups. They satisfy the axioms in Definition \ref{def:oplax+} by construction.
\end{proof}

\begin{definition}
\label{def:bimonoid}
A {\em bisemigroup} in a $\mathsf{Lax}^+\mathsf{Oplax}_+$-duoidal category $(\mathsf D,\Lax,\OpLax, \xi)$ is 
\begin{itemize}
\item[{(i)}] a semigroup in the $\mathsf{Oplax}_+$-monoidal category of cosemigroups in $(\mathsf D,\Lax)$ --- see Theorem \ref{thm:duoidal(co)semigroup}~(1); equivalently,
\item[{(ii)}] a cosemigroup in the $\mathsf{Lax}^+$-monoidal category of semigroups in $(\mathsf D,\OpLax)$ --- see Theorem \ref{thm:duoidal(co)semigroup}~(2); equivalently,
\item[{(iii)}]
\begin{itemize}
\item[$\bullet$] an object $a$
\item[$\bullet$] a semigroup $(a,\mu)$ in $(\mathsf D,\OpLax)$
\item[$\bullet$] a cosemigroup $(a,\delta)$ in $(\mathsf D,\Lax)$
\end{itemize}
subject to the compatibility condition encoded in the commutative diagram
$$
\xymatrix{
a\OpLax a \ar[r]^-\mu \ar[d]_-{\delta \circ \delta} &
a \ar[r]^-\delta & 
a \Lax a. \\
(a\Lax a) \OpLax (a \Lax a) \ar[rr]_-{\xi^2_2} &&
(a \OpLax a) \Lax (a \OpLax a). \ar[u]_-{\mu\bullet \mu}}
$$
\end{itemize}
A {\em morphism of bisemigroups} is a semigroup morphism in the category of cosemigroups in $(\mathsf D,\Lax)$; equivalently, a cosemigroup morphism in the category of semigroups in $(\mathsf D,\OpLax)$; equivalently, a morphism in $\mathsf D$ which is both a cosemigroup morphism in $(\mathsf D,\Lax)$ and a semigroup morphism in $(\mathsf D,\OpLax)$.
\end{definition}

\begin{theorem} \label{thm:cosemigroup_cat}
For any $\mathsf{Lax}^+\mathsf{Oplax}_+$-duoidal category $(\mathsf D,\Lax,\OpLax,\xi)$ the following categories are isomorphic.
\begin{itemize}
\item[{(i)}] The category of bisemigroups and their morphisms in $(\mathsf D,\Lax,\OpLax,\xi)$.
\item[{(ii)}] The category of $\mathsf{Lax}^+\mathsf{Oplax}_+$-duoidal functors $(\mathbbm 1,1,1,1)\to (\mathsf D,\Lax,\OpLax,\xi)$ and their $\mathsf{Lax}^+\mathsf{Oplax}_+$-duoidal natural transformations.
\end{itemize}
Consequently, $\mathsf{Lax}^+\mathsf{Oplax}_+$-duoidal functors preserve bisemigroups.
\end{theorem}

\begin{proof}
By Definition \ref{def:bimonoid}, the category of part (i) is isomorphic to the category of semigroups in the category of cosemigroups in $(\mathsf D,\Lax)$.
So by Theorem \ref{thm:semigroup}, it is isomorphic to the category of $\mathsf{Oplax}_+$-monoidal functors and $\mathsf{Oplax}_+$-monoidal natural transformations from $\mathbbm 1$ to the category of cosemigroups in $(\mathsf D,\Lax)$.
Then by Theorem \ref{thm:cosemigroup}, it is further equivalent to the category of $\mathsf{Oplax}_+$-monoidal 1-cells and $\mathsf{Oplax}_+$-monoidal 2-cells from $(\mathbbm 1,1)$ to $(\mathsf D,\Lax)$ in $\mathsf{Lax}^+$.
By Definition \ref{def:Lax+Oplax+functor} and Definition \ref{def:Lax+Oplax+nattr} this is isomorphic to the category of part (ii).
\end{proof}

\begin{theorem} \label{thm:mod_monoidal}
For any bisemigroup $(a,\delta,\mu)$ in an arbitrary $\mathsf{Lax}^+\mathsf{Oplax}_+$-duoidal category 
$(\mathsf D,(\Lax,\Phi,\iota),(\OpLax,\Psi,\upsilon),\xi)$
the following assertions hold.
\begin{itemize}
\item[{(1)}] The category of modules --- in the sense of Definition \ref{def:module} --- of the semigroup $(a,\mu)$ in the $\mathsf{Oplax}_+$-monoidal category $(\mathsf D,\OpLax,\Psi,\upsilon)$ admits the $\mathsf{Lax}^+$-monoidal structure $(\Lax,\Phi,\iota)$.
\item[{(2)}] The category of comodules --- in the sense of Definition \ref{def:comodule} --- of the cosemigroup $(a,\delta)$ in the $\mathsf{Lax}^+$-monoidal category $(\mathsf D,\Lax,\Phi,\iota)$ admits the $\mathsf{Oplax}_+$-monoidal structure $(\OpLax,\Psi,\upsilon)$.
\end{itemize}
\end{theorem}

\begin{proof}
We only prove part (1), part (2) is verified symmetrically.

For any sequence of $(a,\mu)$-modules $\{(x_i,\varrho_i)\}_{i=1,\dots,n}$, there is an $(a,\mu)$-module 
$$
\xymatrix@C=20pt{
(x_1 \Lax \cdots \Lax x_n) \OpLax a
\ar[r]^-{1\circ \delta_n} &
(x_1 \Lax \cdots \Lax x_n) \OpLax a^{\bullet n}
\ar[r]^-{\xi^n_2} &
(x_1 \OpLax a) \Lax \cdots \Lax (x_n \OpLax a)
\ar[rr]^-{\varrho_1 \bullet \cdots \bullet \varrho_n} &&
x_1 \Lax \cdots \Lax x_n.}
$$
Indeed, one checks by induction on $l$ that the $l$-fold comultiplication of \eqref{eq:delta_k+1} for $a\circ a$ comes out as
$\xymatrix{
a\circ a \ar[r]^-{\delta_l \circ \delta_l} &
a^{\bullet l} \circ a^{\bullet l} \ar[r]^-{\xi^l_2} &
(a\circ a)^{\bullet l}}$.
The induction step uses the top right diagram in part (iii) of Definition \ref{def:++duoidal} for $n=2$, and the top left diagram of Definition \ref{def:++duoidal}~(iii) for $n=2$, $p=2$, $k_1=1$ and $k_2=l$. 
Since by Definition \ref{def:bimonoid} the multiplication $\mu$ is a morphism of cosemigroups, we obtain the commutative diagram
$$
\xymatrix{
a\circ a \ar[d]_-\mu \ar[r]^-{\delta_n \circ \delta_n} &
a^{\bullet n} \circ a^{\bullet n} \ar[r]^-{\xi^n_2} &
(a\circ a)^{\bullet n} \ar[d]^-{\mu^{\bullet n}} \\
a \ar[rr]_-{\delta_n} &&
a^{\bullet n}}
$$
for all non-negative integers $n$. Together with the bottom right diagram of Definition \ref{def:++duoidal}~(iii), this proves the commutativity of the leftmost region of Figure \ref{fig:module_left}.
The region labelled by (BR) in Figure \ref{fig:module_right} also commutes by the bottom right diagram of Definition \ref{def:++duoidal}~(iii). 
The regions marked by (BL) in Figure \ref{fig:module_left} and Figure \ref{fig:module_right} both commute by the bottom left diagram of Definition \ref{def:++duoidal}~(iii) at the respective values $(n, p=2,k_1=1,k_2=2)$ and $(n, p=2,k_1=2,k_2=1)$.
The remaining regions commute by evident naturality.
Whenever each action $\varrho_i$ is associative, the right verticals of the diagrams in Figure \ref{fig:module_left} and Figure \ref{fig:module_right} are equal. Hence also the paths on their left hand sides are equal, proving the associativity of the stated action on $x_1\Lax \cdots \Lax x_n$.
\begin{amssidewaysfigure}
\thisfloatpagestyle{empty}
\centering
\xymatrix@C=65pt@R=40pt{
\lw x_1 \Lax \cdots \Lax x_n \rw \OpLax (a\OpLax a)
\ar[rd]^-{1\circ (\delta_n \circ \delta_n)} 
\ar[dddd]_-{\upsilon \circ \mu} &&
(x_1 \Lax \cdots \Lax x_n) \OpLax  a \OpLax  a
\ar[d]^-{1\circ \delta_n \circ \delta_n}
\ar[ll]_-{\Psi_{1,2}}  \\
& \lw x_1 \Lax \cdots \Lax x_n \rw \OpLax  (a^{\bullet n} \OpLax  a^{\bullet n})
\ar[dd]^-{\xi^n_1 \circ \xi^n_2} \ar@{}[rdd]|-{\textrm{(BL)}} &
( x_1 \Lax \cdots \Lax x_n ) \OpLax  a^{\bullet n}\OpLax a^{\bullet n}
\ar[d]^-{\xi^n_3}
\ar[l]_-{\Psi_{1,2}}  \\
&& (x_1 \OpLax  a \OpLax  a) \Lax \cdots \Lax (x_n \OpLax  a \OpLax  a)
\ar[d]^-{\Psi_{1,2} \Lax \cdots \Lax \Psi_{1,2}} \\
& (\lw x_1 \rw \Lax \cdots \Lax \lw x_n \rw) \OpLax (a \OpLax  a)^{\bullet n}
\ar[r]^-{\xi^n_2}
\ar[d]^-{(\upsilon \Lax \cdots \Lax \upsilon) \circ \mu^{\bullet n}} &
(\lw x_1 \rw \OpLax (a\OpLax a))  \Lax \cdots \Lax (\lw x_n \rw \OpLax  (a\OpLax  a))
\ar[d]^-{(\upsilon \circ \mu) \Lax \cdots \Lax (\upsilon \circ \mu)} \\
(x_1 \Lax \cdots \Lax x_n) \OpLax a
\ar[r]_-{1\circ \delta_n} &
(x_1 \Lax \cdots \Lax x_n) \OpLax a^{\bullet n}
\ar[r]_-{\xi^n_2} &
(x_1 \OpLax a) \Lax \cdots \Lax (x_n \OpLax a)
\ar[d]^-{\varrho_1 \Lax \cdots \Lax \varrho_n} \\
&& x_1 \Lax \cdots \Lax x_n }
\caption{Associativity of the action on \protect\scalebox{1.6}{\raisebox{-1pt}{$\bullet$}}-products of modules --- first part}
\label{fig:module_left}
\end{amssidewaysfigure}
\begin{amssidewaysfigure}
\thisfloatpagestyle{empty}
\centering
\xymatrix@C=20pt@R=40pt{
((x_1 \Lax \cdots \Lax x_n) \OpLax a)\OpLax  \lw a \rw 
\ar[d]_-{(1\circ \delta_n) \circ 1} &&&
(x_1 \Lax \cdots \Lax x_n) \OpLax  a \OpLax  a
\ar[d]^-{1\circ \delta_n \circ \delta_n}
\ar[lll]_-{\Psi_{2,1}} \\
(( x_1 \Lax \cdots \Lax x_n ) \OpLax a^{\bullet n}) \OpLax  \lw a \rw
\ar[r]^-{1\circ \lw \delta_n \rw}
\ar[dd]_-{\xi^n_2 \circ 1} &
(( x_1 \Lax \cdots \Lax x_n ) \OpLax  a^{\bullet n}) \OpLax  \lw a^{\bullet n} \rw 
\ar[dd]^-{\xi^n_2 \circ 1} &&
( x_1 \Lax \cdots \Lax x_n ) \OpLax  a^{\bullet n}\OpLax a^{\bullet n}
\ar[d]^-{\xi^n_3}
\ar[ll]_-{\Psi_{2,1}}  
\ar@{}[lldd]|-{\textrm{(BL)}} \\
&&& (x_1 \OpLax  a \OpLax  a) \Lax \cdots \Lax (x_n \OpLax  a \OpLax  a)
\ar[d]^-{\Psi_{2,1} \Lax \cdots \Lax \Psi_{2,1}} \\
((x_1 \OpLax a) \Lax \cdots \Lax  (x_n \OpLax a))\OpLax \lw a \rw
\ar[r]^-{1\circ \lw \delta_n \rw}
\ar[d]_-{(\varrho_1 \Lax \cdots \Lax \varrho_n) \circ \upsilon} &
((x_1 \OpLax  a) \Lax \cdots \Lax  (x_n \OpLax a))\OpLax \lw a^{\bullet n} \rw
\ar[d]^-{(\varrho_1 \Lax \cdots \Lax \varrho_n ) \circ \upsilon}
\ar[r]^-{1\circ \xi^n_1} &
((x_1 \OpLax  a) \Lax \cdots \Lax  (x_n \OpLax  a))\OpLax  \lw a \rw^{\bullet n}
\ar[r]^-{\xi^n_2}
\ar[d]^-{(\varrho_1 \Lax \cdots \Lax \varrho_n ) \circ \upsilon^{\bullet n}}
\ar@{}[ld]|-{\quad\textrm{(BR)}} &
((x_1 \OpLax  a) \OpLax  \lw a \rw) \Lax \cdots \Lax ((x_n \OpLax a) \OpLax  \lw a \rw) 
\ar[d]^-{(\varrho_1 \circ \upsilon) \Lax \cdots \Lax (\varrho_n \circ \upsilon)} \\
(x_1 \Lax \cdots \Lax x_n) \OpLax a
\ar[r]_-{1\circ \delta_n} &
(x_1 \Lax \cdots \Lax x_n) \OpLax a^{\bullet n}
\ar@{=}[r] &
(x_1 \Lax \cdots \Lax x_n) \OpLax a^{\bullet n}
\ar[r]_-{\xi^n_2} &
(x_1 \OpLax a) \Lax \cdots \Lax (x_n \OpLax a)
\ar[d]^-{\varrho_1 \Lax \cdots \Lax \varrho_n} \\
&&& x_1 \Lax \cdots \Lax x_n}
\caption{Associativity of the action on \protect\scalebox{1.6}{\raisebox{-1pt}{$\bullet$}}-products of modules --- second part}
\label{fig:module_right}
\end{amssidewaysfigure}

The $\Lax$-monoidal products of module morphisms are easily seen to be morphisms of modules.
The natural transformation $\iota$ --- if evaluated at a module --- is a morphism of modules by the top right diagram of Definition \ref{def:++duoidal}~(iii).
In order to see that for all non-negative integers $p,k_1,\dots k_p$, the component of
$\Phi_{k_1,\dots ,k_p}$ at any modules is a morphism of modules, use \eqref{eq:n-coass} and the top left diagram of Definition \ref{def:++duoidal}~(iii) for $n=2$.
\end{proof}

\begin{proposition} \label{prop:Hopfmod}
For any bisemigroup $(a,\delta,\mu)$ in an arbitrary $\mathsf{Lax}^+\mathsf{Oplax}_+$-duoidal category the following assertions hold.
\begin{itemize}
\item[{(1)}] $((a,\delta),\mu)$ is a semigroup in the $\mathsf{Oplax}_+$-monoidal category of comodules over the cosemigroup $(a,\delta)$ (cf. Theorem \ref{thm:mod_monoidal}~(2)).
\item[{(2)}] $((a,\mu),\delta)$ is a cosemigroup in the $\mathsf{Lax}^+$-monoidal category of modules over the semigroup $(a,\mu)$ (cf. Theorem \ref{thm:mod_monoidal}~(1)).
\end{itemize}
\end{proposition}

\begin{proof}
We only prove part (1), part (2) follows symmetrically.

By Example \ref{ex:reg_comod}, $(a,\delta)$ is an $(a,\delta)$-comodule. 
The compatibility diagram of Definition \ref{def:bimonoid}~(iii) can be interpreted as $\mu$ being a morphism of $(a,\delta)$-comodules.
Associativity of $\mu$ in the category of $(a,\delta)$-comodules follows from the associativity of the monoid $(a,\mu)$.
\end{proof}

\begin{definition} \label{def:Hopf_module}
A {\em Hopf module} over a bisemigroup $(a,\delta,\mu)$ in a $\mathsf{Lax}^+\mathsf{Oplax}_+$-duoidal category $(\mathsf D,\Lax,\OpLax,\xi)$ is defined by the following equivalent data.
\begin{itemize} 
\item[{(i)}] A module over the semigroup $((a,\delta),\mu)$ in the category of comodules over the cosemigroup $(a,\delta)$ in $(\mathsf D,\Lax)$.
\item[{(ii)}] A comodule over the cosemigroup $((a,\mu),\delta)$ in the category of modules over the semigroup $(a,\mu)$ in $(\mathsf D,\OpLax)$.
\item[{(iii)}]
\begin{itemize}
\item[$\bullet$] An object $x$ of $\mathsf D$,
\item[$\bullet$] a module $(x,\nu)$ over the semigroup $(a,\mu)$ in $(\mathsf D,\OpLax)$,
\item[$\bullet$] a comodule $(x,\varrho)$ over the cosemigroup $(a,\delta)$ in $(\mathsf D,\Lax)$,
\end{itemize}
rendering commutative the following diagram.
$$
\xymatrix{
x\OpLax a \ar[r]^-\nu \ar[d]_-{\varrho \circ \delta} &
x \ar[r]^-\varrho &
x \Lax a \\
(x \Lax a) \OpLax (a\Lax a)  \ar[rr]_-{\xi^2_2} &&
(x\OpLax a) \Lax (a \OpLax a) \ar[u]_-{\nu \bullet \mu}}
$$
\end{itemize}
\end{definition}

\begin{example}
By Example \ref{ex:reg_comod} and its dual counterpart, and by the compatibility diagram of Definition \ref{def:bimonoid}~(iii), $(a,\delta,\mu)$ is a Hopf module over 
an arbitrary bisemigroup $(a,\delta,\mu)$ in a $\mathsf{Lax}^+\mathsf{Oplax}_+$-duoidal category.
\end{example}

\section{Comonoids in $\mathsf{Lax}^+\mathsf{Oplax}^0$-monoidal categories}

For the definition of monoids and comonoids --- that is, to define units for multiplications and counits for comultiplications --- we also need nullary monoidal products. They may be related to the higher monoidal products in possibly different ways. There are well-known and well-studied structures called lax and oplax monoidal categories, see e.g. \cite{DayStreet:lax}. However, the gadgets motivating this paper, namely, unital \BiHom-monoids and counital \BiHom-comonoids, do not seem to fit these well-studied situations. In this section we show that counital \BiHom-comonoids can be described as comonoids in monoidal categories in which the monoidal products of positive numbers of factors come with lax compatibiliy morphisms, but the nullary part is oplax coherent. The dual situation, suitable to describe unital \BiHom-monoids as monoids, will be discussed in the next section.

\begin{definition}
\label{def:lax+oplax0}
A {\em $\mathsf{Lax}^+\mathsf{Oplax}^0$-monoidal category} consists of 
\begin{itemize}
\item a category $\mathsf L$
\item for any non-negative integer $n$, a functor $\lax{n}:{
\mathsf L}^{n} \to \mathsf L$ 
\item for all non-negative integers $n,k_1,\dots,k_n$, and for the functors 
$$
\xymatrix@C=17pt{\mathsf L^{k_i} \ar[r]^-{[k_i]} & \mathsf L^{\overline k_i}}
:=\left\{
\arraycolsep=1pt\def\arraystretch{.5}
\begin{array}{ll}
\xymatrix@C=17pt{\mathsf L^{k_i} \ar[r]^-1 & \mathsf L^{k_i}}
& \textrm{if} \quad k_i>0 \\
\xymatrix{\mathbbm 1 \ar[r]^-{\scalebox{.7}{$\lax 0$}} & \mathsf L}
& \textrm{if} \quad k_i=0,
\end{array}
\right.
$$
and using the notation of Section \ref{sec:preli}, natural transformations
$$
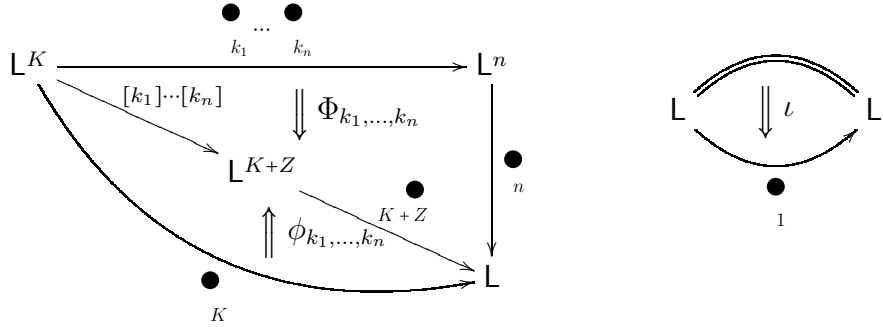
\xymatrix@C=60pt{
{\mathsf L}^K
\ar[rr]^-{\lax {k_1} \ \cdots \ \lax{k_n}}
\ar[rd]^-{\ [k_1]\cdots [k_n]} 
\ar@/_3pc/[rrdd]_-{\lax K} &
\ar@{}[d]|-{\qquad\qquad\qquad  \Longdownarrow {\Phi_{k_1,\dots,k_n}}} &
{\mathsf L}^n \ar[dd]^-{\lax n}  \\
& {\mathsf L}^{K+Z}
\ar@{}[d]|-{\qquad\qquad  \Longuparrow {\phi_{k_1,\dots,k_n}}}
\ar[rd]^-{\lax {K+Z}}  & \\ 
&& {\mathsf L}}
\qquad\qquad
\raisebox{-17pt}{$
\xymatrix@C=60pt{
{\mathsf L} \ar@{=}@/^1.7pc/[r] \ar@/_1.7pc/[r]_-{\lax 1}
\ar@{}[r]|-{\Longdownarrow \iota} &
\mathsf L}$}
$$
\end{itemize}
such that
$$
\Phi_{k_1,\dots,k_n}=1\quad \textrm{if}\quad K=0,
\qquad \qquad
\phi_{k_1,\dots,k_n}=1\quad \textrm{if}\quad Z=0,
\qquad \qquad
\phi_{0}=\iota\cdot 1,
$$
and the diagram
\begin{amssidewaysfigure}
\centering
\xymatrix@C=70pt@R=40pt{
\lax n \cdot (\lax{m_1} \cdots \lax{m_n})\cdot
(\raisebox{-1.5pt}{$\lax{k_{11}}\cdots \lax{k_{nm_n}}$})
\ar[r]^-{\Phi_{m_1,\dots, m_n}\cdot 1}
\ar[ddd]^-{1\cdot (
\Phi_{k_{11},\dots,k_{1m_1}}\cdots
\Phi_{k_{n1},\dots,k_{nm_n}})} &
\lax{M+\sum_{i=1}^n Z(m_i)} \cdot
([m_1]\cdots [m_n])\cdot
(\lax{k_{11}}\cdots \lax{k_{nm_n}})
\ar@{=}[d] &
\ar[l]_-{\phi_{m_1,\dots, m_n}\cdot 1}
\lax M \cdot (\lax{k_{11}}\cdots \lax{k_{nm_n}}) 
\ar[ddd]_-{\Phi_{k_{11},\dots,k_{nm_n}}}
\\
&  \lax{M+\sum_{i=1}^n Z(m_i)} \cdot
(\lax{\widetilde k_{11}} \cdots \lax{\widetilde k_{n {\overline{m}_n}}})
\ar[d]^-{\Phi_{\widetilde k_{11},\dots,\widetilde k_{n{\overline{m}_n}}}} \\
& \lax{K+\widetilde Z} \cdot
([\widetilde k_{11}] \cdots [\widetilde k_{n{\overline{m}_n}}])
\ar@{=}[d]^-{\mathrm{Lemma}~\ref{lem:[F]}} \\
\lax n \cdot 
(\raisebox{-1.5pt}{$\lax{K_1+Z_1} \cdots \lax{K_n+Z_n}$}) \cdot
([k_{11}] \cdots [k_{nm_n}]) 
\ar[r]^-{\Phi_{K_1+Z_1,\dots,K_n+Z_n}\cdot 1} &
\lax{K+\widetilde Z} \cdot 
([K_1+Z_1] \cdots [K_n+Z_n]) \cdot 
([k_{11}] \cdots [k_{nm_n}]) 
\ar@{=}[d]^-{\mathrm{Lemma}~\ref{lem:[F]}} & 
\ar[l]_-{\phi_{K_1+Z_1,\dots,K_n+Z_n}\cdot 1}
\lax{K+Z} \cdot ([k_{11}] \cdots [k_{nm_n}]) \\
& \lax{K+\widetilde Z} \cdot 
([\widehat k_{11}] \cdots [\widehat k_{n{\overline{m}_n}}])\cdot 
([K_1] \cdots [K_n]) \\
\ar[uu]_-{1\cdot (
\phi_{k_{11},\dots,k_{1m_1}}\cdots
\phi_{k_{n1},\dots,k_{nm_n}})}
\lax{n}\cdot 
(\raisebox{-1.5pt}{$\lax{K_1} \cdots \lax{K_n}$})
\ar[r]_-{\Phi_{K_1,\dots,K_n}}& 
\lax{K+\sum_{i=1}^n Z(K_i)}\cdot
([K_1] \cdots [K_n])  
\ar[u]_-{\phi_{\widehat k_{11},\dots,\widehat k_{n{\overline{m}_n}}}\cdot 1} &
\ar[l]^-{\phi_{K_1,\dots,K_n}} 
\ar[uu]^-{\phi_{k_{11},\dots,k_{nm_n}}}
\lax K}
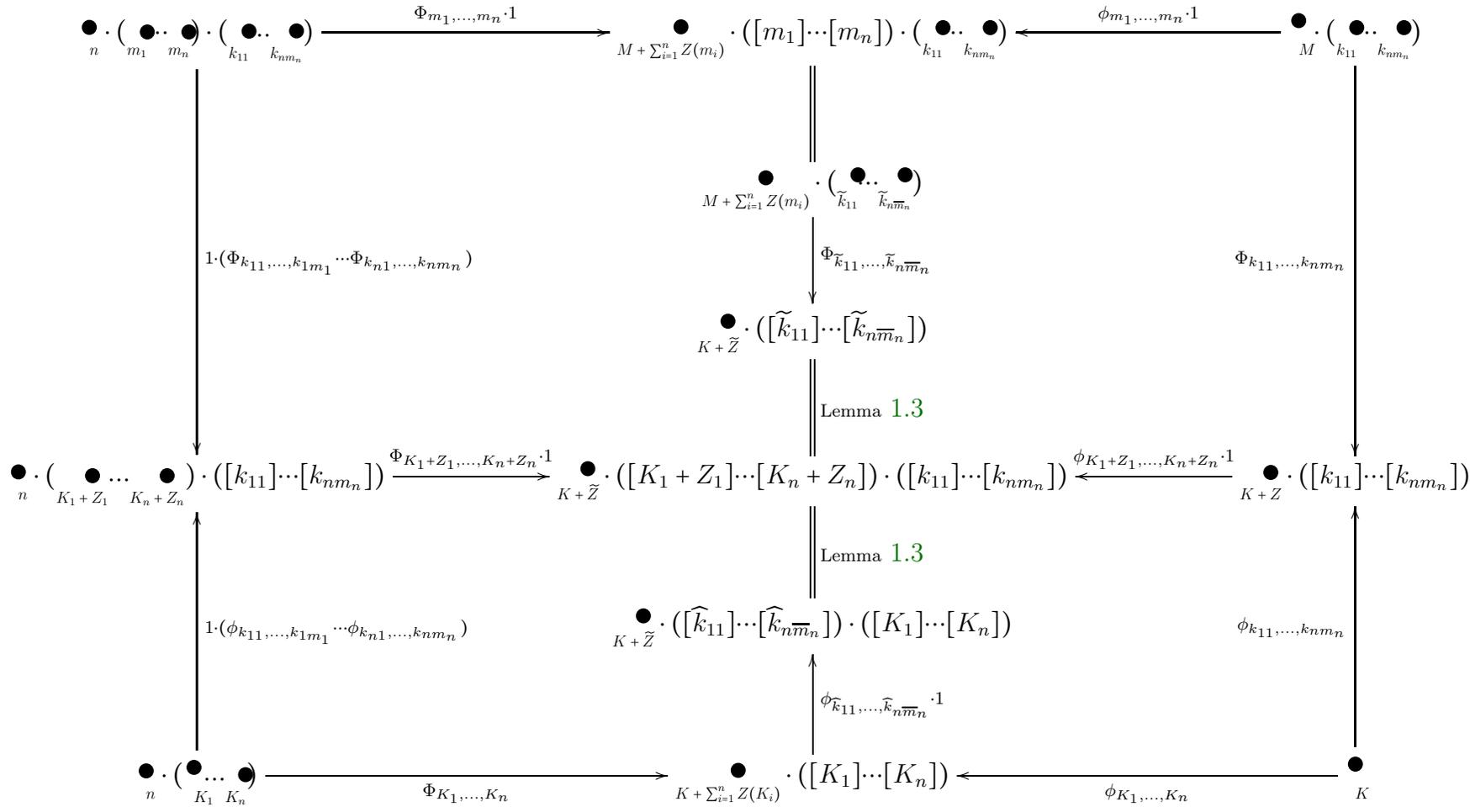
\caption{Axioms of $\mathsf{Lax}^+\mathsf{Oplax}^0$-monoidal category}
\label{fig:lax+oplax0}
\end{amssidewaysfigure}
\begin{equation} \label{eq:lax+oplax0unit}
\xymatrix@C=30pt@R=28pt{
\lax n \ar[r]^-{\iota\cdot 1} \ar[d]_-{1\cdot (\iota \cdots \iota)}
\ar@{=}[rd] &
\raisebox{-1.5pt}{$\lax 1$} \cdot \lax n \ar[d]^-{\Phi_{n}} \\
\lax n \cdot (\raisebox{-1.5pt}{$\lax 1 \cdots \lax 1$}) \ar[r]_-{\Phi_{1,\dots,1}} &
\lax n}
\end{equation}
as well as the regions of the diagram of Figure \ref{fig:lax+oplax0} commute, for all non-negative integers $n$, $\{m_i\}$ labelled by ${i=1,\dots,n}$, and $\{k_{ij}\}$ labelled by $i=1,\dots ,n$ and $j=1,\dots,m_i$.

Remark that the diagram of Figure \ref{fig:lax+oplax0} reduces to that of Definition \ref{def:lax+} if $Z=0$. Hence, forgetting about the natural transformations $\phi_{k_1,\dots,k_n}$, and those instances of $\Phi_{k_1,\dots,k_n}$ for which $Z> 0$, a $\mathsf{Lax}^+\mathsf{Oplax}^0$-monoidal category can be regarded as a $\mathsf{Lax}^+$-monoidal category. 

A $\mathsf{Lax}^+\mathsf{Oplax}^0$-monoidal category is {\em (strict) normal} if it is (strict) normal as a $\mathsf{Lax}^+$-monoidal category.
\end{definition}

Again, Definition \ref{def:lax+oplax0} can be extended in a straightforward way to define $\mathsf{Lax}^+\mathsf{Oplax}^0$-monoids in any Gray monoid --- so in particular in any strict monoidal 2-category.

Throughout, in a $\mathsf{Lax}^+\mathsf{Oplax}^0$-monoidal category we denote by $j$ the image of the single object of $\mathbbm 1$ under the functor $\lax 0:\mathbbm 1 \to \mathsf L$; and keep the earlier notation $\lb a \rb$ for the image of any object $a$ under the functor $\lax 1: \mathsf L \to \mathsf L$; and $a_1 \Lax \cdots \Lax a_n$ for the image of an object $(a_1,\dots,a_n)$ under the functor $\lax n:\mathsf L^n \to \mathsf L$.

Analogously to Proposition \ref{prop:Cartesian+}, the following holds.

\begin{proposition}
\label{prop:Lax+Oplax0Cartesian}
The Cartesian product $\mathsf L \mathsf L'$ of $\mathsf{Lax}^+\mathsf{Oplax}^0$-monoidal categories $(\mathsf L,\Lax,$ $\Phi,\phi,\iota)$ and $(\mathsf L',\Lax',\Phi',\phi',\iota')$ is again $\mathsf{Lax}^+\mathsf{Oplax}^0$-monoidal via 
\begin{itemize}
\item the functors 
$\xymatrix{
(\mathsf L \mathsf L')^n \ar[r]^-{\tau_{2n}} &
\mathsf L^n \mathsf L^{\prime n} \ar[r]^-{\lax n \lax n'} &
\mathsf L \mathsf L'}$
\item the natural transformations $\iota\iota':1\to \lax 1\lax 1'$,
$$
\xymatrix@C=60pt {
(\mathsf L \mathsf L')^K \ar[r]^-{\tau_{2k_1}\cdots \tau_{2k_n}}
\ar[dd]_-{[k_1][k_1]'\cdots [k_n][k_n]'} \ar[rd]^-{\tau_{2K}} &
\mathsf L^{k_1} \mathsf L^{\prime k_1} \cdots 
\mathsf L^{k_n} \mathsf L^{\prime k_n}
\ar[r]^-{\lax {k_1} \lax {k_1}' \cdots \lax {k_n} \lax {k_n}'}
\ar[d]_-{\tau_{2n}} &
(\mathsf L \mathsf L')^n \ar[d]^-{\tau_{2n}} \\
& \mathsf L^K \mathsf L^{\prime K}
\ar@{}[rd]|-{\Longdownarrow {\Phi_{k_1,\dots, k_n}\Phi'_{k_1,\dots, k_n}}}
\ar[r]^-{\lax {k_1} \cdots \lax {k_n}\lax {k_1}' \cdots \lax {k_n}'}
\ar[d]_-{[k_1]\cdots [k_n][k_1]'\cdots [k_n]'} &
\mathsf L^n \mathsf L^{\prime n} 
\ar[d]^-{\lax n \lax n'} \\
(\mathsf L \mathsf L')^{K+Z}
\ar[r]_-{\tau_{2(K+Z)}} &
\mathsf L^{K+Z} \mathsf L^{\prime K+Z}
\ar[r]_-{\lax {K+Z} \lax {K+Z}'} &
\mathsf L \mathsf L'}
$$
$$
\xymatrix@C=60pt{
(\mathsf L \mathsf L')^K \ar[r]^-{\tau_{2K}} 
\ar[d]_-{[k_1][k_1]'\cdots [k_n][k_n]'} &
\mathsf L^K \mathsf L^{\prime K}
\ar@{}[rd]|-{\Longdownarrow {\phi_{k_1,\dots,  k_n}\phi'_{k_1,\dots , k_n}\ \ }}
\ar[d]_-{[k_1]\cdots [k_n][k_1]'\cdots [k_n]'} 
\ar@/^1.7pc/[rd]^-{\lax K \lax K'} \\
\mathsf L^n \mathsf L^{\prime n} 
(\mathsf L \mathsf L')^{K+Z}
\ar[r]_-{\tau_{2(K+Z)}} &
\mathsf L^{K+Z} \mathsf L^{\prime K+Z}
\ar[r]_-{\lax {K+Z} \lax {K+Z}'} &
\mathsf L \mathsf L'}
$$
where the unlabelled regions denote identity natural transformations.
\end{itemize}
\end{proposition}

\begin{proof}
We leave it to the reader to check that the stated datum satisfies the axioms in Definition \ref{def:lax+oplax0}.
\end{proof}

\begin{definition}
\label{def:lax2oplax0-functor}
A {\em $\mathsf{Lax}^+\mathsf{Oplax}^0$-monoidal functor} consists of
\begin{itemize}
\item a functor $G:\mathsf L \to \mathsf L'$
\item for all non-negative integers $n$, a natural transformation
$\Gamma_n:G\cdot \lax n \to \lax n'\cdot (G\cdots G)$
\end{itemize}
such that for all sequences of non-negative integers $\{k_1,\dots,k_n\}$, for the sequence of natural transformations
$$
\xymatrix@C=17pt{
G^{\overline k_i} \cdot [k_i] 
\ar[r]^-{\lsem k_i \rsem} &
[k_i]' \cdot G^{k_i}}:=
\left\{
\begin{array}{ll}
\arraycolsep=1pt\def\arraystretch{.5}
\xymatrix@C=20pt{G ^{k_i} \ar[r]^-1 & G ^{k_i}} 
& \textrm{if}\quad k_i>0 \\
\xymatrix@C=12pt{
G\cdot  \lax 0
\ar[r]^-{\Gamma_0} &  \lax 0'}
& \textrm{if}\quad k_i=0
\end{array}
\right .
$$
for $i=1,\dots,n$, and using the notation of Section \ref{sec:preli}, the following diagrams commute.
$$
\xymatrix{
G \ar[r]^-{1 \cdot \iota} \ar[rd]_-{\iota'\cdot 1} &
G \cdot \lax 1 \ar[d]^-{\Gamma_1} \\
& \lax 1' \cdot G}
$$
$$
\xymatrix@C=45pt{
G\cdot \lax n \cdot (\raisebox{-1.5pt}{$\lax{k_{1}} \cdots \lax{k_{n}}$})
\ar[d]_-{\Gamma_n\cdot 1}
\ar[r]^-{1\cdot \Phi_{k_1,\dots,k_n}} &
G \cdot \lax{K+Z} \cdot ([k_1]\cdots [k_n])
\ar[d]^-{\Gamma_{K+Z}\cdot 1} &
\ar[l]_-{1\cdot \phi_{k_1,\dots,k_n}}
G \cdot \lax K \ar[dd]^-{\Gamma_K} \\
\lax n'\cdot (G \cdots G) \cdot (\raisebox{-1.5pt}{$\lax{k_{1}} \cdots \lax{k_{n}}$})
\ar[d]_-{1\cdot (\Gamma_{k_1} \cdots \Gamma_{k_n})} &
\lax{K+Z}'\cdot (G \cdots G) \cdot ([k_1]\cdots [k_{n}])
\ar[d]^-{1\cdot(\lsem k_1 \rsem \cdots \lsem k_n \rsem)} \\
\lax n'\cdot (\raisebox{-1.5pt}{$\lax{k_{1}}' \cdots \lax{k_{n}}'$})
\cdot (G \cdots G)
\ar[r]_-{\Phi'_{k_1,\dots,k_n}\cdot 1} &
\lax{K+Z}'\cdot  ([k_1]'\cdots [k_{n}]') \cdot (G \cdots G) &
\ar[l]^-{\phi'_{k_1,\dots,k_n}\cdot 1}
\lax K' \cdot (G \cdots G)}
$$
\end{definition}
Forgetting about the nullary part of the structure, $\mathsf{Lax}^+\mathsf{Oplax}^0$-monoidal functors can be seen $\mathsf{Lax}^+$-monoidal.

Analogously to Proposition \ref{prop:comp_functor}, the following holds.

\begin{proposition}
\label{prop:lax+oplax0-functor}
\begin{itemize}
\item[{(1)}]
The composite of (composable) $\mathsf{Lax}^+\mathsf{Oplax}^0$-monoidal functors $(G,\Gamma)$ and $(H,\Xi)$ is again $\mathsf{Lax}^+\mathsf{Oplax}^0$-monoidal via the natural transformations
$$
\xymatrix{
H\cdot G \cdot \lax n \ar[r]^-{1\cdot \Gamma_n} &
H\cdot  \lax n \cdot (G \cdots G) \ar[r]^-{\Xi_n\cdot 1} &
\lax n \cdot (H\cdots H) \cdot (G \cdots G)=
\lax n \cdot (H\cdot G)^n.}
$$
\item[{(2)}]
The Cartesian product of $\mathsf{Lax}^+\mathsf{Oplax}^0$-monoidal functors $(G,\Gamma)$ and $(H,\Xi)$ is again $\mathsf{Lax}^+\mathsf{Oplax}^0$-monoidal via the natural transformations
$$
\xymatrix@C=60pt{
(\mathsf{LN})^n \ar[r]^-{\tau_{2n}} \ar[d]_-{(GH)^n} &
\mathsf{L}^n \mathsf{N}^n \ar[r]^-{\lax n \lax n} \ar[d]_-{G^nH^n}
\ar@{}[rd]|-{\Longdownarrow {\Gamma_n\Xi_n}}  &
\mathsf{LN} \ar[d]^-{GH} \\
(\mathsf{L'N'})^n \ar[r]_-{\tau_{2n}}  &
\mathsf{L'}^n \mathsf{N'}^n \ar[r]_-{\lax n' \lax n'} &
\mathsf{L'N'}}
$$
where the unlabelled region denotes the identity natural transformation.
\end{itemize}
\end{proposition}

\begin{proof}
It is left to the reader to check that the stated datum in both parts satisfies the axioms in Definition \ref{def:lax2oplax0-functor}.
\end{proof}

\begin{definition}
\label{def:lax2oplax0-nattr}
A {\em $\mathsf{Lax}^+\mathsf{Oplax}^0$-monoidal natural transformation} $(G,\Gamma) \to (G',\Gamma')$ is a natural transformation $\omega:G \to G'$ for which the following diagram commutes for all non-negative integers $n$.
$$
\xymatrix@C=45pt{
G\cdot \lax n  \ar[r]^-{\omega \cdot 1} \ar[d]_-{\Gamma_n} &
G'\cdot \lax n \ar[d]^-{\Gamma'_n} \\
\lax n'\cdot (G \cdots G) \ar[r]_-{1\cdot (\omega \cdots \omega)} &
\lax n'\cdot (G'\cdots G')}
$$
\end{definition}
$\mathsf{Lax}^+\mathsf{Oplax}^0$-monoidal natural transformations are in particular $\mathsf{Lax}^+$-monoidal.

Analogously to Proposition \ref{prop:comp_nattr}, the following holds.

\begin{proposition}
\label{prop:2catLax+Oplax0}
All of the composites, the Godement products, and the Cartesian products of (composable) $\mathsf{Lax}^+\mathsf{Oplax}^0$-monoidal natural transformations are again $\mathsf{Lax}^+\mathsf{Oplax}^0$-monoidal. Consequently, there is a strict monoidal 2-category $\mathsf{Lax}^+\mathsf{Oplax}^0$ of $\mathsf{Lax}^+\mathsf{Oplax}^0$-monoidal categories, $\mathsf{Lax}^+\mathsf{Oplax}^0$-monoidal functors and $\mathsf{Lax}^+\mathsf{Oplax}^0$-monoidal natural transformations, admitting a strict monoidal forgetful 2-functor to the 2-category $\mathsf{Lax}^+$ of Proposition \ref{prop:comp_nattr}.
\end{proposition}

\begin{proof}
The straightforward check of the axiom in Definition \ref{def:lax2oplax0-nattr} in each of the stated cases is left to the reader.
\end{proof}

\begin{definition}
\label{def:comonoid}
A {\em comonoid} in a $\mathsf{Lax}^+\mathsf{Oplax}^0$-monoidal category $(\mathsf L,\Lax,\Phi,\phi,\iota)$ consists of 
\begin{itemize}
\item a cosemigroup $(a,\delta)$ in the $\mathsf{Lax}^+$-monoidal category $(\mathsf L,\Lax,\Phi,\iota)$ and
\item a morphism $\varepsilon:a \to j$ in $\mathsf L$
\end{itemize}
rendering commutative the first (counitality) diagram below.
A {\em comonoid morphism} $(a,\delta,\varepsilon) \to (a',\delta',\varepsilon')$ is a morphism of cosemigroups $\phi:(a,\delta) \to (a',\delta')$ rendering commutative the second diagram as well.
$$
\xymatrix{
a\Lax a \ar[d]_-{\varepsilon \bullet 1} &
a \ar[d]^-\iota \ar[l]_-\delta \ar[r]^-\delta &
a\Lax a \ar[d]^-{1 \bullet \varepsilon} \\
j \Lax a  &
\ar[l]^-{\phi_{0,1}}
\lb a \rb
\ar[r]_-{\phi_{1,0}}&
 a \Lax j }
\qquad \qquad
\xymatrix@R=25pt{
a\ar[r]^-\varepsilon \ar[d]_-\phi &
j \ar@{=}[d] \\
a' \ar[r]_-{\varepsilon'} &
j}
$$
\end{definition}

\begin{theorem}
\label{thm:comonoid}
For any $\mathsf{Lax}^+\mathsf{Oplax}^0$-monoidal category $(\mathsf L,\Lax,\Phi,\phi,\iota)$, the following categories are isomorphic.
\begin{itemize}
\item[{(i)}]
The category of comonoids and their morphisms in $(\mathsf L,\Lax,\Phi,\phi,\iota)$.
\item[{(ii)}]
The category of $\mathsf{Lax}^+\mathsf{Oplax}^0$-monoidal functors $(\mathbbm 1,1,1,1,1) \to (\mathsf L,\Lax,\Phi,\phi,\iota)$ and their $\mathsf{Lax}^+\mathsf{Oplax}^0$-monoidal natural transformations. 
\end{itemize}
Consequently, $\mathsf{Lax}^+\mathsf{Oplax}^0$-monoidal functors preserve comonoids.
\end{theorem}

\begin{proof}
A $\mathsf{Lax}^+\mathsf{Oplax}^0$-monoidal functor $(\mathbbm 1,1,1,1,1) \to (\mathsf L,\Lax,\Phi,\phi,\iota)$ as in part (ii) consists of 
\begin{itemize}
\item an object $a$ of $\mathsf L$
\item for all non-negative integers $n$, a morphism $\delta_n$ from $a$ to the $n$-fold product $a^{\bullet n}:=a\Lax\cdots \Lax a$ 
for $n>1$, to $a^{\bullet 1}:= \lb a \rb$ for $n=1$, and to $a^{\bullet 0}:=j$ for $n=0$,
\end{itemize}
such that 
\begin{equation} \label{eq:delta_1}
\delta_1=\iota
\end{equation}
and for all sequences of non-negative integers $\{k_1,\dots,$ $k_n\}$,
and for the double sequences $\{\widehat a_{pq}\}_{\stackrel{p=1,\dots,n}{{}_{q=1,\dots,\overline k_p}}}$ of elements in $\mathsf L$, and 
$\{\langle pq \rangle \}_{\stackrel{p=1,\dots,n}{{}_{q=1,\dots,\overline k_p}}}$ of morphisms in $\mathsf L$, given by 
$$
\widehat a_{pq}=\left\{
\begin{array}{ll}
\arraycolsep=1pt\def\arraystretch{.5}
a & \textrm{ if } k_p>0\\
j & \textrm{ if } k_p=0,
\end{array}
\right.
\qquad  \qquad 
\xymatrix@C=20pt{a \ar[r]^-{\langle pq \rangle} & \widehat a_{pq}} =\left\{
\arraycolsep=1pt\def\arraystretch{.5}
\begin{array}{ll}
\xymatrix@C=12pt{a \ar[r]^-1 &a} & \textrm{ if } k_p>0\\
\xymatrix@C=12pt{a \ar[r]^-{\delta_0} &j}  & \textrm{ if } k_p=0.
\end{array}
\right.
$$
the following diagram commutes.
\begin{equation} \label{eq:+0coass}
\xymatrix@R=20pt@C=60pt{
a^{\bullet n}
\ar[dd]_-{\delta_{k_1} \bullet \cdots \bullet \delta_{k_n}} &
a \ar[l]_-{\delta_n} \ar@{=}[r] \ar[d]^-{\delta_{K+Z}} &
a \ar[dd]^-{\delta_K} \\
& a^{\bullet K+Z} 
\ar[d]^-{\langle 11 \rangle \bullet \cdots \bullet \langle 1\overline k_1 \rangle
\bullet \cdots \bullet \langle n1 \rangle \bullet \cdots \bullet \langle n\overline k_n \rangle} \\
a^{\bullet k_1} \Lax\cdots \Lax a^{\bullet k_n}
\ar[r]_-{\Phi_{k_1,\dots,k_n}} &
\widehat a_{11} \Lax \cdots \Lax \widehat a_{1\overline k_1} \Lax \cdots \Lax \widehat a_{n1}\Lax \cdots \Lax \widehat a_{n\overline k_n} &
\ar[l]^-{\phi_{k_1,\dots,k_n}} a^{\bullet K}}
\end{equation}

If $Z=0$ then the region on the left of \eqref{eq:+0coass} reduces to \eqref{eq:n-coass}. 
So we conclude by Theorem \ref{thm:cosemigroup} that $(a,\delta_2)$ is a cosemigroup. Evaluating the region on the right of \eqref{eq:+0coass} at $n=2$, $k_1=0$, $k_2=1$ and $n=2$, $k_1=1$, $k_2=0$ we infer that $\delta_0$ is a counit for it.

Conversely, if $(a,\delta,\varepsilon)$ is a comonoid in $(\mathsf L,\Lax,\Phi,\phi,\iota)$ then we put $\delta_0:=\varepsilon$, $\delta_1:=\iota$, $\delta_2:=\delta$ and for $i>1$ we define $\delta_{i+1}$ iteratively as in \eqref{eq:delta_k+1} (so that it satisfies \eqref{eq:delta_alt}). By the following similar procedure to that in the proof of Theorem \ref{thm:cosemigroup}, one checks by induction on $K+Z$ that \eqref{eq:+0coass} commutes for the so defined family of morphisms, for all sequences of non-negative integers $\{k_1,\dots,$ $k_n\}$.

If $K+Z=0$ then $n=0$ and \eqref{eq:+0coass} reduces to a trivial identity. If $K+Z=1$ then $n=1$ and $k_1$ is either $0$ or $1$. If $k_1=0$ then \eqref{eq:+0coass} commutes by the axioms $\Phi_{0}=1$ and $\phi_{0}=\iota \cdot 1$
of $\mathsf{Lax}^+\mathsf{Oplax}^0$-monoidal category.
If $k_1=1$ then \eqref{eq:+0coass} commutes by the axiom $\phi_{1}=1$ of $\mathsf{Lax}^+\mathsf{Oplax}^0$-monoidal category and the upper triangle of \eqref{eq:lax+oplax0unit} for $n=1$.

There are now three kinds of the induction step:
\begin{itemize}
\item replacing $\{k_1,\dots,k_n\}$ with $\{0,k_1,\dots,k_n\}$,
\item replacing $\{k_1,\dots,k_n\}$ with $\{1,k_1,\dots,k_n\}$, and
\item replacing $\{k_1,\dots,k_n\}$ with $\{1+k_1,\dots,k_n\}$ if $k_1>0$.
\end{itemize}
We leave it to the reader to check that if \eqref{eq:+0coass} commutes for the original sequence $\{k_1,\dots,k_n\}$ then it also commutes for the modified sequence in each of the listed cases.

These constructions are mutually inverse bijections by \eqref{eq:delta_1} and by \eqref{eq:+0coass} at $n=1$, $k_1=1$, and $k_2=i$ (see the proof of Theorem \ref{thm:cosemigroup}).

A $\mathsf{Lax}^+\mathsf{Oplax}^0$-monoidal natural transformation $(a,\{\delta_i\}_{i\in \mathbb N}) \to (a',\{\delta'_i\}_{i\in \mathbb N})$ is obviously a comonoid morphism $(a,\delta_2,\delta_0) \to (a',\delta'_2,\delta'_0)$. Conversely, a comonoid morphism $(a,\delta,\varepsilon) \to (a',\delta',\varepsilon')$ is compatible with the unary parts $\delta_1=\iota=\delta'_1$ of the corresponding $\mathsf{Lax}^+\mathsf{Oplax}^0$-monoidal functors by the naturality of $\iota$ and it is compatible with their higher components by Theorem \ref{thm:cosemigroup}.
\end{proof}

\begin{definition} \label{def:comonoid_comodule}
A {\em comodule} of a comonoid $(a,\delta,\varepsilon)$ in a $\mathsf{Lax}^+\mathsf{Oplax}^0$-monoidal category $(\mathsf L,\Lax,\Phi,\phi,\iota)$ 
is a comodule $(x,\varrho: x \to x \Lax a)$ in the sense of Definition \ref{def:comodule} over the underlying cosemigroup $(a,\delta)$ such that also the following (counitality) diagram commutes.
$$
\xymatrix{
x \ar[r]^-\varrho \ar[d]_-\iota &
x \Lax a  \ar[d]^-{1 \bullet \varepsilon}  \\
\lb x \rb \ar[r]_-{\phi_{1,0}} &
x \Lax j  }
$$
A {\em morphism of comodules} over a comonoid is a morphism of comodules over the underlying cosemigroup in the sense of Definition \ref{def:comodule}.
\end{definition}

\begin{example} \label{ex:comonoid_reg_comod}
Comparing the counitality conditions of Definition \ref{def:comonoid} and Definition \ref{def:comonoid_comodule}, it follows immediately from Example \ref{ex:reg_comod} that $(a,\delta)$ is a comodule over an arbitrary comonoid $(a,\delta,\varepsilon)$ in any $\mathsf{Lax}^+\mathsf{Oplax}^0$-monoidal category.
\end{example}

Generalizing a construction in \cite{CaenepeelGoyvaerts}, we obtain the following example of $\mathsf{Lax}^+\mathsf{Oplax}^0$-monoidal category (in whose presentation below the notation of Section \ref{sec:preli} is used).

\begin{theorem} \label{thm:Lax}
Associated to any monoidal category $(\mathsf V,\otimes, I)$, there is a strictly normal $\mathsf{Lax}^+\mathsf{Oplax}^0$-monoidal category $(\mathsf L,\otimes,\Phi,\phi,1)$ as follows.
\begin{itemize}
\item In the category $\mathsf L$, an \underline{object} consists of an object $a$ of $\mathsf V$ together with two morphisms $\alpha,\beta:a\to a$ such that $\alpha\cdot \beta=\beta\cdot \alpha$. 

\noindent
A \underline{morphism} $(a,\alpha,\beta) \to (a',\alpha',\beta')$ is a morphism $\vartheta:a\to a'$ such that $\vartheta\cdot\alpha=\alpha'\cdot \vartheta$ and $\vartheta\cdot\beta=\beta'\cdot \vartheta$.
\item For $n=0$ we put $(I,1,1):\mathbbm 1 \to \mathsf L$,
for $n=1$ we take the identity functor $\mathsf L \to \mathsf L$
and for any $n>1$ we take the $n$-fold monoidal product functor $\mathsf L^n \to \mathsf L$,
\begin{eqnarray*}
&&\xymatrix@C=50pt{
\{(a_1,\alpha_1,\beta_1),\dots,(a_n,\alpha_1,\beta_n)\}
\ar[r]^-{\{\vartheta_1,\dots,\vartheta_n\}} &
\{(a'_1,\alpha'_1,\beta'_1),\dots,(a'_n,\alpha'_1,\beta'_n)\}}
\mapsto \\
&& \xymatrix@C=30pt{
(\!a_1\!\otimes \cdots \otimes \!a_n,
\alpha_1\!\otimes \cdots \!\otimes \alpha_n,
\beta_1\!\otimes \cdots \otimes \!\beta_n)
\ar[r]^-{\vartheta_1\!\otimes\cdots \otimes\!\vartheta_n} &
(\!a'_1\!\otimes \cdots \otimes \!a'_n,
\alpha'_1\!\otimes \cdots \otimes \!\alpha'_n,
\beta'_1\!\otimes \cdots \otimes \!\beta'_n).}
\end{eqnarray*}
\item For any sequence of non-negative integers $\{k_1,\dots,k_n\}$, and for any collection of objects $\{(a_{ij},\alpha_{ij},\beta_{ij})\}_{\stackrel{i=1,\dots,n}{{}_{j=1,\dots,k_i}}}$, the natural transformation $\Phi_{k_1,\dots,k_n}$ is taken to be
$$
\otimes_{i=1}^n
\otimes_{j=1}^{k_i} 
\alpha_{ij}^{\sum_{p=i+1}^n(\overline k_p-1)}
\beta_{ij}^{\sum_{p=1}^{i-1}(\overline k_p-1)}:
\otimes_{i=1}^n
\otimes_{j=1}^{k_i} a_{ij} \to 
\otimes_{i=1}^n
\otimes_{j=1}^{k_i} a_{ij},
$$
and the natural transformation 
$\phi_{k_1,\dots,k_n}$ is taken to be
$$
\otimes_{i=1}^n
\otimes_{j=1}^{k_i} 
\alpha_{ij}^{\sum_{p=i+1}^n Z(k_p)}
\beta_{ij}^{\sum_{p=1}^{i-1} Z(k_p)}:
\otimes_{i=1}^n
\otimes_{j=1}^{k_i} a_{ij} \to 
\otimes_{i=1}^n
\otimes_{j=1}^{k_i} a_{ij}.
$$
\end{itemize}
\end{theorem}

\begin{proof}
Obviously, $\Phi_{k_1,\dots,k_n}=1$ if $K=0$ and $\phi_{k_1,\dots,k_n}=1$ if $Z=0$. Since also $\Phi_{n}$ and $\Phi_{1,\dots,1}$ are trivial for all positive integers $n$, we only need to check the commutativity of the regions of the diagram of Figure \ref{fig:lax+oplax0} for the stated natural transformations. For each region of Figure \ref{fig:lax+oplax0} this means a comparison of endomorphisms of 
$\otimes_{i=1}^n
\otimes_{j=1}^{m_i}
\otimes_{l=1}^{k_{ij}} a_{ijl}
$ for non-negative integers $n$, $\{m_i\}_{i=1,\dots,n}$, 
and $\{k_{ij}\}_{\stackrel {i=1,\dots,n}{{}_{j=1,\dots,m_i}}}$; 
and for all collections $\{(a_{ijl},\alpha_{ijl},\beta_{ijl})\}\!\!\!
\raisebox{-2pt}{
{\scalebox{.5}{$
\renewcommand*{\arraystretch}{0}
\begin{array}{l}
i=1,\dots,n \\
j=1,\dots,m_i\\
\,l=1,\dots,k_{ij}
\end{array}$}}}
$ of objects in $\mathsf L$. 
Each of these endomorphisms is a monoidal product of endomorphisms of the individual factors $a_{ijl}$, and each of the occurring endomorphisms of $a_{ijl}$ is a composite of copies of the commuting morphisms $\alpha_{ijl}$ and $\beta_{ijl}$. So we only need to compare the resulting exponents of $\alpha_{ijl}$ and of $\beta_{ijl}$ for all values of $i,j,l$. 
We only do it for $\beta_{ijl}$ in each case; the computations for $\alpha_{ijl}$ are symmetric.

Then the four regions of the diagram of Figure \ref{fig:lax+oplax0} amount to the respective conditions
\begin{eqnarray*}
&&
\sum_{p=1}^{i-1}(\overline m_p-1)+
\sum_{p=1}^{i-1} \sum_{q=1}^{\overline m_p}
(\overline{\widetilde k}_{pq}-1) +
\sum_{q=1}^{j-1} (\overline{\widetilde k}_{iq}-1)=
\sum_{q=1}^{j-1} (\overline{k}_{iq}-1)+
\sum_{p=1}^{i-1}(\overline{K_p+Z}_p-1) \\
&& \sum_{q=1}^{j-1} Z(k_{iq}) +
\sum_{p=1}^{i-1}(\overline{K_p+Z}_p\!-\!1) =
\sum_{p=1}^{i-1}(\overline K_p\!-\!1) +
\sum_{p=1}^{i-1} \sum_{q=1}^{\overline m_p} Z(\widehat k_{pq}) +
\sum_{q=1}^{j-1} Z(\widehat k_{iq}) \\
&& \sum_{p=1}^{i-1} Z(m_p)\! +\!
\sum_{p=1}^{i-1} \sum_{q=1}^{\overline m_p}
(\overline{\widetilde k}_{pq}\!-\!1)\!+\!
\sum_{q=1}^{j-1} (\overline{\widetilde k}_{iq}-1)\!=\!
\sum_{p=1}^{i-1} \sum_{q=1}^{\overline m_p}(\overline{k}_{pq}\!-\!1)\!+\!
\sum_{q=1}^{j-1} (\overline{k}_{iq}-1)\!+\!
\sum_{p=1}^{i-1} Z(K_p+Z_p) \\
&& \sum_{p=1}^{i-1} Z_p +
\sum_{q=1}^{j-1} Z(k_{iq}) +
\sum_{p=1}^{i-1} Z(K_p+Z_p) =
\sum_{p=1}^{i-1} Z(K_p)+
\sum_{p=1}^{i-1} \sum_{q=1}^{\overline m_p} Z(\widehat k_{pq})+
\sum_{q=1}^{j-1} Z(\widehat k_{iq})
\end{eqnarray*}
for all $i=1,\dots,n$ and $j=1,\dots,m_i$, which are easily checked to hold.
\end{proof}

A straightforward application of Definition \ref{def:comonoid} yields the following.

\begin{theorem}
\label{thm:comon_lax+fun}
For any monoidal category $(\mathsf V,\otimes,I)$ the following categories are pairwise isomorphic.
\begin{itemize}
\item[{(1.i)}] The category of cosemigroups in the $\mathsf{Lax}^+$-monoidal category $\mathsf L$ of Theorem \ref{thm:Lax}.
\item[{(1.ii)}] The category of \BiHom-comonoids in $\mathsf V$.
\end{itemize}
and 
\begin{itemize}
\item[{(2.i)}] The category of comonoids in the $\mathsf{Lax}^+\mathsf{Oplax}^0$-monoidal category $\mathsf L$ of Theorem \ref{thm:Lax}.
\item[{(2.ii)}] The category of counital \BiHom-comonoids in $\mathsf V$.
\end{itemize}
\end{theorem}

\begin{proof}
In both categories of part (1), an object consists of 
\begin{itemize}
\item an object $(a,\alpha,\beta)$ of $\mathsf L$
\item a morphism $\delta:(a,\alpha,\beta)\to (a\otimes a,\alpha\otimes \alpha,\beta\otimes \beta)$ of $\mathsf L$
\end{itemize}
such that the first diagram of 
$$
\xymatrix@C=60pt{
a\otimes a \ar[d]_-{\delta\otimes 1} &
\ar[l]_-\delta a \ar[r]^-\delta &
a\otimes a \ar[d]^-{1\otimes \delta} \\
a\otimes a \otimes a \ar[r]_-{\Phi_{2,1}=
1\otimes 1 \otimes \beta} &
a\otimes a \otimes a &
a\otimes a \otimes a \ar[l]^-{\Phi_{1,2}=
\alpha\otimes 1 \otimes 1}}
\qquad
\xymatrix{
a\ar[r]^-{\delta} \ar[d]_-\vartheta &
a\otimes a \ar[d]^-{\vartheta \otimes \vartheta} \\
a'\ar[r]_-{\delta'} &
a' \otimes a'}
$$
commutes. In both categories of part (1), a morphism $((a,\alpha,\beta),\delta) \to ((a',\alpha',\beta'),\delta')$ is a morphism $\vartheta:(a,\alpha,\beta) \to (a',\alpha',\beta')$ in $\mathsf L$ such that the second diagram above commutes.

In both categories of part (2), an object consists of 
\begin{itemize}
\item an object $((a,\alpha,\beta),\delta)$ in the isomorphic categories of part (1)
\item a morphism $\varepsilon:(a,\alpha,\beta) \to (I,1,1)$ in $\mathsf L$ 
\end{itemize}
such that the first diagram of
$$
\xymatrix{
a\otimes a \ar[d]_-{\varepsilon \otimes 1} &
\ar[l]_-\delta a \ar[r]^-\delta \ar@{=}[d] &
a\otimes a \ar[d]^-{1\otimes \varepsilon} \\
a & 
\ar[l]^-{\phi_{0,1}=\beta } a \ar[r]_-{\phi_{1,0}=\alpha} &
a}
\qquad \qquad
\xymatrix{
a\ar[r]^-\varepsilon \ar[d]_-\vartheta &
I \ar@{=}[d] \\
a' \ar[r]_-{\varepsilon'} &
I}
$$
commutes. In both categories of part (2), a morphism $((a,\alpha,\beta),\delta,\varepsilon) \to ((a',\alpha',\beta'),\delta',$ $\varepsilon')$ is a morphism $\vartheta:((a,\alpha,\beta),\delta) \to ((a',\alpha',\beta'),\delta')$ in the isomorphic categories of part (1) such that the second diagram above commutes.
\end{proof}

A comodule in the sense of Definition \ref{def:comodule} over a cosemigroup in the 
$\mathsf{Lax}^+$-monoidal category $\mathsf L$ of Theorem \ref{thm:Lax} is precisely the same as a comodule in the sense of \cite[Definition 5.3]{GraMakMenPan} over the corresponding \BiHom-comonoid in Theorem \ref{thm:comon_lax+fun}~(1.ii).
A comodule in the sense of Definition \ref{def:comonoid_comodule} over a comonoid in the $\mathsf{Lax}^+\mathsf{Oplax^0}$-monoidal category $\mathsf L$ of Theorem \ref{thm:Lax} is precisely the same as a counital comodule in the sense of \cite[Definition 5.3]{GraMakMenPan} over the corresponding counital \BiHom-comonoid in Theorem \ref{thm:comon_lax+fun}~(2.ii).

The same category $\mathsf L$ of Theorem \ref{thm:Lax} also has a `biased' monoidal structure inherited from $(\mathsf V,\otimes,I)$. In this monoidal category --- let us denote it by $\mathsf M$ --- the monoidal product and the monoidal unit are $\raisebox{-3pt}{$\stackrel {\displaystyle \otimes} {{}_2}$}=\otimes$, respectively, $\raisebox{-3pt}{$\stackrel {\displaystyle \otimes} {{}_0}$}=(I,1,1)$ of $\mathsf L$. Mac Lane's coherence natural isomorphisms are given by the same (omitted) morphisms as in $\mathsf V$.
A cosemigroup (respectively, comonoid) in $(\mathsf M,\raisebox{-3pt}{$\stackrel {\displaystyle \otimes} {{}_2}$},\raisebox{-3pt}{$\stackrel {\displaystyle \otimes} {{}_0}$})$ is a triple consisting of a cosemigroup (respectively, comonoid) in $(\mathsf V,\otimes,I)$ together with two commuting cosemigroup (respectively, comonoid) endomorphisms. Cosemigroup (respectively, comonoid) morphisms in $\mathsf M$ are those cosemigroup (respectively, comonoid) morphisms in $(\mathsf V,\otimes,I)$ which are compatible with the endomorphism parts too.

\begin{proposition} \label{prop:Yau}
Consider an arbitrary monoidal category $(\mathsf V,\otimes,I)$, the associated strictly normal $\mathsf{Lax}^+\mathsf{Oplax}^0$-monoidal category $\mathsf L$
of Theorem \ref{thm:Lax}, and the monoidal category $\mathsf M$ 
of the previous paragraph regarded as a $\mathsf{Lax}^+\mathsf{Oplax}^0$-monoidal category. The identity functor $\mathsf M \to \mathsf L$ carries a $\mathsf{Lax}^+\mathsf{Oplax}^0$-monoidal structure given by the morphisms 
$$
\Gamma_n:=\otimes_{i=1}^n \alpha_i^{n-i} \cdot \beta_i^{i-1}:
\otimes_{i=1}^n a_i \to \otimes_{i=1}^n a_i,
$$
for all non-negative integers $n$ and any sequence of objects $\{(a_i,\alpha_i,\beta_i)\}_{i=1,\dots,n}$.

Consequently, it preserves cosemigroups and comonoids. 
\end{proposition}

The action of the $\mathsf{Lax}^+\mathsf{Oplax}^0$-monoidal functor of Proposition \ref{prop:Yau} on cosemigroups and comonoids was termed in \cite[Proposition 5.9]{GraMakMenPan} the {\em Yau twist}.

\begin{proof}
A similar comparison of exponents as in the proof of Theorem \ref{thm:Lax} shows the commutaivity of the diagrams in Definition \ref{def:lax2oplax0-functor}.
\end{proof}

\begin{remark} \label{rem:invertible}
In the category $\mathsf L$ of Theorem \ref{thm:Lax} take the full subcategory for whose objects $(a,\alpha,\beta)$ the morphisms $\alpha$ and $\beta$ are invertible. It is monoidal with the associativity natural isomorphism 
$$\alpha^{-1} \otimes 1 \otimes \beta''\!:\!
(\!\!(\!a\otimes a')\otimes a''\!,\! (\alpha\otimes \alpha')\otimes \alpha''\!,\!(\beta\otimes \beta')\otimes \beta'')\! \to \!
(a\otimes (a'\otimes a'')\!,\! \alpha\otimes (\alpha'\otimes \alpha'')\!,\!\beta\otimes (\beta'\otimes \beta''\!)\!\!),
$$
left unit natural isomorphism
$$
\beta^{-1}: (I\otimes a,1\otimes \alpha,1\otimes \beta) \to (a,\alpha,\beta)
$$
and right unit natural isomorphism
$$
\alpha^{-1}: (a\otimes I, \alpha \otimes 1,\beta\otimes 1) \to (a,\alpha,\beta).
$$
This is the same monoidal category that was constructed in Section 2 of \cite{GraMakMenPan} under the name ${\mathcal H}^{(-1,0),(0,-1),1}(\mathbb Z \times \mathbb Z,\sf V)$. 
By Theorem \ref{thm:comon_lax+fun} the cosemigroups (respectively, comonoids) therein are those \BiHom-comonoids (respectively, counital \BiHom-comonoids) whose endomorphism parts are automorphisms; thus we re-cover \cite[Remark 3.5]{GraMakMenPan}.

Let us take also in the monoidal category $\mathsf M$ of Proposition \ref{prop:Yau} the full (monoidal) subcategory of those objects $(a,\alpha,\beta)$ whose morphisms $\alpha$ and $\beta$ are invertible; this is the monoidal category $\mathcal H (\mathbb Z \times \mathbb Z,\mathsf V)$ of \cite[Section 2]{GraMakMenPan}. 
As shown in \cite[Theorem 2.5]{GraMakMenPan}, the $\mathsf{Lax}^+\mathsf{Oplax}^0$-monoidal functor of Proposition \ref{prop:Yau} restricts to a strong monoidal isomorphism between these monoidal subcategories; hence it induces isomorphisms between the cosemigroups, as well as the comonoids in these monoidal categories. That is, analogously to \cite[Claim 3.7]{GraMakMenPan}, we obtain an isomorphism between
\begin{itemize}
\item[{(i)}] the full subcategory of the isomorphic categories in Theorem \ref{thm:comon_lax+fun}~(1) for whose objects $((a,\alpha,\beta),\delta)$ the morphisms $\alpha$ and $\beta$ are invertible;
\item[{(ii)}] the category whose objects are triples consisting of a cosemigroup in $\mathsf V$ together with two commuting cosemigroup automorphisms; and the morphisms are the cosemigroup morphisms which commute with both of these automorphisms.
\end{itemize}
It induces an analogous isomorphism also between 
\begin{itemize}
\item[{(i)}] the full subcategory of the isomorphic categories in Theorem \ref{thm:comon_lax+fun}~(2) for whose objects $((a,\alpha,\beta),\delta,\varepsilon)$ the morphisms $\alpha$ and $\beta$ are invertible;
\item[{(ii)}] the category whose objects are triples consisting of a comonoid in $\mathsf V$ together with two commuting comonoid automorphisms; and the morphisms are the comonoid morphisms which commute with both of these automorphisms.
\end{itemize}
Restricting further to those objects for which $\beta=\alpha^{-1}$, the latter isomorphism reduces to \cite[Proposition 1.13]{CaenepeelGoyvaerts}.
\end{remark}

\section{Monoids in $\mathsf{Lax}_0\mathsf{Oplax}_+$-monoidal categories}
\label{sec:Lax0Oplax+}

Dually to the previous section, here we introduce monoids in monoidal categories in which all monoidal products with non-negative number of factors are available; and the monoidal products of positive number of factors are oplax coherent and the nullary product is lax coherent. Unital \BiHom-monoids are described as monoids in such monoidal categories.

\begin{definition}
A {\em $\mathsf{Lax}_0\mathsf{Oplax}_+$-monoidal category} consists of 
\begin{itemize}
\item a category $\mathsf R$
\item for all non-negative integers $n$, a functor $\oplax{n}:{\mathsf R}^{n} \to \mathsf R$ 
\item for all non-negative integers $n,k_1,\dots,k_n$, natural transformations
$$
\xymatrix@C=60pt{
{\mathsf R}^K 
\ar[rr]^-{\oplax {k_1} \ \cdots \ \oplax{k_n}}
\ar[rd]^-{\ [k_1]\cdots [k_n]} 
\ar@/_3pc/[rrdd]_-{\oplax K} &
\ar@{}[d]|-{\qquad\qquad\qquad  \Longuparrow {\Psi_{k_1,\dots,k_n}}} &
{\mathsf R}^n \ar[dd]^-{\oplax n}  \\
& {\mathsf R}^{K+Z}
\ar@{}[d]|-{\qquad\qquad  \Longdownarrow {\psi_{k_1,\dots,k_n}}}
\ar[rd]^-{\oplax {K+Z}}  & \\ 
&& {\mathsf R}}
\qquad\qquad
\raisebox{-17pt}{$
\xymatrix@C=60pt{
{\mathsf R} \ar@{=}@/^1.7pc/[r] \ar@/_1.7pc/[r]_-{\oplax 1}
\ar@{}[r]|-{\Longuparrow \upsilon} &
\mathsf R}$}
$$
\end{itemize}
satisfying the so-called coassociativity and counitality axioms encoded in the diagrams which are obtained from the diagrams of Definition \ref{def:lax+oplax0} by reversing the arrows.

Forgetting about the irrelevant structure, any $\mathsf{Lax}_0\mathsf{Oplax}_+$-monoidal category can be regarded as an $\mathsf{Oplax}_+$-monoidal category.

A $\mathsf{Lax}_0\mathsf{Oplax}_+$-monoidal category is {\em (strictly) normal} if it is (strictly) normal as an $\mathsf{Oplax}_+$-monoidal category.
\end{definition}

$(\mathsf R,\OpLax,\Psi,\psi,\upsilon)$ is a $\mathsf{Lax}_0\mathsf{Oplax}_+$-monoidal category if and only if $(\mathsf R^{\mathsf{op}},\OpLax^{\mathsf{op}},\Psi^{\mathsf{op}},\psi^{\mathsf{op}},$ $\upsilon^{\mathsf{op}})$ is a $\mathsf{Lax}^+\mathsf{Oplax}^0$-monoidal category. 

Throughout, in a $\mathsf{Lax}_0\mathsf{Oplax}_+$-monoidal category we denote by $i$ the image of the single object of $\mathbbm 1$ under the functor $\oplax 0:\mathbbm 1 \to \mathsf R$ and we keep our earlier notation $\lw a \rw$ for the image of any object $a$ under the functor $\oplax 1: \mathsf R \to \mathsf R$; and $a_1 \OpLax \cdots \OpLax a_n$ for the image of an object $(a_1,\dots,a_n)$ under the functor $\oplax n:\mathsf R^n \to \mathsf R$.

Symmetrically to Proposition \ref{prop:Lax+Oplax0Cartesian}, the Cartesian product of $\mathsf{Lax}_0\mathsf{Oplax}_+$-monoidal carries a canonical $\mathsf{Lax}_0\mathsf{Oplax}_+$-monoidal structure.

\begin{definition}
\label{def:lax0oplax2-functor}
A {\em $\mathsf{Lax}_0\mathsf{Oplax}_+$-monoidal functor} $(\mathsf R,\OpLax,\Psi,\psi,\upsilon)\to (\mathsf R',\OpLax',\Psi',\psi',\upsilon')$ consists of
\begin{itemize}
\item a functor $G:\mathsf R \to \mathsf R'$
\item for all non-negative integers $n$, a natural transformation
$\oplax n'\cdot (G\cdots G) \to G\cdot \oplax n$
\end{itemize}
such that the diagrams of Definition \ref{def:lax2oplax0-functor} with reversed arrows commute.
\end{definition}

Analogously to Proposition \ref{prop:lax+oplax0-functor}, the composite of (composable) $\mathsf{Lax}_0\mathsf{Oplax}_+$-monoidal functors is again $\mathsf{Lax}^+\mathsf{Oplax}^0$-monoidal via the evident natural transformations; and the Cartesian product of $\mathsf{Lax}_0\mathsf{Oplax}_+$-monoidal functors is $\mathsf{Lax}^+\mathsf{Oplax}^0$-monoidal via the evident natural transformations too.

\begin{definition}
\label{def:lax0oplax2-nattr}
A {\em $\mathsf{Lax}_0\mathsf{Oplax}_+$-monoidal natural transformation} is a natural transformation for which the diagrams of Definition \ref{def:lax2oplax0-nattr} with reversed arrows commute.
\end{definition}

Symmetrically to Proposition \ref{prop:2catLax+Oplax0}, all of the composites, the Godement products, and the Cartesian products of (composable) $\mathsf{Lax}_0\mathsf{Oplax}_+$-monoidal natural transformations are again $\mathsf{Lax}_0\mathsf{Oplax}_+$-monoidal. Consequently, there is a strict monoidal 2-category $\mathsf{Lax}_0\mathsf{Oplax}_+$ of $\mathsf{Lax}_0\mathsf{Oplax}_+$-mon\-oid\-al categories, $\mathsf{Lax}_0\mathsf{Oplax}_+$-monoidal functors and $\mathsf{Lax}_0\mathsf{Oplax}_+$-monoidal natural transformations, admitting a strict monoidal forgetful 2-functor to the 2-category $\mathsf{Oplax}_+$ of Section \ref{sec:semigroup}.

\begin{definition}
A {\em monoid} in $(\mathsf R,\OpLax,\Psi,\psi,\upsilon)$ is a comonoid in $(\mathsf R^{\mathsf{op}},\OpLax^{\mathsf{op}},\Psi^{\mathsf{op}},\psi^{\mathsf{op}},\upsilon^{\mathsf{op}})$. That is, it consists of 
\begin{itemize}
\item a semigroup $(a,\mu)$ and
\item a morphism $\eta:i \to a$
\end{itemize}
rendering commutative the first (unitality) diagram below.
A {\em morphism of monoids} $(a,\mu,\eta) \to (a',\mu',\eta')$ is a morphism of semigroups $\phi: (a,\mu) \to (a',\mu')$ rendering commutative the second diagram too.
$$
\xymatrix{
i \OpLax a  \ar[d]_-{\eta \circ 1}
\ar[r]^-{\psi_{0,1}}  &
\lw a \rw  \ar[d]^-{\upsilon} &
\ar[l]_-{\psi_{1,0}} a \OpLax i  \ar[d]^-{1 \circ \eta} \\
a\OpLax a  \ar[r]_-\mu&
a    &
a\OpLax a  \ar[l]^-\mu}
\qquad \qquad
\xymatrix@R=27pt{
i \ar[r]^-\eta \ar@{=}[d] &
a \ar[d]^-\phi \\
i \ar[r]_-{\eta'} &
a'}
$$
\end{definition}

\begin{theorem}
\label{thm:monoid}
For any $\mathsf{Lax}_0\mathsf{Oplax}_+$-monoidal category $(\mathsf R,\OpLax,\Psi,\psi,\upsilon)$, the following categories are isomorphic.
\begin{itemize}
\item[{(i)}]
The category of monoids and their morphisms in $(\mathsf R,\OpLax,\Psi,\psi,\upsilon)$.
\item[{(ii)}]
The category of $\mathsf{Lax}_0\mathsf{Oplax}_+$-monoidal functors $(\mathbbm 1,1,1,1,1) \to (\mathsf R,\OpLax,\Psi,\psi,\upsilon)$ and their $\mathsf{Lax}_0\mathsf{Oplax}_+$-monoidal natural transformations. 
\end{itemize}
Consequently, $\mathsf{Lax}_0\mathsf{Oplax}_+$-monoidal functors preserve monoids.
\end{theorem}

\begin{proof}
Apply Theorem \ref{thm:comonoid} to the $\mathsf{Lax}^+\mathsf{Oplax}^0$-monoidal category $(\mathsf R^{\mathsf{op}},\OpLax^{\mathsf{op}},\Psi^{\mathsf{op}},\psi^{\mathsf{op}},$ $\upsilon^{\mathsf{op}})$.
\end{proof}

\begin{definition} \label{def:monoid_module}
A {\em module} of a monoid $(a,\mu,\eta)$ in a $\mathsf{Lax}_0\mathsf{Oplax}_+$-monoidal category $(\mathsf R,\OpLax,\Psi,\psi,\upsilon)$ is a comodule of the comonoid $(a,\mu,\eta)$ in the $\mathsf{Lax}^+\mathsf{Oplax}^0$-monoidal category $(\mathsf R^{\mathsf{op}},\OpLax^{\mathsf{op}},\Psi^{\mathsf{op}},\upsilon^{\mathsf{op}})$. Equivalently, it is a module $(x,\varrho:x\OpLax a \to x)$ in the sense of Definition \ref{def:module} of the semigroup $(a,\mu)$ such that also the diagram
$$
\xymatrix{
x \OpLax i \ar[r]^-{\psi_{1,0}} \ar[d]_-{1 \circ \eta} &
\lw x \rw \ar[d]^-\upsilon \\
x \OpLax a \ar[r]_-\varrho &
x}
$$
commutes.
A {\em morphism of modules} over a monoid is a morphism of modules over the underlying semigroup in the sense of Definition \ref{def:module}.
\end{definition}

\begin{theorem}
\label{thm:OpLax}
Associated to a monoidal category $(\mathsf V,\otimes,I)$, there is a strictly normal $\mathsf{Lax}_0\mathsf{Oplax}_+$-monoidal category $\mathsf R$ as follows. 
\begin{itemize}
\item As a category, $\mathsf R$ is identical to the category $\mathsf L$ of Theorem \ref{thm:Lax}.
\item For all non-negative integers $n$, we take the same functors $\raisebox{-2pt}{$
\stackrel {\displaystyle \otimes} {{}_n}$}:\mathsf R^n \to \mathsf R$ as in Theorem \ref{thm:Lax}.
\item For any sequence of non-negative integers $\{k_1,\dots,k_n\}$, and for any collection of objects $\{(a_{ij},\kappa_{ij},\nu_{ij})\}_{\stackrel{i=1,\dots,n}{{}_{j=1,\dots,k_i}}}$, the natural transformation 
$\Psi_{k_1,\dots,k_n}$ is taken to be
$$
\otimes_{i=1}^n(\otimes_{j=1}^{k_i} 
\kappa_{ij}^{\sum_{p=i+1}^n(\overline k_p-1)}
\nu_{ij}^{\sum_{p=1}^{i-1}(\overline k_p-1)}):
\otimes_{i=1}^n(\otimes_{j=1}^{k_i} a_{ij}) \to 
\otimes_{i=1}^n(\otimes_{j=1}^{k_i} a_{ij}),
$$ 
and the natural transformation 
$\psi_{k_1,\dots,k_n}$ is taken to be
$$
\otimes_{i=1}^n(\otimes_{j=1}^{k_i} 
\kappa_{ij}^{\sum_{p=i+1}^n Z(k_p)}
\nu_{ij}^{\sum_{p=1}^{i-1} Z(k_p)}):
\otimes_{i=1}^n(\otimes_{j=1}^{k_i} a_{ij}) \to 
\otimes_{i=1}^n(\otimes_{j=1}^{k_i} a_{ij}).
$$
\end{itemize}
\end{theorem}

\begin{proof}
Apply Theorem \ref{thm:Lax} to the monoidal category $(\mathsf V^{\mathsf{op}}, \otimes^{\mathsf{op}},I^{\mathsf{op}} )$; and take the opposite of the resulting $\mathsf{Lax}^+\mathsf{Oplax}^0$-monoidal category.
\end{proof}

In complete analogy with Theorem \ref{thm:comon_lax+fun} we get the following.

\begin{theorem}
\label{thm:mon_lax+fun}
For any monoidal category $(\mathsf V,\otimes,I)$ the following categories are pairwise isomorphic.
\begin{itemize}
\item[{(1.i)}] The category of semigroups in the $\mathsf{Oplax}_+$-monoidal category $\mathsf R$ of Theorem \ref{thm:OpLax}.
\item[{(1.ii)}] The category of \BiHom-monoids in $\mathsf V$.
\end{itemize}
and 
\begin{itemize}
\item[{(2.i)}] The category of monoids in the $\mathsf{Lax}_0\mathsf{Oplax}_+$-monoidal category $\mathsf R$ of Theorem \ref{thm:Lax}.
\item[{(2.ii)}] The category of unital \BiHom-monoids in $\mathsf V$.
\end{itemize}
\end{theorem}

A module in the sense of Definition \ref{def:module} over a semigroup in the 
$\mathsf{Oplax}_+$-monoidal category $\mathsf R$ of Theorem \ref{thm:OpLax} is precisely the same as a module in the sense of \cite[Definition 4.1]{GraMakMenPan} over the corresponding \BiHom-monoid in Theorem \ref{thm:mon_lax+fun}~(1.ii).
A module in the sense of Definition \ref{def:monoid_module} over a monoid in the $\mathsf{Lax}_0\mathsf{Oplax}_+$-monoidal category $\mathsf R$ of Theorem \ref{thm:OpLax} is precisely the same as an unital module in the sense of \cite[Definition 4.1]{GraMakMenPan} over the corresponding unital \BiHom-monoid in Theorem \ref{thm:mon_lax+fun}~(2.ii).

Applying Proposition \ref{prop:Yau} to the opposite categories, we obtain the following explanation of the Yau twist of semigroups and monoids in \cite[Proposition 5.9]{GraMakMenPan}.

\begin{proposition} \label{prop:YauDual}
Consider an arbitrary monoidal category $(\mathsf V,\otimes,I)$, the associated strictly normal $\mathsf{Lax}_0\mathsf{Oplax}_+$-monoidal category $\mathsf R$
of Theorem \ref{thm:OpLax}, and the monoidal category $\mathsf M$ of Proposition \ref{prop:Yau} regarded now as a $\mathsf{Lax}_0\mathsf{Oplax}_+$-monoidal category. The identity functor $\mathsf M \to \mathsf R$ carries a $\mathsf{Lax}_0\mathsf{Oplax}_+$-monoidal structure given by the same morphisms in Proposition \ref{prop:Yau}. Consequently, it preserves semigroups and monoids.
\end{proposition}

\begin{remark}
In the coinciding categories $\mathsf R$ of Theorem \ref{thm:OpLax} and $\mathsf L$ of Theorem \ref{thm:Lax} take the full monoidal subcategory described in Remark \ref{rem:invertible}.
The category of semigroups (resp. monoids) in it is isomorphic to the category whose objects are triples consisting of a semigroup (resp. monoid) in $\mathsf V$ together with two commuting semigroup (resp. monoid) automorphisms.
\end{remark}

\section{Bimonoids in $\mathsf{Lax}^+_0\mathsf{Oplax}^0_+$-duoidal categories}

In this section we introduce compatibility conditions between $\mathsf{Lax}^+\mathsf{Oplax}^0$-monoidal, and $\mathsf{Lax}_0\mathsf{Oplax}_+$-monoidal structures on the same category. Similarly to Section \ref{sec:bisgr}, they imply that the category of comonoids w.r.t. the $\mathsf{Lax}^+\mathsf{Oplax}^0$-monoidal structure inherits the $\mathsf{Lax}_0\mathsf{Oplax}_+$-monoidal structure; symmetrically, the category of monoids w.r.t. the $\mathsf{Lax}_0\mathsf{Oplax}_+$-monoidal structure inherits the $\mathsf{Lax}^+\mathsf{Oplax}^0$-monoidal structure. We define then bimonoids as monoids in the category of comonoids; equivalently, as comonoids in the category of monoids. Unital and counital \BiHom-bimonoids are described as bimonoids in this sense.

\begin{definition}
\label{def:+0duoidal}
A {\em $\mathsf{Lax}^+_0\mathsf{Oplax}^0_+$-duoidal category} consists of
\begin{itemize}
\item a category $\mathsf D$
\item a $\mathsf{Lax}^+\mathsf{Oplax}^0$-monoidal structure $(\mathsf D,\Lax,\Phi,\phi,\iota)$
\item a $\mathsf{Lax}_0\mathsf{Oplax}_+$-monoidal structure $(\mathsf D,\OpLax,\Psi,\psi,\upsilon)$
\item for all non-negative integers $n$ and $p$, natural transformations
$$
\xymatrix@C=60pt{
\mathsf D^{np} \ar[r]^-{\tau_{np}} 
\ar[d]_-{\oplax n \cdots \oplax n}
\ar@{}[rrd]|-{\Longdownarrow {\xi_n^p}} &
\mathsf D^{pn} \ar[r]^-{\lax p \cdots \lax p} &
\mathsf D^{n} \ar[d]^-{\oplax n} \\
\mathsf D^{p} \ar[rr]_-{\lax p} &&
\mathsf D}
$$
\end{itemize}
satisfying the following equivalent compatibility conditions, for all non-negative integers $n,p,k_1,\dots,k_p$.
\begin{itemize}
\item[{(i)}] 
\begin{itemize}
\item[$\bullet$]
$(\oplax n,\xi_n)$ is a $\mathsf{Lax}^+\mathsf{Oplax}^0$-monoidal functor $(\mathsf D,\Lax,\Phi,\phi,\iota)^n \to(\mathsf D,\Lax,\Phi,\phi,\iota);$ 
\item[$\bullet$]
$\Psi_{k_1,\dots,k_p}$, $\psi_{k_1,\dots,k_p}$ and $\upsilon$ are $\mathsf{Lax}^+\mathsf{Oplax}^0$-monoidal natural transformations. 
\end{itemize}
\noindent
Succinctly, $((\mathsf D,\Lax,\Phi,\phi,\iota),(\OpLax,\xi),\Psi,\psi,\upsilon)$ is a $\mathsf{Lax}_0\mathsf{Oplax}_+$-monoid in the strict monoidal 2-category $\mathsf{Lax}^+\mathsf{Oplax}^0$ of Proposition \ref{prop:2catLax+Oplax0}.
\item[{(ii)}]
\begin{itemize}
\item[$\bullet$]
$(\lax n,\xi^n)$ is a $\mathsf{Lax}_0\mathsf{Oplax}_+$-monoidal functor $(\mathsf D,\OpLax,\Psi,\psi,\upsilon)^n \to (\mathsf D,\OpLax,\Psi,\psi,\upsilon);$ 
\item[$\bullet$]
$\Phi_{k_1,\dots,k_p}$, $\phi_{k_1,\dots,k_p}$ and $\iota$ are $\mathsf{Lax}_0\mathsf{Oplax}_+$-monoidal natural transformations.
\end{itemize}
\noindent
Succinctly, $((\mathsf D,\OpLax,\Psi,\psi,\upsilon),(\Lax,\xi),\Phi,\phi,\iota)$ is a $\mathsf{Lax}^+\mathsf{Oplax}^0$-monoid in the strict monoidal 2-category $\mathsf{Lax}_0\mathsf{Oplax}_+$ of Section \ref{sec:Lax0Oplax+}.
\item[{(iii)}]
The diagrams
\begin{amssidewaysfigure}
\centering
\xymatrix@C=25pt@R=25pt{
\\
\\
\oplax n \cdot (\lax p \cdots \lax p) \cdot
(\raisebox{-1.5pt}{$\lax {k_1} \cdots \lax {k_p}{\cdots } \lax {k_1} \cdots \lax {k_p}$}) \cdot
\tau_{nK}
\ar@{=}[d]
\ar[r]^-{1\cdot (\Phi_{k_1,\dots,k_p}\cdots \Phi_{k_1,\dots,k_p})\cdot 1} &
\oplax n \cdot(\raisebox{-1.5pt}{$\lax {K+Z} \cdots \lax {K+Z}$} ) \cdot
([k_1]^\bullet \cdots [k_p]^\bullet \cdots [k_1]^\bullet \cdots [k_p]^\bullet) \cdot \tau_{nK} \ar@{=}[d] &
\oplax n \cdot(\raisebox{-1.5pt}{$\lax {K} \cdots \lax {K} $}) \cdot \tau_{nK} 
\ar[l]_-{\raisebox{8pt}{${}_{
1\cdot (\phi_{k_1,\dots,k_p}\cdots \phi_{k_1,\dots,k_p})\cdot 1}$}} 
\ar[ddd]^-{\xi^K_n} \\
\oplax n \cdot (\lax p \cdots \lax p) \cdot \tau_{np} \cdot
(\raisebox{-1.5pt}{$\lax {k_1} \cdots \lax {k_1} \cdots \lax {k_p} \cdots \lax {k_p}$}) \cdot (\tau_{nk_1} \cdots \tau_{nk_p}) 
\ar[d]_-{\xi^p_n\cdot 1 \cdot 1} &
\oplax n \cdot(\raisebox{-1.5pt}{$\lax {K+Z} \cdots \lax {K+Z}$}) \cdot \tau_{n(K+Z)}
\cdot ([k_1]^\bullet \cdots [k_1]^\bullet \cdots [k_p]^\bullet \cdots [k_p]^\bullet)
\ar[d]^-{\xi^{K+Z}_n \cdot 1}  \\
\lax p \cdot (\oplax n \cdots \oplax n) \cdot 
(\raisebox{-1.5pt}{$\lax {k_1} \cdots \lax {k_1} \cdots \lax {k_p} \cdots \lax {k_p}$}) \cdot (\tau_{nk_1} \cdots \tau_{nk_p}) 
\ar[d]_-{1\cdot (\xi^{k_1}_n \cdots \xi^{k_p}_n)} &
\raisebox{-1.5pt}{$\lax {K+Z}$} \cdot (\oplax n \cdots \oplax n) \cdot 
([k_1]^\bullet \cdots [k_1]^\bullet \cdots [k_p]^\bullet \cdots [k_p]^\bullet)
\ar[d]^-{1\cdot (\lsem k_1 \rsem^\bullet_{n} \cdots \lsem k_p \rsem^\bullet_{n} )} \\
\lax p \cdot (\raisebox{-1.5pt}{$\lax {k_1} \cdots \lax {k_p}$}) 
\cdot (\oplax n \cdots \oplax n) 
\ar[r]_-{\Phi_{k_1,\dots, k_p}\cdot 1} &
\raisebox{-1.5pt}{$\lax{K+Z}$} \cdot ([k_1]^\bullet \cdots [k_p]^\bullet) \cdot 
(\oplax n \cdots \oplax n) &
\ar[l]^-{\phi_{k_1,\dots, k_p}\cdot 1}
\raisebox{-1.5pt}{$\lax K$} \cdot (\oplax n \cdots \oplax n) \\
\\
\lax n \cdot (\oplax p \cdots \oplax p) \cdot
(\raisebox{-1.5pt}{$\oplax {k_1} \cdots \oplax {k_p}\cdots  \oplax {k_1} \cdots \oplax {k_p}$}) &
\lax n \cdot(\raisebox{-1.5pt}{$\oplax {K+Z} \cdots \oplax {K+Z}$} ) \cdot
([k_1]^\circ \cdots [k_p]^\circ \cdots [k_1]^\circ \cdots [k_p]^\circ) 
\ar[l]_-{1\cdot (\Psi_{k_1,\dots,k_p}\cdots \Psi_{k_1,\dots,k_p})} 
\ar[r]^-{\raisebox{8pt}{${}_{
1\cdot (\psi_{k_1,\dots,k_p}\cdots \psi_{k_1,\dots,k_p})}$}}
&
\lax n \cdot(\raisebox{-1.5pt}{$\oplax {K} \cdots \oplax {K}$} )  \\
\ar[u]^-{\xi^n_p\cdot 1}
\oplax p \cdot (\lax n \cdots \lax n) \cdot \tau_{pn} \cdot
(\raisebox{-1.5pt}{$\oplax {k_1} \cdots \oplax {k_p} \cdots \oplax {k_1} \cdots \oplax {k_p}$}) 
\ar@{=}[d] &
\ar[u]_-{\xi^n_{K+Z} \cdot 1}
\raisebox{-1.5pt}{$\oplax {K+Z}$} \cdot (\lax n \cdots \lax n) \cdot \tau_{(K+Z)n} \cdot ([k_1]^\circ \cdots [k_p]^\circ \cdots [k_1]^\circ \cdots [k_p]^\circ)
\ar@{=}[d]  \\
\oplax p \cdot (\lax n \cdots \lax n)  \cdot
(\raisebox{-1.5pt}{$\oplax {k_1} \cdots \oplax {k_1} \cdots \oplax {k_p} \cdots \oplax {k_p}$}) \cdot {\tau_{pn}} &
\raisebox{-1.5pt}{$\oplax {K+Z}$} \cdot (\lax n \cdots \lax n) \cdot 
([k_1]^\circ \cdots [k_1]^\circ \cdots [k_p]^\circ \cdots [k_p]^\circ)
\cdot \tau_{Kn} \\
\ar[u]^-{1\cdot (\xi_{k_1}^n \cdots \xi_{k_p}^n)\cdot 1}
\oplax p \cdot (\raisebox{-1.5pt}{$\oplax {k_1} \cdots \oplax {k_p}$}) 
\cdot (\lax n \cdots \lax n) \cdot  
{(\tau_{k_1n} \dots \tau_{k_pn}) \cdot \tau_{pn}}
&
\ar[u]_-{1\cdot (\lsem k_1 \rsem^\circ_{n} \cdots \lsem k_p \rsem^\circ_{n})\cdot 1} \raisebox{-1.5pt}{$\oplax{K+Z}$}\cdot ([k_1]^\circ  \cdots [k_p]^\circ) \cdot 
(\lax n \cdots \lax n) \cdot \tau_{Kn} 
\ar[l]^-{\Psi_{k_1,\dots, k_p}\cdot 1 \cdot 1}
\ar[r]_-{\psi_{k_1,\dots, k_p}\cdot 1 \cdot 1} &
\raisebox{-1.5pt}{$\oplax K$} \cdot (\lax n \cdots \lax n) \cdot \tau_{Kn}
\ar[uuu]_-{\xi_K^n}}
\caption{Axioms of $\mathsf{Lax}^+_0\mathsf{Oplax}^0_+$-duoidal category}
\label{fig:duoidal}
\end{amssidewaysfigure}
$$ 
\xymatrix@C=40pt{
\raisebox{-1.5pt}{$\oplax 1$} \cdot \lax n \ar[rd]^-{\upsilon \cdot 1} 
\ar[d]_-{\xi^n_1} \\
\lax n \cdot (\raisebox{-1.5pt}{$\oplax 1 \cdots \oplax 1$}) 
\ar[r]_-{1\cdot (\upsilon \cdots \upsilon)} &
\lax n}
\qquad \qquad
\xymatrix@C=40pt{
\oplax n \ar[r]^-{1\cdot (\iota \cdots \iota)} \ar[rd]_-{\iota\cdot 1} &
\oplax n \cdot (\raisebox{-1.5pt}{$\lax 1 \cdots \lax 1$}) \ar[d]^-{\xi^1_n} \\
& \raisebox{-1.5pt}{$\lax 1$} \cdot \oplax n}
$$
and those of Figure \ref{fig:duoidal} commute, for all non-negative integers $n,p,k_1,\dots,k_p$, for the families of functors
\begin{eqnarray*}
&&
\xymatrix@C=17pt{\mathsf D^{k_i} \ar[r]^-{[ k_i ]^\bullet} & \mathsf D^{\overline k_i}}:=
\left\{
\arraycolsep=1pt\def\arraystretch{.5}
\begin{array}{ll}
\xymatrix@C=17pt{\mathsf D^{k_i} \ar[r]^-1 & \mathsf D^{k_i}} & \textrm{if}\quad k_i>0 \\
\xymatrix@C=27pt{ \mathbbm 1 \ar[r]^-{\scalebox{.7}{$\lax 0$}} & \mathsf D}
& \textrm{if}\quad k_i=0
\end{array}
\right .
\\
&& 
\xymatrix@C=17pt{\mathsf D^{k_i} \ar[r]^-{[k_i ]^\circ}& \mathsf D^{\overline k_i}}:=
\left\{
\arraycolsep=1pt\def\arraystretch{.5}
\begin{array}{ll}
\xymatrix@C=17pt{\mathsf D^{k_i} \ar[r]^-1 & \mathsf D^{k_i}} & \textrm{if}\quad k_i>0 \\
\xymatrix@C=27pt{ \mathbbm 1 \ar[r]^-{\scalebox{.7}{$\oplax 0$}} & \mathsf D}
& \textrm{if}\quad k_i=0
\end{array}
\right .
\end{eqnarray*}
and for the families of natural transformations
\begin{eqnarray*}
&&
\xymatrix@C=17pt{ 
\oplax n^{\overline k_i} \cdot ([k_i]^\bullet)^n  
\ar[r]^-{\lsem k_i \rsem^\bullet_{n}} &
[k_i]^\bullet \cdot \oplax n^{k_i}}:=
\left\{
\begin{array}{ll}
\arraycolsep=1pt\def\arraystretch{.5}
\xymatrix@C=12pt{\hspace*{.52cm} \oplax n ^{k_i} \ar[r]^-1 & \oplax n ^{k_i}}
& \textrm{if}\quad k_i>0 \\
\xymatrix@C=12pt{
\oplax n\cdot  \raisebox{-1.5pt}{$ \lax 0^n$}
\ar[r]^-{\xi^0_n} &  \raisebox{-1.5pt}{$\lax 0$}}
& \textrm{if}\quad k_i=0
\end{array}
\right .
\\
&&
\xymatrix@C=17pt{
[k_i]^\circ \cdot \lax n^{k_i} \ar[r]^-{ \lsem k_i \rsem^\circ_{n}} & \lax n^{\overline k_i} \cdot ([k_i]^\circ)^n}:=
\left\{
\begin{array}{ll}
\arraycolsep=1pt\def\arraystretch{.5}
\xymatrix@C=12pt{ \lax n ^{k_i} \ar[r]^-1 & \lax n ^{k_i}}
& \textrm{if}\quad k_i>0 \\
\xymatrix@C=12pt{ \raisebox{-1.5pt}{$\oplax 0$} \hspace*{.3cm} \ar[r]^-{\xi_0^n} & \lax n \cdot \raisebox{-1.5pt}{$\oplax 0^n$}}
& \textrm{if}\quad k_i=0.
\end{array}
\right .
\end{eqnarray*}
\end{itemize}
\end{definition}

Forgetting about the irrelevant structure, any $\mathsf{Lax}^+_0\mathsf{Oplax}^0_+$-duoidal category can be regarded as a $\mathsf{Lax}^+\mathsf{Oplax}_+$-duoidal category.

It is immediate from Definition \ref{def:+0duoidal}~(i) and Theorem \ref{thm:comonoid} that for any $\mathsf{Lax}^+_0\mathsf{Oplax}^0_+$-duoidal category $(\mathsf D,\Lax,\OpLax)$, $(i,\xi^2_0,\xi^0_0)$ is a comonoid in $(\mathsf D,\Lax)$. Dually, it follows from Definition \ref{def:+0duoidal}~(ii) and Theorem \ref{thm:monoid} that $(j,\xi_2^0,\xi^0_0)$ is a monoid in $(\mathsf D,\OpLax)$.

\begin{definition} \label{def:Lax+0Oplax0+functor}
A {\em $\mathsf{Lax}^+_0\mathsf{Oplax}^0_+$-duoidal functor} $(\mathsf D,\Lax,\OpLax) \to (\mathsf D',\Lax',\OpLax')$ is defined as
\begin{itemize}
\item[{(i)}] a $\mathsf{Lax}_0\mathsf{Oplax}_+$-monoidal 1-cell 
$((G,\Gamma^\bullet),\Gamma^\circ):((\mathsf D,\Lax),\OpLax) \to ((\mathsf D',\Lax'),\OpLax')$
in the strict monoidal 2-category $\mathsf{Lax}^+\mathsf{Oplax}^0$ of Proposition \ref{prop:2catLax+Oplax0}; equivalently,
\item[{(ii)}] A $\mathsf{Lax}^+\mathsf{Oplax}^0$-monoidal 1-cell 
$((G,\Gamma^\circ),\Gamma^\bullet):((\mathsf D,\OpLax),\Lax) \to ((\mathsf D',\OpLax'),\Lax')$ in the strict monoidal 2-category $\mathsf{Lax}_0\mathsf{Oplax}_+$ of Section \ref{sec:Lax0Oplax+}; equivalently,
\item[{(iii)}]
\begin{itemize}
\item[{$\bullet$}] a functor $G: \mathsf D \to \mathsf D'$
\item[{$\bullet$}] a $\mathsf{Lax}^+\mathsf{Oplax}^0$-monoidal structure $\{\Gamma^\bullet_n:G\cdot \lax n \to \lax n'\cdot (G \cdots G)\}_{n \in \mathbb N}$
\item[{$\bullet$}] a $\mathsf{Lax}_0\mathsf{Oplax}_+$-monoidal structure $\{\Gamma^\circ_p: \oplax p'\cdot (G \cdots G) \to G\cdot \oplax p \}_{p\in \mathbb N}$
\end{itemize}

\noindent
rendering commutative the diagram of \eqref{eq:Lax+Oplax+functor} for all non-negative integers $n,p$.
\end{itemize}
\end{definition}

Forgetting about the nullary parts of the structures, $\mathsf{Lax}^+_0\mathsf{Oplax}^0_+$-duoidal functors can be seen as $\mathsf{Lax}^+\mathsf{Oplax}_+$-duoidal functors.

It follows immediately from Definition \ref{def:Lax+0Oplax0+functor}~{(ii)} and Proposition \ref{prop:lax+oplax0-functor}~(1) that the composite of $\mathsf{Lax}^+_0\mathsf{Oplax}^0_+$-duoidal functors is again $\mathsf{Lax}^+_0\mathsf{Oplax}^0_+$-duoidal via the composite $\mathsf{Lax}^+\mathsf{Oplax}^0$-monoidal structure in Proposition \ref{prop:lax+oplax0-functor}~(1) and the symmetrically constructed composite $\mathsf{Lax}_0\mathsf{Oplax}_+$-monoidal structure.

\begin{definition} \label{def:Lax+0Oplax0+nattr}
A {\em $\mathsf{Lax}^+_0\mathsf{Oplax}^0_+$-duoidal natural transformation} is defined as 
\begin{itemize}
\item[{(i)}] a $\mathsf{Lax}_0\mathsf{Oplax}_+$-monoidal 2-cell in the strict monoidal 2-category $\mathsf{Lax}^+\mathsf{Oplax}^0$ of Proposition \ref{prop:2catLax+Oplax0}; equivalently,
\item[{(ii)}] a $\mathsf{Lax}^+\mathsf{Oplax}^0$-monoidal 2-cell in the strict monoidal 2-category $\mathsf{Lax}_0\mathsf{Oplax}_+$ of Section \ref{sec:Lax0Oplax+}; equivalently,
\item[{(iii)}] a natural transformation which is both $\mathsf{Lax}_0\mathsf{Oplax}_+$-monoidal and $\mathsf{Lax}^+\mathsf{Oplax}^0$-monoidal. That is, it renders commutative both the diagram of Definition \ref{def:lax2oplax0-nattr} and its symmetric counterpart with reversed arrows and inverted colors.
\end{itemize}
\end{definition}

$\mathsf{Lax}^+_0\mathsf{Oplax}^0_+$-duoidal natural transformations are in particular $\mathsf{Lax}^+\mathsf{Oplax}_+$-duoidal.

\begin{theorem}
\label{thm:duoidal(co)monoid}
For any $\mathsf{Lax}^+_0\mathsf{Oplax}^0_+$-duoidal category $(\mathsf D,(\Lax,\Phi,\phi,\iota),(\OpLax,\Psi,\psi,\upsilon),\xi)$, the following assertions hold.
\begin{itemize}
\item[{(1)}] The comonoids in the $\mathsf{Lax}^+\mathsf{Oplax}^0$-monoidal category $(\mathsf D,\Lax,\Phi,\phi,\iota)$ constitute a $\mathsf{Lax}_0\mathsf{Oplax}_+$-monoidal category with the structure $(\OpLax,\Psi,\psi,\upsilon)$.
\item[{(2)}] The monoids in the $\mathsf{Lax}_0 \mathsf{Oplax}_+$-monoidal category $(\mathsf D,\OpLax,\Psi,\psi,\upsilon)$ constitute a $\mathsf{Lax}^+\mathsf{Oplax}^0$-monoidal category with the structure $(\Lax,\Phi,\phi,\iota)$.
\end{itemize}
\end{theorem}

\begin{proof}
In exactly the same way as we derived Theorem \ref{thm:duoidal(co)semigroup} from Theorem \ref{thm:cosemigroup} and Definition \ref{def:++duoidal}~(ii),
part (1) follows from Theorem \ref{thm:comonoid} and Definition \ref{def:+0duoidal}~(ii). Part (2) follows symmetrically.
\end{proof}

\begin{definition}
\label{def:+0bimonoid}
A {\em bimonoid} in a $\mathsf{Lax}^+_0\mathsf{Oplax}^0_+$-duoidal category $(\mathsf D,\Lax,\OpLax,\xi)$ is defined by the following equivalent data.
\begin{itemize}
\item[{(i)}] A monoid in the $\mathsf{Lax}_0\mathsf{Oplax}_+$-monoidal category of comonoids in $(\mathsf D,\Lax)$ --- see Theorem \ref{thm:duoidal(co)monoid}~(1).
\item[{(ii)}] A comonoid in the $\mathsf{Lax}^+\mathsf{Oplax}^0$-monoidal category of monoids in $(\mathsf D,\OpLax)$ --- see Theorem \ref{thm:duoidal(co)monoid}~(2).
\item[{(iii)}]
\begin{itemize}
\item[$\bullet$] An object $a$ of $\mathsf D$,
\item[$\bullet$] a monoid $(a,\mu,\eta)$ in $(\mathsf D,\OpLax)$
\item[$\bullet$] a comonoid $(a,\delta,\varepsilon)$ in $(\mathsf D,\Lax)$
\end{itemize}
\noindent
such that $(a,\mu,\delta)$ is a bisemigroup in the $\mathsf{Lax}^+\mathsf{Oplax}_+$-duoidal category $(\mathsf D,\Lax,\OpLax,$ ${\xi})$
and the following 
diagrams commute as well.
$$
\xymatrix{
a\OpLax a\ar[r]^-\mu 
\ar[d]_-{\varepsilon \circ \varepsilon} &
a \ar[d]^-\varepsilon \\
j\OpLax j \ar[r]_-{\xi^0_2} &
j} 
\qquad
\xymatrix{
i \ar[r]^-{\xi^2_0} \ar[d]_-\eta &
i\Lax i \ar[d]^-{\eta \bullet \eta} \\
a\ar[r]_-\delta &
a\Lax a}
\qquad 
\xymatrix{
i \ar[r]^-\eta \ar[rd]_-{\xi^0_0} &
a\ar[d]^-\varepsilon \\
& j}
$$
\end{itemize}
A {\em morphism of bimonoids} is a monoid morphism in the category of comonoids in $(\mathsf D,\Lax)$; equivalently, a comonoid morphism in the category of monoids in $(\mathsf D,\OpLax)$; equivalently, a morphism in $\mathsf D$ which is both a comonoid morphism in $(\mathsf D,\Lax)$ and a monoid morphism in $(\mathsf D,\OpLax)$.
\end{definition}

\begin{theorem}
For any $\mathsf{Lax}^+_0\mathsf{Oplax}^0_+$-duoidal category $(\mathsf D,\Lax,\OpLax,\xi)$ the following categories are isomorphic.
\begin{itemize}
\item[{(i)}] The category of bimonoids and their morphisms in $(\mathsf D,\Lax,\OpLax,\xi)$.
\item[{(ii)}] The category of $\mathsf{Lax}^+_0\mathsf{Oplax}^0_+$-duoidal functors $(\mathbbm 1,1,1,1)\to (\mathsf D,\Lax,\OpLax,\xi)$ and their $\mathsf{Lax}^+_0\mathsf{Oplax}^0_+$-duoidal natural transformations.
\end{itemize}
Consequently, $\mathsf{Lax}^+_0\mathsf{Oplax}^0_+$-duoidal functors preserve bimonoids.
\end{theorem}

\begin{proof}
By Definition \ref{def:+0bimonoid}, the category of part (i) is isomorphic to the category of monoids in the category of comonoids in $(\mathsf D,\Lax)$.
So by Theorem \ref{thm:monoid}, it is isomorphic to the category of $\mathsf{Lax}_0\mathsf{Oplax}_+$-monoidal functors and $\mathsf{Lax}_0\mathsf{Oplax}_+$-monoidal natural transformations from $\mathbbm 1$ to the category of comonoids in $(\mathsf D,\Lax)$.
Then by Theorem \ref{thm:comonoid}, it is further equivalent to the category of $\mathsf{Lax}_0\mathsf{Oplax}_+$-monoidal 1-cells and $\mathsf{Lax}_0\mathsf{Oplax}_+$-monoidal 2-cells from $(\mathbbm 1,1)$ to $(\mathsf D,\Lax)$ in $\mathsf{Lax}^+\mathsf{Oplax}^0$.
By Definition \ref{def:Lax+0Oplax0+functor} and Definition \ref{def:Lax+0Oplax0+nattr} this is isomorphic to the category of part (ii).
\end{proof}

Analogously to Theorem \ref{thm:mod_monoidal}, we have the following.
\begin{theorem} \label{thm:monoid_mod_monoidal}
For any bimonoid $(a,(\delta,\varepsilon),(\mu,\eta))$ in a $\mathsf{Lax}^+_0\mathsf{Oplax}^0_+$-duoidal category 
$(\mathsf D,(\Lax,\Phi,\phi,\iota),(\OpLax,\Psi,\psi,\upsilon),\xi)$
the following assertions hold.
\begin{itemize}
\item[{(1)}] The category of modules --- in the sense of Definition \ref{def:monoid_module} --- over the underlying monoid $(a,\mu,\eta)$ in the $\mathsf{Lax}_0\mathsf{Oplax}_+$-monoidal category $(\mathsf D,\OpLax,\Psi,\psi,\upsilon)$ admits the $\mathsf{Lax}^+\mathsf{Oplax}^0$-monoidal structure $(\Lax,\Phi,\phi,\iota)$.
\item[{(2)}] The category of comodules --- in the sense of Definition \ref{def:comonoid_comodule} --- over the underlying comonoid $(a,\delta,\varepsilon)$ in the $\mathsf{Lax}^+\mathsf{Oplax}^0$-monoidal category $(\mathsf D,\Lax,\Phi,\phi,\iota)$ admits the $\mathsf{Lax}_0\mathsf{Oplax}_+$-monoidal structure $(\OpLax,\Psi,\psi,\upsilon)$.
\end{itemize}
\end{theorem}

\begin{proof}
We only prove part (1), part (2) is verified symmetrically.

For any sequence of $(a,\mu,\eta)$-modules $\{(x_i,\varrho_i)\}_{i=1,\dots,n}$, 
the object $x_1 \Lax \cdots \Lax x_n$ with the action
$$
\xymatrix@C=20pt{
(x_1 \Lax \cdots \Lax x_n) \OpLax a
\ar[r]^-{1\circ \delta_n} &
(x_1 \Lax \cdots \Lax x_n) \OpLax a^{\bullet n}
\ar[r]^-{\xi^n_2} &
(x_1 \OpLax a) \Lax \cdots \Lax (x_n \OpLax a)
\ar[rr]^-{\varrho_1 \bullet \cdots \bullet \varrho_n} &&
x_1 \Lax \cdots \Lax x_n}
$$
 is an $(a,\mu)$-module --- in the sense of Definition \ref{def:module} --- by Theorem \ref{thm:mod_monoidal}. 

Let us check its unitality whenever all of the actions $\varrho_i$ are unital.
The diagram on the right of part (iii) of Definition \ref{def:+0duoidal} for $n=0$ says the equality of $\iota: i \to \lb i \rb$ and $\xi^1_0$.
By Definition \ref{def:+0duoidal}~(i), $(\oplax 0,\xi_0):\mathbbm 1 \to \mathsf D$ is a $\mathsf{Lax}^+\mathsf{Oplax}^0$-monoidal functor. Equivalently, by Theorem \ref{thm:comonoid}, $(i,\xi^2_0,\xi^0_0)$ is a comonoid in $(\mathsf D,\Lax)$.
With these facts at hand,
from the commutativity of the left half of the upper diagram of Figure \ref{fig:duoidal} at $n=0$, $p=2$, $k_1=1$ and $k_2=l$, the $l$-fold comultiplication $i\to i^{\bullet l}$ as in \eqref{eq:delta_k+1} comes out as $\xi^l_0$. 
By Definition \ref{def:+0bimonoid}, $\eta$ is a comonoid morphism. Therefore the leftmost region of 
$$
\scalebox{.95}{$
\xymatrix@C=8pt{
(x_1\Lax \cdots \Lax x_n)\OpLax i
\ar[rrrr]^-{\psi_{1,0}}
\ar[rd]^-{1\circ \xi^n_0} 
\ar[dd]_-{1\circ \eta} &&&&
\lw x_1\Lax \cdots \Lax x_n \rw
\ar[dd]^-\upsilon \ar[ld]_-{\xi^n_1} \\
& (x_1\Lax \cdots \Lax x_n) \OpLax i^{\bullet n} 
\ar[r]^-{\raisebox{8pt}{${}_{\xi^n_2}$}}
\ar[d]^-{1\circ \eta^{\bullet n}} &
(x_1 \OpLax i) \Lax \cdots \Lax (x_n \OpLax i)
\ar[r]^-{\raisebox{8pt}{${}_{\psi_{1,0} \bullet \cdots \bullet \psi_{1,0}}$}}
\ar[d]^-{(1\circ \eta)\Lax \cdots \Lax (1\circ \eta)} &
\lw x_1 \rw \Lax \cdots \Lax \lw x_n \rw 
\ar[rd]_-{\upsilon \Lax \cdots \Lax \upsilon} \\
(x_1\Lax \cdots \Lax x_n) \OpLax a
\ar[r]_-{1\circ \delta_n} &
(x_1\Lax \cdots \Lax x_n) \OpLax a^{\bullet n} 
\ar[r]_-{\xi^n_2} &
(x_1 \OpLax a) \Lax \cdots \Lax (x_n \OpLax a)
\ar[rr]_-{\varrho_1 \Lax \cdots \Lax \varrho_n} &&
x_1\Lax \cdots \Lax x_n.}$}
$$
commutes.
The triangular region on the right commutes by the diagram on the left in part (iii) of Definition \ref{def:+0duoidal}; and the large pentagonal region at the top commutes by the right half of the lower diagram of Figure \ref{fig:duoidal} for the values $n$, $p=2$, $k_1=1$ and $k_2=0$. The quadratic region at the bottom right commutes by the unitality of each action $\varrho_i$, and the region on its left commutes by the naturality of $\xi^n_2$.
This proves the unitality of the action in the bottom row.

The $\Lax$-monoidal products of module morphisms, as well as the natural transformations $\iota$ and $\Phi_{k_1,\dots, k_p}$ for all $k_i>0$ --- if evaluated at $(a,\mu,\eta)$-modules --- are morphisms of $(a,\mu)$-modules by Theorem \ref{thm:mod_monoidal}. Then they are morphisms of $(a,\mu,\eta)$-modules by definition. 
We leave it to the reader to check --- similarly to the proof of Theorem \ref{thm:mod_monoidal} --- that the remaining natural transformations $\phi$ and $\Phi_{k_1,\dots, k_p}$ with possibly zero values of $k_i$ --- if evaluated at $(a,\mu,\eta)$-modules --- are morphisms of $(a,\mu,\eta)$-modules too. 
\end{proof}

\begin{proposition} \label{prop:a_monoid}
For any bimonoid $(a,(\delta,\varepsilon),(\mu,\eta))$ in an arbitrary $\mathsf{Lax}^+_0\mathsf{Oplax}_+^0$-duoidal category the following assertions hold.
\begin{itemize}
\item[{(1)}] $((a,\delta),\mu,\eta)$ is a monoid in the $\mathsf{Lax}_0\mathsf{Oplax}_+$-monoidal category of comodules over the comonoid $(a,\delta,\varepsilon)$ (cf. Theorem \ref{thm:monoid_mod_monoidal}~(2)).
\item[{(2)}] $((a,\mu),\delta,\varepsilon)$ is a comonoid in the $\mathsf{Lax}^+\mathsf{Oplax}^0$-monoidal category of modules over the monoid $(a,\mu,\eta)$ (cf. Theorem \ref{thm:monoid_mod_monoidal}~(1)).
\end{itemize}
\end{proposition}

\begin{proof}
By Example \ref{ex:comonoid_reg_comod} $(a,\delta)$ is a comodule and by Proposition \ref{prop:Hopfmod} $\mu$ is a morphism of comodules.
Also $\eta$ is a morphism of comodules by the second diagram of Definition \ref{def:+0bimonoid}~(iii) which completes the proof of part (1).
Part (2) follows symmetrically, using now the first diagram of Definition \ref{def:+0bimonoid}~(iii).
\end{proof}

\begin{definition}
A {\em Hopf module} over a bimonoid $(a,(\delta,\varepsilon),(\mu,\eta))$ in an arbitrary $\mathsf{Lax}^+_0\mathsf{Oplax}_+^0$-duoidal category $(\mathsf D,\Lax,\OpLax,\xi)$ is defined by the following equivalent data.
\begin{itemize} 
\item[{(i)}] A module over the monoid $((a,\delta),\mu,\eta)$ in the category of comodules over the comonoid $(a,\delta,\varepsilon)$ in $(\mathsf D,\Lax)$ (cf. Proposition \ref{prop:a_monoid}~(1)).
\item[{(ii)}] A comodule over the comonoid $((a,\mu),\delta,\varepsilon)$ in the category of modules over the monoid $(a,\mu,\eta)$ in $(\mathsf D,\OpLax)$ (cf. Proposition \ref{prop:a_monoid}~(2)).
\item[{(iii)}]
\begin{itemize}
\item[$\bullet$] An object $x$ of $\mathsf D$,
\item[$\bullet$] a module $(x,\nu)$ over the monoid $(a,\mu,\eta)$ in $(\mathsf D,\OpLax)$,
\item[$\bullet$] a comodule $(x,\varrho)$ over the comonoid $(a,\delta,\varepsilon)$ in $(\mathsf D,\Lax)$,
\end{itemize}
rendering commutative the diagram of Definition \ref{def:Hopf_module}~(iii).
\end{itemize}
\end{definition}

\begin{example}
By Example \ref{ex:comonoid_reg_comod} and its dual counterpart, and by the compatibility diagram of Definition \ref{def:bimonoid}~(iii), $(a,\delta,\mu)$ is a Hopf module over an arbitrary bimonoid $(a,(\delta,\varepsilon),(\mu,\eta))$ in any $\mathsf{Lax}^+_0\mathsf{Oplax}_+^0$-duoidal category.
\end{example}

Our final task is a description of unital and counital \BiHom-bimonoids as bimonoids in a suitable $\mathsf{Lax}^+_0\mathsf{Oplax}^0_+$-duoidal category.
\black

\begin{theorem} \label{thm:DuoLax}
Associated to any symmetric monoidal category $(\mathsf V,\otimes,I,\sigma)$, there is a $\mathsf{Lax}^+_0\mathsf{Oplax}^0_+$-duoidal category $\mathsf D$ as follows.
\begin{itemize}
\item The \underline{objects} of the category $\mathsf D$ consist of an object $a$ of $\mathsf V$ and four pairwise commuting endomorphisms $\alpha,\beta,\kappa,\nu$ of $a$.

\noindent
The \underline{morphisms} in $\mathsf D$ are those morphisms in $\mathsf V$ which commute with all of the four endomorphisms of the source and target objects.
\item The $\mathsf{Lax}^+\mathsf{Oplax}^0$-monoidal structure is given by the same data as in $\mathsf L$ of Theorem \ref{thm:Lax}; in terms of the first two endomorphisms $\alpha,\beta$ at each object.
\item The $\mathsf{Lax}_0\mathsf{Oplax}_+$-monoidal structure is given by the same data as in $\mathsf R$ of Theorem \ref{thm:OpLax}; in terms of the last two endomorphisms $\kappa,\nu$ at each object.
\item For any non-negative integers $n,p$, the natural transformation $\xi^p_n$ is given by the unique component of the symmetry 
$\sigma_{pn}:\ 
\raisebox{-3pt}{$\stackrel {\displaystyle \otimes} {{}_n}$} \cdot 
(\raisebox{-3pt}{$\stackrel {\displaystyle \otimes} {{}_p}$} \cdots 
\raisebox{-3pt}{$\stackrel {\displaystyle \otimes} {{}_p}$})\cdot \tau_{np} \to 
\raisebox{-3pt}{$\stackrel {\displaystyle \otimes} {{}_p}$} \cdot 
(\raisebox{-3pt}{$\stackrel {\displaystyle \otimes} {{}_n}$} \cdots 
\raisebox{-3pt}{$\stackrel {\displaystyle \otimes} {{}_n}$})$.
\end{itemize}
\end{theorem}

\begin{proof}
An easy substitution, using naturality of the symmetry $\sigma$, shows that the axioms of Definition \ref{def:+0duoidal} hold; we leave it to the reader.
\end{proof}

As an immediate consequence of Theorem \ref{thm:comon_lax+fun} and Theorem \ref{thm:mon_lax+fun} we obtain the following.

\begin{theorem}
For any symmetric monoidal category $(\mathsf V,\otimes,I,\sigma)$ the following categories are pairwise isomorphic.
\begin{itemize}
\item[{(1.i)}] The category of bisemigroups in the $\mathsf{Lax}^+\mathsf{Oplax}_+$-duoidal category $\mathsf D$ of Theorem \ref{thm:DuoLax}.
\item[{(1.ii)}] The category of \BiHom-bimonoids in $\mathsf V$.
\end{itemize}
and 
\begin{itemize}
\item[{(2.i)}] The category of bimonoids in the $\mathsf{Lax}^+_0\mathsf{Oplax}^0_+$-duoidal category $\mathsf D$ of Theorem \ref{thm:DuoLax}.
\item[{(2.ii)}] The category of unital and counital \BiHom-bimonoids in $\mathsf V$.
\end{itemize}
\end{theorem}

Next we explain the Yau twist of (unital and counital) \BiHom-bimonoids in \cite[Proposition 5.9]{GraMakMenPan}. Analogously to the monoidal category $\mathsf M$ of Proposition \ref{prop:Yau}, we associate the following symmetric monoidal category $\mathsf S$ to any symmetric monoidal category $(\mathsf V,\otimes,I,\sigma)$. As a category, $\mathsf S$ is the same as $\mathsf D$ of Theorem \ref{thm:DuoLax}. Its monoidal product is $\raisebox{-3pt}{$\stackrel {\displaystyle \otimes} {{}_2}$}=\otimes$ of $\mathsf D$ and the monoidal unit is $\raisebox{-3pt}{$\stackrel {\displaystyle \otimes} {{}_0}$}=(I,1,1,1,1)$. Mac Lane's coherence natural isomorphisms are given by the same (omitted) morphisms as in $\mathsf V$ and the symmetry is given by $\sigma$.
A bisemigroup (respectively, bimonoid) in $\mathsf S$
is a quintuple consisting of a bisemigroup (respectively, bimonoid) in $(\mathsf V,\otimes,I,\sigma)$ together with four commuting bisemigroup (respectively, bimonoid) endomorphisms. Bisemigroup (respectively, bimonoid) morphisms are those bisemigroup (respectively, bimonoid) morphisms in $(\mathsf V,\otimes,I,\sigma)$ which are compatible with the endomorphism parts too.

\begin{proposition} \label{prop:YauBi}
Consider an arbitrary symmetric monoidal category $(\mathsf V,\otimes,I,\sigma)$, the corresponding $\mathsf{Lax}^+_0\mathsf{Oplax}^0_+$-duoidal category $\mathsf D$ of Theorem \ref{thm:DuoLax}, and the symmetric monoidal category $\mathsf S$ 
in the previous paragraph regarded as a $\mathsf{Lax}^+_0\mathsf{Oplax}^0_+$-duoidal category. The identity functor $\mathsf S \to \mathsf D$ is $\mathsf{Lax}^+_0\mathsf{Oplax}^0_+$-duoidal via the $\mathsf{Lax}^+\mathsf{Oplax}^0$-monoidal structure in Proposition \ref{prop:Yau} and the $\mathsf{Lax}_0\mathsf{Oplax}_+$-monoidal structure in Proposition \ref{prop:YauDual}. Consequently it preserves bisemigroups and bimonoids.
\end{proposition}

\begin{proof}
The easy check of the commutativity of \eqref{eq:Lax+Oplax+functor} for the stated data, using the naturality of $\sigma$, is left to the reader.
\end{proof}

\begin{remark}
In the category $\mathsf D$ of Theorem \ref{thm:DuoLax}, the full subcategory $\mathsf D^\times$ for whose objects $(a,\alpha,\beta,\kappa,\nu)$ all morphisms $\alpha,\beta,\kappa$ and $\nu$ are invertible, carries a duoidal structure. Its monoidal structures are as in Remark \ref{rem:invertible}, and the compatibility morphisms are either trivial or given by the symmetry $\sigma$ of $\mathsf V$.

The $\mathsf{Lax}^+_0\mathsf{Oplax}^0_+$-duoidal functor of 
Proposition \ref{prop:YauBi} restricts to a strong duoidal isomorphism between $\mathsf D^\times$ and the full symmetric monoidal subcategory $\mathsf S^\times$ of $\mathsf S$ in Proposition \ref{prop:YauBi} with the same objects.
This strong duoidal isomorphism $\mathsf D^\times \cong  \mathsf S^\times$ gives rise to an isomorphism of the categories of bisemigroups (respectively, bimonoids) in the isomorphic duoidal categories $\mathsf D^\times$ and $\mathsf S^\times$. This yields an isomorphism between those \BiHom-bialgebras (resp. unital and counital \BiHom-bialgebras) whose four endomorphisms are invertible; and 
the bisemigroups (resp. bimonoids) in $\mathsf V$ together with four commuting bisemigroup (resp. bimonoid) automorphisms.

This amounts to saying that any bimonoid in $\mathsf D^\times$ is a Yau twist of a bimonoid in $\mathsf V$. Moreover, for any bimonoid $((a,\alpha,\beta,\kappa,\nu),(\delta,\varepsilon), (\mu,\eta))$ in $\mathsf D^\times$; and the corresponding bimonoid 
$(a,
(\widetilde\mu=\mu \cdot (\kappa^{-1} \otimes \nu^{-1}),\eta),
(\widetilde \delta= (\alpha^{-1} \otimes \beta^{-1})\cdot \delta,\varepsilon)
)$, 
the following assertions are equivalent.
\begin{itemize}
\item[{(i)}] $(a,(\widetilde \delta,\varepsilon),(\widetilde\mu,\eta))$ is a Hopf monoid in $\mathsf V$. That is, there is a (unique) morphism $\chi:a\to a$ --- the `antipode' --- rendering commutative the following diagram.
$$
\xymatrix@C=35pt @R=15pt{
a \ar[r]^-{\widetilde \delta} \ar[dd]_-{\widetilde \delta} 
\ar[rd]_-{\varepsilon}&
a \otimes a \ar[r]^-{1\otimes \chi} &
a \otimes a \ar[dd]^-{\widetilde \mu} \\
& I \ar[rd]^-{\eta} \\
a\otimes a \ar[r]_-{\chi \otimes 1} &
a\otimes a \ar[r]_-{\widetilde \mu} &
a}
$$
\item[{(ii)}] The following `canonical morphism' (see \cite[(1.10)]{BohmChenZhang}) is invertible, for all objects $y$ and $a$-modules $(x,\varrho)$ in $\mathsf D^\times$.
$$
\xymatrix@C=35pt{
x\otimes y \otimes a \ar[r]^-{1\otimes 1 \otimes \delta} &
x\otimes y \otimes a \otimes a\ar[r]^-{1\otimes \sigma \otimes 1}&
x\otimes a \otimes y \otimes a \ar[r]^-{\varrho \otimes 1\otimes 1} &
x \otimes y \otimes a}
$$
\item[{(iii)}] $((a,\alpha,\beta,\kappa,\nu), (\mu,\eta),(\delta,\varepsilon))$ satisfies the axioms of \cite[Definition 6.9]{GraMakMenPan}. That is, there is a (unique) morphism $\chi:a\to a$ --- the `\BiHom-antipode' --- rendering commutative the following diagram.
$$
\xymatrix@C=15pt@R=15pt{
a \ar[rr]^-\delta \ar[dd]_-\delta \ar[rrrd]_-\varepsilon &&
a \otimes a \ar[rr]^-{1\otimes \chi} &&
a \otimes a \ar[rr]^-{\beta \nu \otimes \alpha \kappa} &&
a \otimes a \ar[dd]^-\mu \\
&&& I \ar[rrrd]^-\eta \\
a\otimes a \ar[rr]_-{\chi \otimes 1} &&
a\otimes a \ar[rr]_-{\beta \nu \otimes \alpha \kappa} &&
a\otimes a \ar[rr]_-\mu &&
a}
$$
\end{itemize}
The morphism $\chi$ occurring in parts (i) and (iii) is the same.
Although the assertions in parts (ii) and (iii) (but not in (i)) are meaningful for any bimonoid in the $\mathsf{Lax}^+_0\mathsf{Oplax}^0_+$-duoidal category $\mathsf D$ of Theorem \ref{thm:DuoLax}, there seems to be no reason to expect their equivalence any longer. 
\end{remark}


\bibliographystyle{plain}

\begin{thebibliography}{10}

\bibitem{AguiarMahajan}
Marcelo Aguiar and Swapneel Mahajan, 
{\em Monoidal functors, species and Hopf algebras,} 
volume 29 of CRM Monograph Series. American Mathematical Society, Providence, RI, 2010. 

\bibitem{BatCaeVer}
Eliezer Batista, Stefaan Caenepeel and Joost Vercruysse, 
{\em Hopf Categories,} 
Algebr. Represent. Theory 19 no. 5 (2016) 1173--1216.

\bibitem{BaezDolan}
John C. Baez and James Dolan,
{\em Higher-Dimensional Algebra III. n-Categories and the Algebra of Opetopes,}
Adv. in Math. 135 no. 2 (1998) 145--206.

\bibitem{Bohm}
Gabriella B\"ohm,
{\em Hopf polyads, Hopf categories and Hopf group monoids viewed as Hopf monads,} 
Theor. and Appl. of Categories 32 no. 37 (2017) 1229--1257.

\bibitem{BohmChenZhang}
Gabriella B\"ohm, Yuanyuan Chen and Liangyun Zhang,
{\em On Hopf monoids in duoidal categories,}
J. Algebra 394 (2013) 139--172.

\bibitem{BohmLack}
Gabriella B\"ohm and Stephen Lack,
{\em Hopf comonads on naturally Frobenius map-monoidales,}  
J. Pure Appl. Algebra 220 no. 6 (2016) 2177--2213.

\bibitem{Bruguieres}
Alain Brugui\`eres,
{\em Hopf Polyads,}
Alg. Represent. Theory 20 no. 5 (2017) 1151--1188.

\bibitem{CaenepeelGoyvaerts}
Stefaan Caenepeel and Isar Goyvaerts, 
{\em Monoidal Hom-Hopf algebras,} 
Comm. Algebra 39 (2011) 2216--2240.

\bibitem{DayStreet:lax}
Brian Day and Ross Street,
{\em Lax monoids, pseudo-operads, and convolution,}
in: ``Diagrammatic Morphisms and Applications"
David E. Radford, Fernando J. O. Souza and David N. Yetter (eds.)
pp. 75--96
Contemp. Math. 318, 2003

\bibitem{DayStreet}
Brian Day and Ross Street,
{\em Quantum categories, star autonomy, and quantum groupoids,} Fields Institute Comm. 43 (2004), 193-231.

\bibitem{GraMakMenPan}
Giacomo Grazianu, Abdenacer Makhlouf, Claudia Menini and Florin Panaite,
{\em \BiHom-Associative Algebras, \BiHom-Lie Algebras and \BiHom-Bialgebras,}
SIGMA 11 (2015), 086, 34 pages.

\bibitem{Leinster}
Tom Leinster,
{\em Higher Operads, Higher Categories,}
Cambridge University Press 2004.

\bibitem{Turaev}
Vladimir Turaev,
{\em Homotopy field theory in dimension 3 and crossed group-categories,}
preprint available at
\href{https://arxiv.org/abs/math/0005291}{https://arxiv.org/abs/math/0005291}.

\bibitem{ZhangWang:BiHom}
Xiaohui Zhang and Dingguo Wang,
{\em Cotwists of Bicomonads and BiHom-bialgebras,}
Alg. Represent. Theory (2019) — available online at
\href{https://link.springer.com/article/10.1007/s10468-019-09888-2}{https://link.springer.com/article/10.1007/s10468-019-09888-2}.
\end{thebibliography}

\end{document}